\documentclass{article}

\usepackage{amsmath,amssymb,amsthm}
\usepackage[utf8]{inputenc}
\usepackage[T1]{fontenc}
\usepackage[a4paper,textheight=630pt]{geometry} 

\usepackage{authblk}
\title{Stable approximation of Helmholtz solutions in the disk\\
by evanescent plane waves}
\author[1,2]{Emile Parolin}
\author[3]{Daan Huybrechs}
\author[4]{Andrea Moiola}
\affil[1]{Laboratoire Jacques-Louis Lions, Sorbonne Université, Paris, France}
\affil[2]{Alpines, Inria, Paris, France, \texttt{emile.parolin@inria.fr}}
\affil[3]{KU Leuven, Leuven, Belgium, \texttt{daan.huybrechs@kuleuven.be}}
\affil[4]{Universit\`a di Pavia, Pavia, Italy, \texttt{andrea.moiola@unipv.it}}

\usepackage[
    sorting=nyt,
    minbibnames=5,
    maxbibnames=5,
    giveninits=true,
    url=false,
    isbn=false,
    eprint=false,
    doi=false,
]{biblatex}
\bibliography{biblio}
\usepackage{hyperref}
\hypersetup{
    bookmarksopen=true,
    colorlinks=true,
    linkcolor=blue,
    citecolor=blue,
    allcolors = {blue},
}
\usepackage{mathtools}  
\mathtoolsset{showonlyrefs} 
\usepackage{enumitem}
\usepackage{graphicx,caption,subcaption}
\theoremstyle{plain}
\newtheorem{theorem}{Theorem}[section]
\newtheorem{proposition}[theorem]{Proposition}
\newtheorem{lemma}[theorem]{Lemma}
\newtheorem{corollary}[theorem]{Corollary}
\newtheorem{definition}[theorem]{Definition}

\newtheorem{remark}[theorem]{Remark}

\newtheorem{conjecture}[theorem]{Conjecture}
\numberwithin{equation}{section}
\begin{document}

\maketitle

\begin{abstract}
    Superpositions of plane waves are known to approximate well the solutions
    of the Helmholtz equation.
    Their use in discretizations is typical of Trefftz methods for Helmholtz
    problems, aiming to achieve high accuracy with a small number of degrees of
    freedom. 
    However, Trefftz methods lead to ill-conditioned linear systems, and it is
    often impossible to obtain the desired accuracy in floating-point
    arithmetic.
    In this paper we show that a judicious choice of plane waves can ensure
    high-accuracy solutions in a numerically stable way, in spite of having to
    solve such ill-conditioned systems.

    Numerical accuracy of plane wave methods is linked not only to the
    approximation space, but also to the size of the coefficients in the plane
    wave expansion. We show that the use of plane waves can lead to
    exponentially large coefficients, regardless of the orientations and the
    number of plane waves, and this causes numerical instability.
    We prove that all Helmholtz fields are continuous superposition of
    evanescent plane waves, i.e., plane waves with complex propagation vectors
    associated with exponential decay, and show that this leads to bounded
    representations.
    We provide a constructive scheme to select a set of real and complex-valued
    propagation vectors numerically.
    This results in an explicit selection of plane waves and an associated
    Trefftz method that achieves accuracy and stability.

    The theoretical analysis is provided for a two-dimensional domain with
    circular shape.
    However, the principles are general and we conclude the paper with a
    numerical experiment demonstrating practical applicability also for
    polygonal domains.
\end{abstract}



\section{Introduction}\label{sec:intro}

The space dependence of time-harmonic solutions
$U(\mathbf x,t)=\Re\{e^{-\imath\omega t}u(\mathbf x)\}$
of the scalar wave equation
$\frac1{c^2}\frac{\partial^2 U}{\partial t^2}-\Delta U=0$
is characterized by the homogeneous Helmholtz equation 
\begin{equation}\label{eq:Helmholtz}
    -\Delta u-\kappa^2 u=0.
\end{equation}
The wavenumber is $\kappa:=\omega/c>0$, with $c$ the wave speed and $\omega$
the time frequency.
Solutions of boundary value problems for the Helmholtz equation are 
oscillatory, making their numerical approximation notoriously
computationally expensive at high frequencies, namely when the wavelength
$\lambda:=2\pi/\kappa$ is much smaller than the characteristic length of the
domain.

A well-studied way to efficiently represent Helmholtz solutions in a domain of
$\mathbb R^n$ is to approximate them with linear combinations of propagative
plane waves
$\mathbf x\mapsto e^{\imath\kappa\mathbf d\cdot\mathbf x}$,
which are particular solutions of~\eqref{eq:Helmholtz} if
the propagation direction $\mathbf d\in\mathbb R^n$
satisfies $\mathbf d\cdot\mathbf d=1$.
Plane waves indeed offer better accuracy for fewer degrees of
freedom compared to polynomial spaces, as supported by
the theory developed in\ \cite{Moiola2011}, building
on previous results in\ \cite[Sec.~8.4]{Melenk1995}
and\ \cite[Sec.~3.3.5]{CessenatDespres1998}.
Approximation by plane waves has been extensively used in the context of
Trefftz schemes for the Helmholtz equation, a class of methods that use trial
and test functions satisfying~\eqref{eq:Helmholtz} locally on each element of
a mesh, see\ \cite{Hiptmair2016} for a comprehensive survey.
The simple expression of plane waves allows for very cheap
implementations; for instance, integrals of products of these
functions can be computed in closed form with
wavenumber-independent effort, see \cite[Sec.~4.1]{Hiptmair2016}.
A second widespread use of plane wave approximation is the reconstruction of
sound fields from point measurements (representing microphones) in experimental
acoustics, see\ \cite{Chardon2014,Jin2015,Verburg2018,Hahmann2021}.

The computation of plane wave approximations is however known to be numerically
unstable, imposing strong limits to the achievable
accuracy\ \cite{PerreyDebain2006,Barucq2021}.
This issue is often understood as an effect of the ill-conditioning of the
linear system that is solved\ \cite[Sec.~4.3]{Hiptmair2016}, which inevitably
arises from the almost-linear dependence of plane waves with similar
propagation directions.
Different techniques have been proposed to overcome this instability,
e.g.\ \cite{Antunes2018,Congreve2018,Barucq2021}.
A well-known recommendation suggests using not more than a prescribed number of
waves in elements of a given size, e.g.\ \cite[Eq.~(14)]{Huttunen2009}:
this keeps the instability at bay but limits the achievable accuracy.

The first purpose of this paper is to shed a new light on the numerical
instability experienced with propagative plane waves and explain the
fundamental mathematical reasons for their limitations as described above.
The second objective is to propose a practical remedy, in the form of including
evanescent plane waves, which may decay exponentially in one direction, and using which one can achieve arbitrary accuracy in a numerically stable way.
The approach is substantiated by theoretical analysis in combination with numerical evidence.
As a first step in this direction, we focus mainly on the model approximation
problem of Helmholtz solutions in the unit disk, using the modal analysis tools
described in Section~\ref{sec:Helmholtz_circular-waves}.

\paragraph{A new point of view on plane wave instability.}

Recent advances in approximation theory, in particular based on the theory of frames and overcomplete bases \cite{Christensen2016}, have shown
that in the presence of ill-conditioning it is not sufficient to study
best approximation errors in order to obtain accurate results in floating-point
arithmetic\ \cite{Adcock2019,Adcock2020}.
Rather, one is led to study the approximation error in relation to the 
coefficient norm, i.e., the norm of the coefficients in the expansion.
The former depends solely on the approximation space, but the coefficient norm
also depends on its chosen representation (i.e.\ on the spanning set used).
We formalize this in Section~\ref{sec:stability-notion} with the notion of
\emph{stable approximation} in Definition~\ref{def:stability}.
The corresponding error analysis in Section~\ref{sec:error-estimate}, based
largely on results in\ \cite{Adcock2019,Adcock2020}, allows us to conclude in
Section~\ref{sec:propagative-plane-waves} that the set of propagative plane
waves does not yield stable approximations.
That is, we can formally state that there exist Helmholtz solutions, with
relatively high Fourier frequency components in the angular coordinate, 
that are well approximated in the
approximation space, but are nevertheless not numerically computable, see
Theorem~\ref{th:instability}.
In the terminology of approximation theory,
no countable set of propagative plane waves is a frame for the space of
Helmholtz solutions.
(We recall that a frame of a Hilbert space is a natural generalization of a basis that allows for redundancy, see \cite{Christensen2016,Adcock2019}.)
In particular, it lacks a so-called lower frame bound which is the property
that ensures that bounded functions can be represented with bounded
coefficients.
This point of view is reminiscent of a similar work in the context of the
Method of Fundamental Solutions\ \cite{Barnett2008}, which pre-dates the
stability analysis from frame theory.

Unfortunately, while the theory in\ \cite{Adcock2019,Adcock2020}
allows to identify this problem, it offers no concrete suggestions as to how it
can be remedied.
If the approximation space remains unchanged, a lower frame bound can only be
established through a change of basis, such as orthogonalization
as in\ \cite{Antunes2018,Congreve2018,Brubeck2021}.
However, that changes the representation: the solution would no longer be
represented in the simple form of an expansion in plane waves, which is a
key feature we would like to retain. Moreover, it may not be straightforward to ensure that the orthogonalization process itself is numerically stable.

\paragraph{The evanescent plane wave remedy.}

To obtain stable representations,
we propose to enrich the approximation space with evanescent
plane waves, i.e.\ plane waves whose direction vectors are complex-valued, 
\( \mathbf{d} \in\mathbb{C}^{n} \), 
as defined in Section~\ref{sec:evanescent-plane-waves}.
The Helmholtz equation is still satisfied provided
\(\mathbf{d} \cdot \mathbf{d} = 1\)
and, importantly, the expression remains simple and cheap to use in
numerical schemes.
Since their modulus decays exponentially in the direction $\Im[\mathbf d]$,
evanescent plane waves are localized in bounded physical domains but also in
the Fourier domain, hence are natural candidates for the approximation of the
high frequency Fourier content exhibited by certain Helmholtz solutions.
This idea is already present in the Wave Based Method, a special class of
Trefftz schemes, see e.g.\ \cite{Deckers2014} for a survey.
Evanescent plane waves also proved particularly effective in the approximation
of interface problems in Trefftz methods, e.g.\ \cite{Luostari2013b}, and the
approximation of integral kernels in some versions of the Fast Multipole
Method\ \cite{Chaillat2015}.

To support the use of evanescent plane wave, we prove in
Section~\ref{sec:continuous-analysis} our main positive result,
Theorem~\ref{th:T_op}, which states that any Helmholtz solution in the unit
disk can be uniquely represented in the form of a continuous superposition of
evanescent plane waves.
This integral representation has the key property of being stable, i.e.~it
features a provably bounded density (in a weighted \(L^{2}\) space), and it can
be seen as a generalization of the classical Herglotz representation, see e.g.\
\cite{Colton1985,Weck2004}.
This result implies that evanescent plane waves form a \emph{continuous} frame
for the space of Helmholtz solutions, see Theorem~\ref{th:continuous_frame}.
While this is stated at the continuous level, such a property paves the way for
successful stable discrete expansions. Indeed, from the stability of the representation one may expect that
discretizations exist with bounded coefficient norms, thereby solving the main issue with propagative plane waves.

\paragraph{A practical numerical recipe.}

In view of practical implementations, we investigate the non-trivial task of
identifying suitable sets of evanescent plane waves which deliver controllable
accuracy in combination with stability.
A heuristic choice for a set of complex directions \(\mathbf{d}\) is suggested
in\ \cite[Sec.~3.2]{Deckers2014} (see also\ \cite[Sec.~3.2]{Hiptmair2016}), but
no mathematical justification is provided.

A first idea to obtain stable \emph{discrete} representations (i.e.~with
bounded coefficients) would be to discretize the continous frame, but
unfortunately, our setting does not fall within the assumptions of existing
results (e.g.~the boundedness assumption of\ \cite[Th.~1.3]{Freeman2019} is not
satisfied).
Instead, the construction of approximation sets described
in Section~\ref{sec:discrete-recipe} is largely based on the optimal sampling
procedure for weighted least-squares recently described by Cohen and
Migliorati\ \cite{Cohen2017} (see also\ \cite{Hampton2015}) and subsequently
used in\ \cite{Migliorati2022}, and it is illustrated with numerical
experiments in Section~\ref{sec:numerics}.
The strategy employed can be interpreted as the construction of a quadrature
rule in the two-dimensional unbounded parametric domain of the integral
representation. 
In practice, the recipe consists in drawing the quadrature points (i.e.~select
the directions of the plane waves) according to an explicit probability density
function~\eqref{eq:density-function} which is a generalization to the
multivariate setting of the Christoffel function density, The latter is sometimes
called spectral function.
While the rigorous numerical analysis of the above approach is thus far incomplete,
we conjecture that such a construction provides stable discrete representations,
see Conjecture~\ref{conj:approx-conjecture}.
In fact, the experimental results in Section~\ref{sec:numerics} show that
the resulting approximations are both controllably accurate and numerically
stable, provided one uses sufficient oversampling and regularization.

Although the recipe is derived from the analysis on the disk, we include
numerical results on a triangular cell showing that it appears to be
effective also for other shapes.
Approximation and stability properties of evanescent plane waves in more general domains
and their use in mesh-based Trefftz methods (e.g.~the Trefftz-Discontinuous
Galerkin method\ \cite[Sec.~2.2]{Hiptmair2016}) will be considered in
future publications (see \cite{Galante2023} for the extension of the theory of
this paper to three-dimensional problems).

\section{Helmholtz equation in circular geometry}\label{sec:Helmholtz_circular-waves}

We first present the setting of the paper and introduce notation.
The proofs of the statements follow standard arguments and are collected in
Appendix~\ref{app:sec2-proofs}.

\subsection{Circular waves}

In this paper we only consider circular two-dimensional geometries.
Without loss of
generality, we assume that the domain is the open unit disk, henceforth
denoted \(B_1:=\{\mathbf x\in\mathbb R^2| \;\|\mathbf x\|<1\}\).
The circular geometry enables modal analysis via separation of variables.
The \emph{circular waves} are the bounded solutions of the Helmholtz equation
in the unit disk that are separable in polar coordinates.
They are sometimes also referred to as \emph{Fourier--Bessel functions} or as
\emph{Generalized Harmonic Polynomials}\ \cite{Melenk1995}.

The results of this paper are fomulated most concisely using the following
\(\kappa\)-dependent scalar product and norm:
for any \(u, v\in H^{1}(B_{1})\),
\begin{equation}\label{eq:Bnorm}
    \left(u,\, v\right)_{\mathcal{B}} :=
    \left(u,\, v\right)_{L^{2}(B_{1})}
    +\kappa^{-2}
    \left(\nabla u,\, \nabla v\right)_{L^{2}(B_{1})},
    \qquad\qquad
    \|u\|_{\mathcal{B}}^2 :=
    \left(u,\, u\right)_{\mathcal{B}}.
\end{equation}
\begin{definition}[Circular waves]
    We define, for any \(p\in\mathbb{Z}\)
    \begin{equation}\label{eq:btilde_p}
        \begin{cases}
            \tilde{b}_{p} (\mathbf{x})
            := J_{p}(\kappa r)e^{\imath p\theta},
            \quad\forall \mathbf{x}=\left(r,\,\theta\right)\in B_{1},\\
            b_{p} := \beta_{p}\tilde{b}_{p},
            \quad\text{where}\quad
            \beta_{p} := \|\tilde{b}_{p}\|_{\mathcal{B}}^{-1},
        \end{cases}
        \quad \text{and} \qquad
        \mathcal{B} := \overline{
            \operatorname{span}
            \left\{{b}_{p}\right\}_{p\in\mathbb{Z}}
        }^{\|\cdot\|_{\mathcal{B}}}
        \subsetneq H^{1}(B_{1}).
    \end{equation}
\end{definition}
In this definition, \(J_{p}\) is the usual Bessel function of the first
kind\ \cite[Eq.~(10.2.2)]{DLMF} and $\imath$ the imaginary unit $\imath^2=-1$.
The space \(\mathcal{B}\) is a strict subspace of \(H^{1}(B_{1})\), whose
elements are solutions of the Helmholtz equation, see
Lemma~\ref{lem:B_space_Helmholtz_solution} below.
A representation of the real part of some circular waves is given in
Figure~\ref{fig:bptilde_k16_real}.
We will refer to the circular waves with mode number \(|p| < \kappa\) as
\emph{propagative} modes.
The `energy' of such modes is distributed in the bulk of the domain.
On the contrary, for \(|p| \gg \kappa\), the circular waves are termed
\emph{evanescent}.
Their `energy' is concentrated near the boundary of the domain.
In between, the waves such that \(|p| \approx \kappa\) are called
\emph{grazing} modes.
\begin{figure}
    \centering
    \begin{subfigure}{0.3\textwidth}
        \centering
        \includegraphics[width=\textwidth]{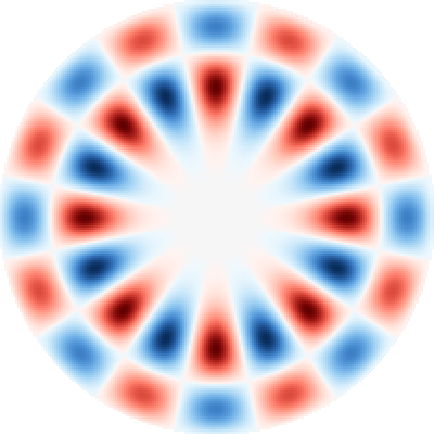}
        \caption{Propagative \(p=8\).}\label{fig:bptilde_k16_real_p8}
    \end{subfigure}
    \begin{subfigure}{0.3\textwidth}
        \centering
        \includegraphics[width=\textwidth]{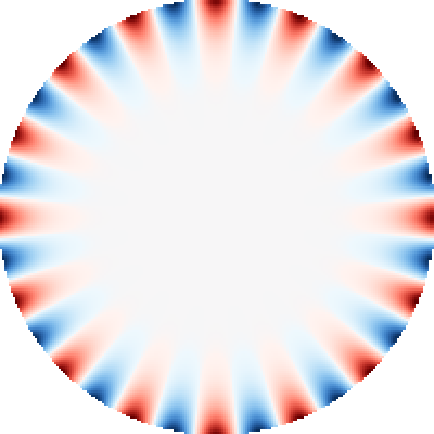}
        \caption{Grazing \(p=16\).}\label{fig:bptilde_k16_real_p16}
    \end{subfigure}
    \begin{subfigure}{0.3\textwidth}
        \centering
        \includegraphics[width=\textwidth]{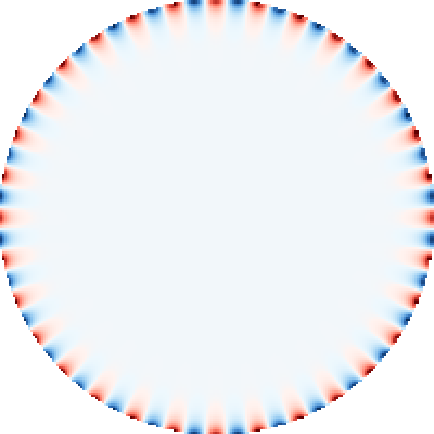}
        \caption{Evanescent \(p=32\).}\label{fig:bptilde_k16_real_p32}
    \end{subfigure}
    \caption{
	Real part of the circular waves \(\tilde{b}_{p}\) for three
	different modes (wavenumber \(\kappa=16\)).
	}\label{fig:bptilde_k16_real}
\end{figure}
\begin{lemma}\label{lem:b_Hilbert_basis}
    The space
    \(\left(\mathcal{B},\, \|\cdot\|_{\mathcal{B}}\right)\)
    is a Hilbert space
    and the family \(\{b_{p}\}_{p\in\mathbb{Z}}\) is a Hilbert basis
    (i.e.~an orthonormal basis):
    \begin{equation}
        \left(b_{p},\, b_{q}\right)_{\mathcal{B}}
        =\delta_{pq},
        \qquad\forall p,q\in \mathbb{Z},
        \qquad\text{and}\qquad
        u = \sum_{p\in\mathbb{Z}}
        \left(u,\, b_{p}\right)_{\mathcal{B}} b_{p},
        \qquad\forall u\in \mathcal{B}.
    \end{equation}
\end{lemma}

The main reason for introducing circular waves is the possibility to use them
to expand any Helmholtz solution on the disk, as we show in the next lemma.
Related results for more general domains and different norms are available, see
e.g.\ \cite[Sec. 3.1]{Hiptmair2016}.

\begin{lemma}\label{lem:B_space_Helmholtz_solution}
    \(u\in H^{1}(B_{1})\) satisfies the Helmholtz equation~\eqref{eq:Helmholtz}
    if and only if \(u\in\mathcal{B}\).
\end{lemma}

Circular and spherical waves have been used as basis functions in many
Trefftz schemes, see\ \cite[Sec. 3.1]{Hiptmair2016} and the references therein.
An interesting feature of such waves is that the approximation sets are
naturally hierarchical.

\subsection{Asymptotics of normalization coefficients}

The normalization coefficients \(\beta_{p}\) in~\eqref{eq:btilde_p} grow
super-exponentially with \(|p|\) after a pre-asymptotic regime up to
\(|p| \approx \kappa\).
The precise asymptotic behavior is given by the following lemma.

\begin{lemma}\label{lem:beta_asymptotic}
    For all \(p \in \mathbb{Z}\),
    \begin{equation}\label{eq:beta_explicit}
        \beta_{p} = \left(
        2\pi \left[J_p^2(\kappa)
        - J_{p-1}(\kappa)J_{p+1}(\kappa)
        + J_{p}'(\kappa)J_p(\kappa) / \kappa\right]
        \right)^{-1/2}
        \underset{|p|\to +\infty}{\sim} \kappa
        \left(\frac{2}{e\kappa}\right)^{|p|}\; |p|^{|p|}.
    \end{equation}
\end{lemma}

\begin{remark}\label{rmk:normalization-btildep}
    The circular waves are normalized using the rather natural \(\mathcal{B}\)
    norm~\eqref{eq:Bnorm}, i.e.\ the wavenumber-weighted $H^1(B_1)$ norm.
    The use of $L^2(B_1)$ or other Sobolev norms in the definition of $\beta_p$
    would not modify the exponential dependence on $|p|$ of the asymptotics
    \eqref{eq:beta_explicit}, but it does introduce an additional moderate power of
    $|p|$, as is visible in the proof in Appendix~\ref{app:sec2-proofs}.
\end{remark}

\section{Stable numerical approximation}\label{sec:stability-notion}

The purpose of this section is to explain the crucial notion of stable
approximation, which we could informally call ``approximation with small
coefficients'', and to clarify how it enables accurate numerical computations.
Our approach builds on the results in\ \cite{Adcock2019,Adcock2020} which
highlight the importance for stability of having representations with bounded
coefficients.
We also describe the practical procedure, a regularized sampling method,
that we use to investigate the existence of stable numerical approximations of
Helmholtz solutions in this paper.
An error bound is formulated in Proposition~\ref{prop:approx-error-estimate}.

\subsection{The notion of stable approximation}

Let us consider a sequence of finite approximation sets in \(\mathcal{B}\)
\begin{equation}\label{eq:generic-approx-sets}
    \boldsymbol{\Phi} := \{\boldsymbol{\Phi}_{k}\}_{k \in \mathbb{N}}
    \qquad\text{where}\qquad
    \boldsymbol{\Phi}_{k} := \{\phi_{k,l}\}_{l},
    \quad|\boldsymbol{\Phi}_k|<\infty,
    \quad \forall k \in \mathbb{N},
\end{equation}
and for each \(k,l\), \(\phi_{k,l} \in \mathcal{B}\) is a solution of the
Helmholtz equation~\eqref{eq:Helmholtz} in the unit disk.
These sets need not be nested.
When the $\{\phi_{k,l}\}_l$ are linearly independent,
$\boldsymbol\Phi_k$ is a basis of the approximation space used for numerical
computations. However, more generally, we also allow for linearly dependent
sets.
Associated to any set \(\boldsymbol{\Phi}_{k}\) for some \(k \in \mathbb{N}\),
we define the synthesis operator
\begin{equation}
    \mathcal{T}_{\boldsymbol{\Phi}_{k}} \;:\;
    \mathbb{C}^{|\boldsymbol{\Phi}_{k}|} \to \mathcal{B},\ 
    \boldsymbol{\mu} = \{\mu_{l}\}_{l}
    \mapsto
    \sum_{l} \mu_{l}\phi_{k,l}.
\end{equation}
Here and in the following, we use the notation \(|X|\)
to indicate the cardinality of the set \(X\).
We are now ready to define a notion of \emph{stable approximation}, which is at
the heart of this paper.

\begin{definition}[Stable approximation]\label{def:stability}
    The sequence \(\boldsymbol{\Phi}\) of approximation
    sets~\eqref{eq:generic-approx-sets}
    is said to be a \emph{stable approximation} for \(\mathcal{B}\) if,
    for any tolerance \(\eta > 0\),
    there exist a stability exponent \(s \geq 0\)
    and a stability constant \(C_{\mathrm{stb}} \geq 0\)
    such that
    \begin{equation}\label{eq:stability}
        \forall u \in \mathcal{B},\ 
        \exists\, \boldsymbol{\Phi}_{k} \in \boldsymbol{\Phi},\ 
        \boldsymbol{\mu} \in \mathbb{C}^{|\boldsymbol{\Phi}_{k}|}
        \quad\text{such that}\quad
        \begin{cases}
            \|u - \mathcal{T}_{\boldsymbol{\Phi}_{k}}{\boldsymbol{\mu}}\|_{\mathcal{B}}
            \leq \eta \|u\|_{\mathcal B}
            \quad\text{and}\\[1mm]
            \|\boldsymbol{\mu}\|_{\ell^{2}} \leq C_{\mathrm{stb}}
            |\boldsymbol{\Phi}_{k}|^{s} \|u\|_{\mathcal B}.
        \end{cases}
    \end{equation}
\end{definition}

Having a sequence of stable approximation sets means that one can approximate
any Helmholtz solution to a given accuracy in the form of a finite expansion 
\(\mathcal{T}_{\boldsymbol{\Phi}_{k}}{\boldsymbol{\mu}}\)
where the coefficients \(\boldsymbol{\mu}\) have bounded \(\ell^{2}\)-norm.
This bound on the coefficients admits a polynomial growth in the number
$|\boldsymbol{\Phi}_{k}|$ of terms in the expansion, but not an exponential
growth.
The stability exponent \(s \geq 0\) of a stable approximation sequence controls
the growth of the coefficient norm \(\|\mu\|_{\ell^{2}}\): the smaller $s$ the
more stable the sequence.
This notion of stability is not related to a space but rather to a particular
sequence of sets that are used to represent the numerical approximation.
In practice the computation of approximations using stable sequences may lead
to ill-conditioned linear systems if there is redundancy in the approximation
sets.
The rest of this section shows that, in spite of possible ill-conditioning,
stable sequences lead to accurate approximations, thanks to the boundedness of
the expansion coefficients.

The simplest stable approximation is provided by the truncation of any
orthonormal basis of $\mathcal B$, in which case $s=0$ and
$C_{\mathrm{stb}}=1$, e.g.\ the circular waves
$\boldsymbol\Phi_k=\{b_p\}_{|p|\le k}$.
However, in view of the application to Trefftz methods on polygonal meshes, we
describe two examples of approximations sets of the
type of~\eqref{eq:generic-approx-sets}:
propagative plane waves (PPWs) in~\eqref{eq:propagative-pw-sets}
and evanescent plane waves (EPWs) in~\eqref{eq:evanescent-pw-sets}.
They exhibit different stability properties.
In Theorem~\ref{th:instability} we prove rigorously that PPWs are necessarily unstable.
In contrast, numerical evidence from Section~\ref{sec:numerics}
indicates that the sets of EPWs constructed following the
numerical recipe that we propose in Section~\ref{sec:conj-stability} are
stable.

\subsection{Boundary sampling method}\label{sec:Sampling}

We explain how we compute the coefficients in practice, which builts on results
in\ \cite{Huybrechs2019}.
All the numerical results obtained in this paper are obtained using the method
described here.

Let us introduce a `trace operator' \(\gamma\), namely a (continuous) linear
operator defined on \(H^{1}(B_{1})\) such that the following problem is well-posed:
find \(u \in H^{1}(B_{1})\) such that
\begin{equation}
    -\Delta u - \kappa^{2} u = 0, \quad\text{in}\ B_{1},
    \qquad\text{and}\qquad
    \gamma u = g, \quad\text{on}\ \partial B_{1},
\end{equation}
for some suitable boundary data \(g\).
Examples of such a trace operator \(\gamma\) are:
the Dirichlet trace operator, extension to \(H^{1}(B_{1})\) of
\(u \mapsto u|_{\partial B_{1}}\),
when \(\kappa^{2}\) is not an eigenvalue of the Dirichlet Laplacian;
the Neumann trace operator, extension to \(H^{1}(B_{1})\) of
\(u \mapsto \partial_{\mathbf{n}}u\),
when \(\kappa^{2}\) is not an eigenvalue of the Neumann Laplacian;
the Robin trace operator, extension to \(H^{1}(B_{1})\) of
\(u \mapsto \partial_{\mathbf{n}}u - \imath \kappa u|_{\partial B_{1}}\)
(without assumptions on the wavenumber \(\kappa\)).

We aim at reconstructing a solution \(u \in \mathcal{B}\)
having access to its trace \(\gamma u\) on the boundary
for such a `good' trace operator \(\gamma\). 
For simplicity, we use the Dirichlet trace operator and therefore assume that
\(\kappa^{2}\) is away from the eigenvalues of the Dirichlet Laplacian.
Further we will assume that \(u \in \mathcal{B} \cap C^{0}(\overline{B_{1}})\),
so that it makes sense to consider point evaluations of the Dirichlet trace.

The reconstruction process is \emph{not} the main subject of the paper
and we stress that we make these two assumptions mainly for convenience and
definiteness (in particular for the numerical experiments).
One can consider alternative reconstruction procedures using other types of
data, such as point evaluation in the bulk of the domain or by taking inner
product of the solution with suitable test functions.
See\ \cite{Chardon2014} for a more general discussion on the subject
of reconstructing Helmholtz solutions from point evaluations.

Let \(u \in \mathcal{B} \cap C^{0}(\overline{B_{1}})\)
be the target of the approximation problem.
We look for a set of coefficients
\(\boldsymbol{\xi} \in \mathbb{C}^{|\boldsymbol{\Phi}_{k}|}\)
for a given approximation set \(\boldsymbol{\Phi}_{k}\)
(introduced in~\eqref{eq:generic-approx-sets})
such that \(\mathcal{T}_{\boldsymbol{\Phi}_{k}}\boldsymbol{\xi} \approx u\).
We also assume that for any \(l\),
\(\phi_{k,l} \in \mathcal{B} \cap C^{0}(\overline{B_{1}})\).
Define the set of \(S \geq |\boldsymbol{\Phi}_{k}|\) sampling points
\(\{\mathbf{x}_{s}\}_{s=1}^{S}\) on the unit circle
parametrized by the angle
\begin{equation}\label{eq:boundary_sampling_nodes}
    \theta_{s} := \frac{2\pi s}{S},
    \qquad 1\leq s \leq S.
\end{equation}
Let us introduce
the matrix \(A = (A_{s,l})_{s,l} \in \mathbb{C}^{S \times |\boldsymbol{\Phi}_{k}|}\)
and the vector \(\mathbf{b} = (\mathbf{b}_{s})_{s} \in \mathbb{C}^{S}\)
such that
\begin{equation}\label{eq:def_matrix_rhs}
    A_{s,l} =
    \gamma(\phi_{k,l})(\mathbf{x}_{s}),
    \quad
    \mathbf{b}_{s} = (\gamma u)(\mathbf{x}_{s}),
    \qquad\qquad 1\leq l\leq |\boldsymbol{\Phi}_{k}|,\ 1\leq s\leq S.
\end{equation}
The sampling method then consists in approximately solving the 
rectangular linear system
\begin{equation}\label{eq:linear-system}
    A\boldsymbol{\xi} = \mathbf{b}.
\end{equation}

\subsection{Regularization}\label{sec:regularization}

It often happens that the matrix \(A\) is ill-conditioned (see
Section~\ref{sec:instability}).
In finite precision arithmetic, severe ill-conditioning may prevent
us from obtaining accurate approximations.
However, the type of ill-conditioning encountered here is benign if it arises
only from the redundancy of the approximating functions.
In that case, ill-conditioning is associated with the numerical
non-uniqueness of the solution of the linear system, yet all associated
expansions may approximate the target to similar accuracy.  
If among those expansions there exist some with small coefficient norms, then
it is possible to numerically compute an accurate approximation.
To this aim, we rely on the combination of oversampling and regularization
techniques developed in\ \cite{Adcock2019,Adcock2020}.
Alternative techniques to curb ill-conditioning can be found in the
literature, see\ \cite{Antunes2018} where a suitable change of basis is used that
works well for circular geometries, \cite{Congreve2018} which uses
orthogonalization, and\ \cite{Barucq2021,Huybrechs2019} in the context of
Trefftz methods.

The first step is to compute the Singular Value Decomposition (SVD) of the
matrix \(A\), namely
\begin{equation}
    A = U\Sigma V^{*}.
\end{equation}
Let us denote by \((\sigma_{m})_{m}\) for \(m=1,\dots,|\boldsymbol{\Phi}_{k}|\)
the singular values of \(A\), assumed to be sorted in descending order.
For notational clarity, the largest singular value is renamed
\(\sigma_{\max} := \sigma_{1}\).
Then, the regularization amounts to trimming the tail of \emph{relatively}
small singular values, which are approximated by zero.
Let \(\epsilon \in (0, 1]\) be a chosen threshold,
we denote by \(\Sigma_{\epsilon}\) the approximation of the diagonal matrix
\(\Sigma\) where all diagonal elements \(\sigma_{m}\) such that
\(\sigma_{m} < \epsilon \sigma_{\max}\) are replaced by zero.
This leads to the approximate factorization
\begin{equation}\label{eq:mat_approx_SVDr}
    A_{S,\epsilon} := U\Sigma_{\epsilon}V^{*},
\end{equation}
of the matrix \(A\).
An approximate solution to~\eqref{eq:linear-system} is then obtained by
\begin{equation}\label{eq:solution_SVDr}
    \boldsymbol{\xi}_{S,\epsilon} :=
    A_{S,\epsilon}^{\dagger} \mathbf{b} =
    V\Sigma_{\epsilon}^{\dagger} U^{*}
    \,\mathbf{b}.
\end{equation}
Here \(\Sigma_{\epsilon}^{\dagger}\) denotes the pseudo-inverse of the matrix
\(\Sigma_{\epsilon}\), namely the diagonal matrix with
\((\Sigma_{\epsilon}^{\dagger})_{j,j}=(\Sigma_{j,j})^{-1}\) if
\(\Sigma_{j,j}\ge\epsilon\sigma_{\max}\) and
\((\Sigma_{\epsilon}^{\dagger})_{j,j}=0\) otherwise.
Robust computation of \(\boldsymbol{\xi}_{S,\epsilon}\)
requires to compute the right-hand-side of~\eqref{eq:solution_SVDr}
from right to left, namely 
\(\boldsymbol{\xi}_{S,\epsilon} :=
V\left(\Sigma_{\epsilon}^{\dagger} \left(U^{*}
\,\mathbf{b}\right)\right)\), in order to avoid mixing small and large
values on the diagonal of $\Sigma_{\epsilon}^{\dagger}$.

\subsection{Error estimates for the sampling method with regularization}\label{sec:error-estimate}

With the regularization technique described above
together with \emph{oversampling}, i.e., \(S\) larger than
\(|\boldsymbol{\Phi}_{k}|\), accurate approximations can be effectively
computed, provided the set sequence is a stable approximation in the sense of
Definition~\ref{def:stability}.
This broad statement is the main message of\ \cite[Th.~5.3]{Adcock2019}
and\ \cite[Th.~1.3 and 3.7]{Adcock2020},
and is the starting point of our quest of stable approximation sets for 
Helmholtz solutions.
More precisely, the following proposition is a rewording
of\ \cite[Th.~3.7]{Adcock2020} from the context of \emph{generalized sampling}
to our setting, with the notations just introduced.
See Appendix~\ref{app:sec3-proofs} for the proof.

\begin{proposition}\label{prop:approx-error-estimate}
    Let \(\gamma\) be the Dirichlet trace operator and
    \(u \in \mathcal{B} \cap C^{0}(\overline{B_{1}})\).
    Given some approximation set \(\boldsymbol{\Phi}_{k}\) (\(k \in \mathbb{N}\) fixed)
    such that for any \(l\),
    \(\phi_{k,l} \in \mathcal{B} \cap C^{0}(\overline{B_{1}})\);
    a sampling set of size \(S \in \mathbb{N}\)
    as described in~\eqref{eq:boundary_sampling_nodes}
    and some regularization parameter \(\epsilon \in (0,1]\),
    we consider the approximate solution of the linear
    system~\eqref{eq:linear-system}, namely
    \(\boldsymbol{\xi}_{S,\epsilon} \in \mathbb{C}^{|\boldsymbol{\Phi}_{k}|}\)
    as defined in~\eqref{eq:solution_SVDr}.
    Then
    \begin{equation}\label{eq:approx-error-estimate1}
        \begin{aligned}
            & \forall \boldsymbol{\mu} \in \mathbb{C}^{|\boldsymbol{\Phi}_{k}|},
            \,\exists S_{0} > 0, \forall S \geq S_{0}, \\
            & \| \gamma(
                u -
                \mathcal{T}_{\boldsymbol{\Phi}_{k}}\boldsymbol{\xi}_{S,\epsilon})
            \|_{L^{2}(\partial B_{1})}
            \leq
            3 \| \gamma(
                u -
                \mathcal{T}_{\boldsymbol{\Phi}_{k}}\boldsymbol{\mu})
            \|_{L^{2}(\partial B_{1})}
            + 2 \sqrt{\pi} \, \frac{\epsilon \,\sigma_{\max}}{\sqrt{S}}
            \|\boldsymbol{\mu}\|_{\ell^{2}}.
        \end{aligned}
    \end{equation}
    Assume moreover that $\kappa^2$ is not an eigenvalue of the Dirichlet
    Laplacian in $B_1$.
    Then,
    there exists a constant
    \(C_{\mathrm{err}}\) independent of \(u\) and \(\boldsymbol{\Phi}_{k}\),
    such that
    \begin{equation}\label{eq:approx-error-estimate2}
        \begin{aligned}
            & \forall \boldsymbol{\mu} \in \mathbb{C}^{|\boldsymbol{\Phi}_{k}|},
            \,\exists S_{0} > 0, \ \forall S \geq S_{0}, \\
            & \|u - \mathcal{T}_{\boldsymbol{\Phi}_{k}}\boldsymbol{\xi}_{S,\epsilon}\|_{L^{2}(B_{1})}
            \leq C_{\mathrm{err}} 
            \Big(
                \| 
                    u -
                    \mathcal{T}_{\boldsymbol{\Phi}_{k}}\boldsymbol{\mu}
                \|_{\mathcal{B}}
                + \frac{\epsilon \,\sigma_{\max}}{\sqrt{S}} \|\boldsymbol{\mu}\|_{\ell^{2}}
            \Big).
        \end{aligned}
    \end{equation}
\end{proposition}

Proposition~\ref{prop:approx-error-estimate} shows that having stable
approximation sets in the sense of Definition~\ref{def:stability} is a
sufficient condition for the accurate reconstruction of a Helmholtz solution
from its samples on the boundary of the disk, provided enough sampling points
\(S\) and a sufficiently small regularization parameter \(\epsilon\) are used.
This is summed up in the following result,
see Appendix~\ref{app:sec3-proofs} for its proof.

\begin{corollary}\label{cor:approx-error-estimate}
    Let \(\delta>0\).
    We assume to have a sequence of approximation sets
    \(\{\boldsymbol{\Phi}_{k}\}_{k \in \mathbb{N}}\) that is stable in the sense of
    Definition~\ref{def:stability}.
    Assume also that $\kappa^2$ is not a Dirichlet eigenvalue in $B_1$.
    Then,
    \begin{equation}\label{eq:cor:approx-ee}
        \begin{aligned}
            & \forall u \in \mathcal{B} \cap C^{0}(\overline{B_{1}}),\ 
            \exists \boldsymbol{\Phi}_{k},\ S_{0} > 0,\ \epsilon_{0} \in (0,1],
            \quad\text{such that}\quad\\
            & \forall S \geq S_{0}, \epsilon \in (0,\epsilon_{0}], \qquad
            \|u - \mathcal{T}_{\boldsymbol{\Phi}_{k}}\boldsymbol{\xi}_{S,\epsilon}\|_{L^{2}(B_{1})}
            \leq \delta \| u \|_{\mathcal{B}},
        \end{aligned}
    \end{equation}
    where
    \(\boldsymbol{\xi}_{S,\epsilon} \in \mathbb{C}^{|\boldsymbol{\Phi}_{k}|}\)
    is defined in~\eqref{eq:solution_SVDr}.
    Moreover, we can take the regularization parameter \(\epsilon\) as large as 
    \begin{equation}\label{eq:epsilon-estimate}
        \epsilon_{0} = \frac{\delta \; \sqrt{S}}
        {2 C_{\mathrm{err}} \sigma_{\max} C_{\mathrm{stb}}|\boldsymbol{\Phi}_{k}|^{s}}.
    \end{equation}
\end{corollary}

The point of Corollary~\ref{cor:approx-error-estimate} is not only that the
solution of the regularized SVD problem provides an accurate approximation of
$u$, but also that it is numerically computable. 
This is in contrast with the classical theory for the
approximation by PPWs, e.g.\ \cite{Moiola2011}, which
provides rigorous best-approximation error bounds that often can not be
attained numerically, precisely because accurate approximations require large
coefficients and cancellation, so exact-arithmetic results cannot be reflected
by floating-point computations.

The assumption on the eigenvalues in Corollary~\ref{cor:approx-error-estimate}
can be lifted if in \eqref{eq:cor:approx-ee} the $L^2(B_1)$ norm is replaced by
$L^2(\partial B_1)$.
Moreover, in this case, the constant \(C_\mathrm{err}\) at the right-hand side
of~\eqref{eq:epsilon-estimate} can be dropped.
Finally, the largest singular value \(\sigma_{\max}\) of the matrix \(A\)
appears in the above results: in our numerical experiments \(\sigma_{\max}\) is
moderate, see Figure~\ref{fig:k16_svs_scatter}.

In the following, it will be convenient to measure the approximation error
by the relative residual
\begin{equation}\label{eq:relative-error}
    \mathcal{E} = \mathcal{E}(u, \boldsymbol{\Phi}_{k}, S, \epsilon) :=
    \frac{ \|A \boldsymbol{\xi}_{S,\epsilon} - \mathbf{b}\|_{\ell^{2}} }
    { \|\mathbf{b}\|_{\ell^{2}} },
\end{equation}
where \(\boldsymbol{\xi}_{S,\epsilon}\) is the 
solution~\eqref{eq:solution_SVDr} of the regularized linear system.
Arguing as in the proof of Proposition~\ref{prop:approx-error-estimate},
(see Appendix~\ref{app:sec3-proofs})
for sufficiently large $S$, the quantity \(\mathcal{E}\) satisfies 
(for a constant \(\tilde{C}\) independent of \(u\),
\(\boldsymbol{\Phi}_{k}\), \(S\))
\begin{equation}
    \|u - \mathcal{T}_{\boldsymbol{\Phi}_{k}}\boldsymbol{\xi}_{S,\epsilon}\|_{L^{2}(B_{1})}
    \leq \tilde{C} \|u\|_{\mathcal{B}} \; \mathcal{E} .
\end{equation}

\section{Instability of propagative plane wave sets}\label{sec:propagative-plane-waves}

The purpose of this section is to present the pitfalls encountered when using
\emph{propagative plane waves} (PPW) to approximate Helmholtz solutions in the unit
disk.
In particular, we show that PPWs with equispaced angles in general
fail to yield stable approximations.
This implies that problems can not be solved numerically to arbitrary accuracy
or, in some cases, to any accuracy at all. The main effort is to show that PPW
approximations lead to large expansion coefficients and that this problem can
not be avoided using PPWs alone.

\subsection{Propagative plane waves and Jacobi--Anger identity}

We introduce the notion of a \emph{propagative} plane wave.
The adjective \emph{propagative} is not customary in the literature, but serves
to distinguish the following definition with the notion of \emph{evanescent}
plane waves (EPWs) that will be introduced in
Definition~\ref{def:evanescent-plane-wave}.

\begin{definition}[Propagative plane wave]\label{def:propagative-plane-wave}
    For any angle \(\varphi\in[0,2\pi)\), we let
    \begin{equation}\label{eq:direction_propagative-plane-wave}
        \mathrm{PW}_{\varphi}(\mathbf{x})
        := e^{\imath \kappa \mathbf{d}(\varphi)\cdot\mathbf{x}},
        \ \forall \mathbf{x}\in\mathbb{R}^2,
        \quad\text{where}\quad
        \mathbf{d}(\varphi) := \left(\cos \varphi,\, \sin \varphi\right)
        \in \mathbb{R}^{2}.
    \end{equation}
\end{definition}
All PPWs satisfy the homogeneous Helmholtz equation~\eqref{eq:Helmholtz} since
\(\mathbf{d}(\varphi) \cdot \mathbf{d}(\varphi) = 1\)
for any angle \(\varphi\in[0,2\pi)\).

Propagative plane waves are a common choice in Trefftz schemes,
see\ \cite[Sec. 3.2]{Hiptmair2016} and the references therein.
In 2D, isotropic approximations are obtained by using equispaced angles:
for some \(M \in \mathbb{N}\), the approximation set is defined as
\begin{equation}\label{eq:propagative-pw-sets}
    \boldsymbol{\Phi}_{M} := \{M^{-1/2}\,\mathrm{PW}_{\varphi_{M,m}}\}_{m=1}^{M},
    \quad\text{where}\quad
    \varphi_{M,m}:=\frac{2\pi m}{M},
    \qquad 1\leq m \leq M.
\end{equation}
In contrast to circular waves, the approximation sets based on such PPWs are in
general not hierarchical.
Plane waves spaces have been studied in the literature, in particular explicit
\(hp\)-estimates in suitable Sobolev semi-norms are available for general
domains, see\ \cite[Th.~5.2 and 5.3]{Moiola2011}.
These results ensure more than exponential convergence (with respect to the
number of plane waves used) of the approximation of homogeneous Helmholtz
solutions by a finite superposition of PPWs.
Therefore, at least in principle, PPWs are well-suited for
Trefftz approximations.

The \emph{Jacobi--Anger identity}\ \cite[Eq.~(10.12.1)]{DLMF} provides a link
between plane waves and circular waves and is ubiquitous in the analysis that
follows:
\begin{equation}\label{eq:Jacobi--Anger}
    \mathrm{PW}_{\varphi}(r,\theta)
    = e^{\imath \kappa \mathbf{d}(\varphi)\cdot\mathbf{x}}
    = \sum_{p\in\mathbb{Z}}
    \imath^{p} J_{p}(\kappa r)
    e^{\imath p(\theta-\varphi)},
    \qquad\forall\mathbf{x} = (r,\theta) \in B_{1},\ \varphi \in [0,2\pi).
\end{equation}

\subsection{Herglotz representation}\label{sec:Herglotz-representation}

We recall the so-called \emph{Herglotz functions}.
They are defined for any \(v \in L^{2}([0,2\pi])\) as
\begin{equation}\label{eq:Herglotz-ppw}
    u(\mathbf{x}; v) :=
    \int_{0}^{2\pi} v(\varphi) \mathrm{PW}_{\varphi}(\mathbf{x})
    \;\mathrm{d}\varphi,
    \qquad\forall\mathbf{x} \in \mathbb{R}^{2},
\end{equation}
see\ \cite[Eq.~(1.1)]{Colton1985},\ \cite[Eq.~(6)]{Weck2004}
and\ \cite[Def.~3.18]{Colton2013}.
Such an expression is termed \emph{Herglotz representation}.
The function $v$ is called \emph{Herglotz kernel} or \emph{density}.
These functions \(u(\cdot;v) \in C^{\infty}(\mathbb{R}^{2})\)
are entire solutions of the Helmholtz equation
and can be seen as a continuous superposition of PPWs,
weighted according to \(v\).
To see that \(u(\cdot, v)\in\mathcal{B}\), let \(v \in L^{2}([0,2\pi]) \),
which we write as a Fourier expansion
\begin{equation}\label{eq:vHerglotz}
    v(\varphi) = \frac{1}{2\pi}
    \sum_{p\in\mathbb{Z}} \hat v_{p} e^{\imath p \varphi},
    \qquad\forall\varphi\in [0,2\pi],
\end{equation}
for a sequence of coefficients
\((\hat v_{p})_{p\in\mathbb{Z}}\in\ell^2(\mathbb Z)\).
Plugging this expression into~\eqref{eq:Herglotz-ppw} and using the
Jacobi--Anger expansion~\eqref{eq:Jacobi--Anger} together with the
orthogonality of the complex exponentials
\(\{\theta\mapsto e^{\imath p \theta}\}_{p\in\mathbb{Z}}\),
we obtain, for any \(\mathbf{x} = (r,\theta) \in \mathbb{R}^{2}\),
\begin{equation}
    u(\mathbf{x}; v) =
    \int_{0}^{2\pi} v(\varphi) \mathrm{PW}_{\varphi}(\mathbf{x})
    \;\mathrm{d}\varphi
    = \sum_{p\in\mathbb{Z}} \imath^{p} \hat v_{p}
    J_{p}(\kappa r) e^{\imath p \theta}
    = \sum_{p\in\mathbb{Z}} \frac{\imath^p\hat v_p}{\beta_p}\,
    b_{p}(\mathbf{x})\in \mathcal{B},
\end{equation}
thanks to the super-exponential growth of the coefficients
\(\{\beta_{p}\}_{p\in\mathbb{Z}}\) shown in
Lemma~\ref{lem:beta_asymptotic}.

While circular waves do have a Herglotz representation, their Herglotz densities
are not bounded uniformly with respect to the index \(p\).
For any \(p\in\mathbb{Z}\) and \(\mathbf{x} = (r,\theta) \in B_{1}\),
using once again
Jacobi--Anger expansion~\eqref{eq:Jacobi--Anger} together with the
orthogonality of the complex exponentials, we have
\begin{equation}
    \int_{0}^{2\pi} e^{\imath p \varphi}
    \mathrm{PW}_{\varphi}(\mathbf{x}) \;\mathrm{d}\varphi
    = 
    \int_{0}^{2\pi} e^{\imath p \varphi}
    \sum_{q\in\mathbb{Z}} \imath^{q}
    J_{q}(\kappa r) e^{\imath p(\theta-\varphi)}
    \;\mathrm{d}\varphi
    = 2\pi \imath^{p}
    J_{q}(\kappa r) e^{\imath p\theta}.
\end{equation}
Hence, we obtain the Herglotz representation of the circular waves,
\begin{equation}\label{eq:Herglotz-representation-of-bp-using-ppw}
    b_{p}(\mathbf{x}) =
    \int_{0}^{2\pi} \left[
        \frac{\beta_{p}}{2\pi \imath^{p}} e^{\imath p \varphi}
    \right]
    \mathrm{PW}_{\varphi}(\mathbf{x}) \;\mathrm{d}\varphi,
\end{equation}
sometimes referred to as Bessel's first integral
identity\ \cite[Eq.~(6)]{PerreyDebain2006}.
The associated Herglotz density,
\(\varphi \mapsto \beta_{p}(2\pi)^{-1} \imath^{-p} e^{\imath p \varphi}\),
is clearly not bounded uniformly with respect to the mode number \(p\),
as a consequence of Lemma~\ref{lem:beta_asymptotic}.
As a result, the discretization of this exact integral representation
(e.g.~by the trapezoidal rule), \emph{cannot}
yield approximate discrete representations with bounded coefficients, as we
establish next.

Moreover, several solutions of the Helmholtz equation can
\emph{not} be represented in the form \eqref{eq:Herglotz-ppw} for any
\(v \in L^{2}([0,2\pi])\).
For any sequence \((\hat u_{p})_{p\in\mathbb{Z}}\in\ell^2(\mathbb Z)\), the
function $u=\sum_{p\in\mathbb Z}\hat u_p b_p$ belongs to $\mathcal B$, because
$\{b_{p}\}_{p\in\mathbb{Z}}$
is a Hilbert basis.
If this \(u\) admits a Herglotz representation in the
form~\eqref{eq:Herglotz-ppw}
then the coefficients \(\{\hat v_{p}\}_{p\in\mathbb{Z}}\) of the Fourier
expansion \eqref{eq:vHerglotz} of the density $v$ satisfy the relation
\(\hat v_{p} = \imath^{-p} \beta_{p} \hat u_{p}\) for all \(p\in\mathbb{Z}\).
For \(v\) to belong to \(L^{2}([0,2\pi])\),
these coefficients would need to belong to \(\ell^{2}(\mathbb{Z})\).
This is only possible if the coefficients \(\{\hat u_{p}\}_{p\in\mathbb{Z}}\)
decay super-exponentially, to compensate for
the growth of \(\{\beta_{p}\}_{p\in\mathbb{Z}}\), again by
Lemma~\eqref{lem:beta_asymptotic}.
For instance, the PPWs themselves are not Herglotz
functions, because their Fourier coefficients do not decay sufficiently fast,
as can be readily seen from the Jacobi--Anger identity~\eqref{eq:Jacobi--Anger}
(in particular, for a PPW
\(|\imath^{-p} \beta_{p} \hat u_{p}|=1\) for all $p$).
In fact, the density \(v\) for a PPW would need to
be a generalized function, the Dirac distribution.

\subsection{Propagative plane waves do not give stable
approximations}\label{sec:instability}

We investigate the approximation of a circular wave \(b_{p}\) for some
\(p\in\mathbb{Z}\) by a generic sequence of approximation sets made of PPWs.
It is shown that the two conditions in~\eqref{eq:stability}, namely accurate
approximation and bounded coefficients, are mutually exclusive.
Thus, stable approximations with PPWs are not possible.

\begin{lemma}\label{lem:instability-propagative}
    Recall the definition of \(b_{p}\) and \(\beta_{p}\) in~\eqref{eq:btilde_p}.
    Let \(p\in\mathbb{Z}\) and some tolerance \(1 \geq \eta > 0\) be given.
    For all \(M\in\mathbb{N}\),
    any approximation set
    \(
        \boldsymbol{\Phi}_{M} := \{M^{-1/2}\,\mathrm{PW}_{\varphi_{m}}\}_{m=1}^{M}
    \)
    made of PPWs with any distribution of angles
    \(
        \{\varphi_{m}\}_{m=1}^{M} \subset [0,2\pi)
    \),
    satisfies
    \begin{equation}\label{eq:lower-bound-coef-instability}
        \forall \boldsymbol{\mu} \in \mathbb{C}^{M},
        \qquad
        \|b_{p} - \mathcal{T}_{\boldsymbol{\Phi}_{M}}\boldsymbol{\mu}\|_{\mathcal{B}}
        \leq \eta \|b_p\|_{\mathcal B}
        \quad\Rightarrow\quad
        \|\boldsymbol{\mu}\|_{\ell^{2}}
        \geq (1-\eta) \beta_{p} \|b_p\|_{\mathcal B}.
    \end{equation}
\end{lemma}
\begin{proof}
    Let \(M \in \mathbb{N}\) and
    \( \boldsymbol{\mu} := \{\mu_{m}\}_{m=1}^{M} \in \mathbb{C}^{M}\).
    Using the Jacobi--Anger identity~\eqref{eq:Jacobi--Anger}
    we obtain at \(\mathbf{x} = (r,\theta) \in B_{1}\)
    \begin{equation}
        \sqrt{M}
        (\mathcal{T}_{\boldsymbol{\Phi}_{M}}\boldsymbol{\mu})(r,\theta)
        = \sum_{1 \leq m \leq M}
        \mu_{m} \sum_{q\in\mathbb{Z}}
        \imath^{q} J_{q}(\kappa r)
        e^{\imath q(\theta-\varphi_{m})}
        = \sum_{q\in\mathbb{Z}}
        \Big(
            \imath^{q}\sum_{1 \leq m \leq M}
            \mu_{m} \; e^{-\imath q\varphi_{m}}
        \Big)
        J_{q}(\kappa r)
        e^{\imath q\theta},
    \end{equation}
    so that 
    \(
        \mathcal{T}_{\boldsymbol{\Phi}_{M}}\boldsymbol{\mu}
        = \sum_{q\in\mathbb{Z}} c_{q}
        \tilde{b}_{p}
    \),
    where the coefficients
    \(c_{q} := \imath^{q}/\sqrt{M}
    \sum_{1 \leq m \leq M}\mu_{m} \;e^{-\imath q\varphi_{m}} \)
    satisfy
    \begin{equation}
        \left| c_q \right| =
        M^{-1/2}
        \Big|
            \imath^q 
            \sum_{1 \leq m \leq M}
            \mu_{m} \;
            e^{-\imath q\varphi_{m}}
        \Big|
        \leq
        M^{-1/2}
        \sum_{1 \leq m \leq M} \left| \mu_{m} \right|
        =
        M^{-1/2}
        \|\boldsymbol{\mu}\|_{\ell^{1}}
            \leq
        \|\boldsymbol{\mu}\|_{\ell^{2}}
        ,  \qquad\forall q\in\mathbb{Z}.
    \end{equation}
    To ensure that the approximation error
    \(\|b_{p} - \mathcal{T}_{\boldsymbol{\Phi}_{M}}\boldsymbol{\mu}\|_{\mathcal{B}}=(\sum_{q\in\mathbb{Z}} 
    |\delta_{p q} - c_{q} \beta_{q}^{-1}|^2)^{1/2}\) is
    below the tolerance \(\eta > 0\), we need at least
    $|\delta_{p q} - c_{q} \beta_{q}^{-1}| < \eta$, $\forall q\in\mathbb{Z}$.
    For $q=p$ this reads
    \begin{equation}
        \eta > \left|1 - c_{p} \beta_{p}^{-1}\right| 
        \ge 1 - |c_{p}| \beta_{p}^{-1}
        \ge 1 - \|\boldsymbol{\mu}\|_{\ell^{2}} \; \beta_{p}^{-1},
    \end{equation}
    which can be rewritten as~\eqref{eq:lower-bound-coef-instability},
    recalling that $\|b_p\|_{\mathcal B}=1$.
\end{proof}

This bound means that if one approximates circular waves \(b_{p}\)
in the form of PPW expansions
\(\mathcal{T}_{\boldsymbol{\Phi}_{M}}\boldsymbol{\mu}\)
with a given accuracy (i.e.~small \(\eta>0\)), then the
norms of the coefficients \(\|\boldsymbol{\mu}\|_{\ell^{2}}\)
need to increase at least like the normalization
constant \(\beta_{p}\), i.e.~super-exponentially fast in \(|p|\),
see Lemma~\ref{lem:beta_asymptotic}.
This is a clear example of accuracy and stability properties being
opposite to each other.
We state this important conclusion as a theorem to stress the
message.

\begin{theorem}\label{th:instability}
    There does \emph{not} exist a sequence of approximation sets made of PPWs
    that is a stable approximation for the space of Helmholtz solutions on the
    disk.
\end{theorem}
\begin{proof}
    Lemma~\ref{lem:instability-propagative} exhibits a particular sequence,
    the sequence of circular waves \(\{b_{p}\}_{p \in \mathbb{Z}}\),
    for which any generic sequence of PPW approximation sets
    \(\{\boldsymbol{\Phi}_{M}\}_{M \in \mathbb{N}}\)
    does not provide stable approximations.
    Indeed, let \(p\in\mathbb{Z}\) and suppose there exist \(M \in \mathbb{N}\)
    and \(\boldsymbol{\mu} \in \mathbb{C}^{M}\) such that
    \( \|b_{p} - \mathcal{T}_{\boldsymbol{\Phi}_{M}}\boldsymbol{\mu}\|_{\mathcal{B}}
    \leq \eta \|b_p\|_{\mathcal B}\)
    for some \(\eta\in(0,1)\).
    Then
    \(
        \|\boldsymbol{\mu}\|_{\ell^{2}}
        \geq (1-\eta) \beta_{p} \|b_p\|_{\mathcal B}
    \),
    which implies that \(\|\boldsymbol{\mu}\|_{\ell^{2}}\) cannot be bounded
    uniformly with respect to \(p\) in virtue of
    Lemma~\ref{lem:beta_asymptotic}.
    The stability condition~\eqref{eq:stability} is not satisfied and we
    conclude that any sequence of PPW approximation sets
    \(\{\boldsymbol{\Phi}_{M}\}_{M \in \mathbb{N}}\) is unstable in the sense
    of Definition~\ref{def:stability}.
\end{proof}

More generally, this statement has implications for other Trefftz methods as
well.
It is not sufficient to study the best approximation error in a space spanned
by Trefftz elements.
If one is interested in numerical methods, one has to study approximation
properties in relation to coefficient norm, and the latter depends not only on
the approximation space but also on its chosen representation,
i.e.\ the approximation set.

In the context of the Method of Fundamental Solutions (MFS), similar
instability results (exponential growth of the coefficient size) are obtained
if the analytic extension of the Helmholtz solution presents a singularity 
closer to the boundary than the MFS charge
points\ \cite[Th.~7]{Barnett2008}.

\paragraph{Modal analysis of a propagative plane wave.}

Another point of view on the same issue is directly given by the Jacobi--Anger
identity \eqref{eq:Jacobi--Anger}.
This identity allows us to get quantitative insight into the modal content
of PPWs.
For any \(\mathbf{x} = (r, \theta) \in B_{1}\) and \(\varphi \in [0,2\pi)\),
we have
\begin{equation}
    e^{\imath \kappa \mathbf{d}(\varphi)\cdot\mathbf{x}}
    = \sum_{p\in\mathbb{Z}}
    \left( \imath^{p} e^{-\imath p\varphi} J_{p}(\kappa r) \right)
    e^{\imath p\theta}
    = \sum_{p\in\mathbb{Z}} 
    \left( \imath^{p} e^{-\imath p\varphi} \beta_{p}^{-1} \right)
    b_{p}(r,\theta).
\end{equation}
The modulus of the coefficients
\(\imath^{p} e^{-\imath p\varphi} \beta_{p}^{-1}\)
in the expansion as a function of \(p\) can be directly deduced from
Lemma~\ref{lem:beta_asymptotic} (for large $|p|$) and is reported in
Figure~\ref{fig:Jacobi--Anger_k16} (left).
This quantity does not depend on the propagation angle \(\varphi\) which
parametrizes the PPW.

These coefficients decay super-exponentially fast in modulus in the evanescent
regime \(|p| \geq \kappa\).
Recalling Remark~\ref{rmk:normalization-btildep}, 
the coefficients with respect to a normalization in alternative sensible norms
(\(L^{2}(B_{1})\), \(L^{2}(\partial B_{1})\) or
\(L^{\infty}(\partial B_{1})\) for instance) modify the decay only by some
moderate powers of \(|p|\).
This does not come as a surprise, since PPWs are entire functions.
Yet, the modal content of any PPW is fixed and low-frequency.
The direct implication is that they are not suited for approximating Helmholtz
solutions with a high-frequency modal content (large \(|p|\)).

\section{Evanescent plane waves}\label{sec:evanescent-plane-waves}

The goal of this section is to introduce \emph{evanescent plane waves} (EPWs) with a complex-valued
direction vector \(\mathbf{d}\in\mathbb{C}^{2}\), as opposed to
propagative ones with \(\mathbf{d}\in\mathbb{R}^{2}\),
and to provide intuitive reasons for their better stability
properties.
PPWs and EPWs are sometimes respectively called
homogeneous and inhomogeneous plane waves, since only the former
have constant amplitude.
Combinations of PPWs and EPWs have already been used to approximate Helmholtz
solutions, e.g.\ in the Wave Base Method\ \cite{Deckers2014}, and Laplace
eigenfunctions, e.g.\ in \cite[\S6.1.3]{Barnett2000}.

\subsection{Definition}

\begin{definition}[Evanescent plane wave]\label{def:evanescent-plane-wave}
    For any parameter
    \(\mathbf{y} := (\varphi, \zeta) \in [0,2\pi) \times \mathbb{R}\),
    we let
    \begin{equation}
        \mathrm{EW}_{\mathbf{y}} (\mathbf{x}) =
        \mathrm{EW}_{\varphi, \zeta} (\mathbf{x}) :=
        e^{\imath \kappa \mathbf{d}({\mathbf{y}})\cdot\mathbf{x}},
        \ \forall \mathbf{x}\in\mathbb{R}^2,
        \quad\text{where}\quad
        \mathbf{d}({\mathbf{y}}) := \big(\cos(\varphi + \imath \zeta),\,
            \sin (\varphi + \imath \zeta)\big)\in\mathbb{C}^{2}.
    \end{equation}
\end{definition}
EPWs are solutions of the
homogeneous Helmholtz equation~\eqref{eq:Helmholtz} since
\(\mathbf{d}({\mathbf{y}}) \cdot \mathbf{d}({\mathbf{y}}) = 1\)
for any \({\mathbf{y}} \in [0,2\pi) \times \mathbb{R}\).
A number of EPWs are illustrated in
Figure~\ref{fig:gpw_k16_real_varphipiover8}.
EPWs can be seen as standard plane waves after the
`complexification' of the angle \(\varphi \in \mathbb{R}\)
into \(\varphi + \imath \zeta \in \mathbb{C}\).
For \(\mathbf{y}=\left(\varphi,0\right)\) (i.e.~setting \(\zeta=0\)), we
recover the usual PPW of
Definition~\ref{def:propagative-plane-wave}, whose direction is defined solely
by the angle \(\varphi\): \(\mathrm{EW}_{\varphi,0} = \mathrm{PW}_\varphi\).

Since the angle is complex, the behavior of the ``wave'' might be unclear.
Two more explicit expressions of EPWs are,
for \(\mathbf{x} = (r,\theta)\in \mathbb{R}^2\):
\begin{equation}
    \begin{aligned}
        & \mathrm{EW}_{\varphi,\zeta}(\mathbf{x})
        = e^{\imath\kappa(\cosh\zeta)\mathbf{x}\cdot\mathbf{d}(\varphi)}\;
        e^{-\kappa(\sinh\zeta)\mathbf{x}\cdot\mathbf{d}^\perp(\varphi)},
        \qquad \text{where} \qquad
        \mathbf{d}^\perp(\varphi):=\left(-\sin \varphi,\, \cos \varphi\right),\\
        \text{and}\quad
        & \mathrm{EW}_{\varphi,\zeta}(\mathbf{x})
        = e^{\imath \kappa r (\cosh\zeta) \cos\left(\varphi-\theta\right)}\;
        e^{\kappa r (\sinh\zeta)  \sin\left(\varphi-\theta\right) }.
    \end{aligned}
\end{equation}
We see from these formulas that the wave \emph{oscillates} with apparent
wavenumber \(\kappa \cosh\zeta \ge\kappa\) in the direction of
\(\mathbf{d}(\varphi) := \left(\cos \varphi,\, \sin \varphi\right)\),
which was defined in~\eqref{eq:direction_propagative-plane-wave}
and is parallel to \(\Re[ \mathbf{d}(\mathbf{y}) ]\).
In addition, the wave \emph{decays} exponentially with rate
\(\kappa\sinh\zeta\) in the orthogonal direction
\(\mathbf{d}(\varphi)^{\perp}\),
which is parallel to \(\Im[ \mathbf{d}(\mathbf{y}) ]\).
This justifies naming the new parameter \(\zeta \in \mathbb{R}\),
which controls the imaginary part of the angle, the \emph{evanescence}
parameter.

\begin{figure}[htb]
    \centering
    \includegraphics[width=0.2\textwidth]{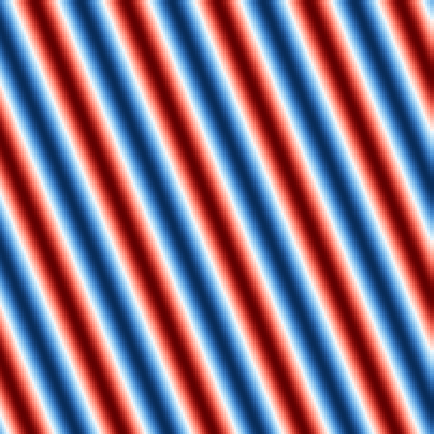}
    \hspace{0.1\textwidth}
    \includegraphics[width=0.2\textwidth]{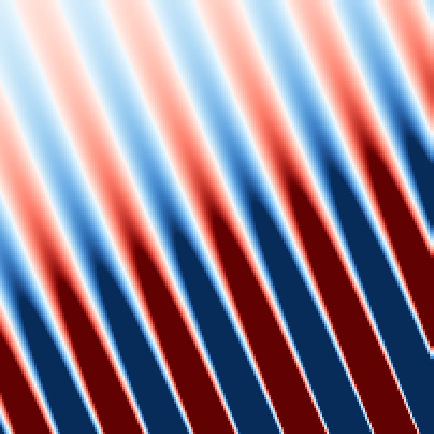}
    \hspace{0.1\textwidth}
    \includegraphics[width=0.2\textwidth]{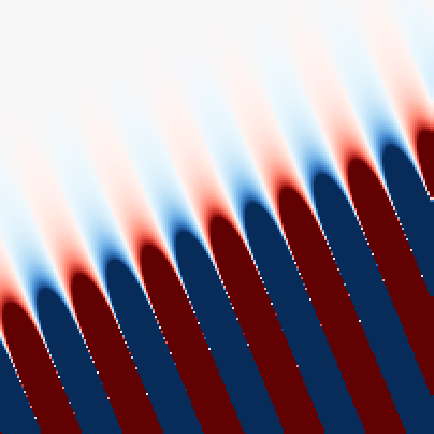}
    \caption{Real part of \(\mathrm{EW}_{\varphi,\zeta}\)
    with \(\varphi=\pi/8\), \(\zeta \in \{0,1/10,1/2\}\) (left to right) and
    \(\kappa=16\).}\label{fig:gpw_k16_real_varphipiover8}
\end{figure}

\subsection{Modal analysis of evanescent plane waves}

The Jacobi--Anger expansion~\eqref{eq:Jacobi--Anger} extends to complex
\(\mathbf{d}\), i.e.\ to EPWs, see\ \cite[Eq.~(10.12.1),
(10.11.1)]{DLMF}: for any \(\mathbf{x} = (r, \theta) \in B_{1}\)
and \(\mathbf{y} = (\varphi, \zeta) \in [0,2\pi) \times \mathbb{R}\),
\begin{equation}\label{eq:Jacobi--Anger-epw}
    \mathrm{EW}_{\mathbf{y}} (\mathbf{x}) =
    e^{\imath \kappa \mathbf{d}(\mathbf{y})\cdot\mathbf{x}}
    = \sum_{p\in\mathbb{Z}}
    \imath^{p} J_{p}(\kappa r)
    e^{\imath p(\theta-[\varphi+\imath\zeta])}
    = \sum_{p\in\mathbb{Z}} 
    \left( \imath^{p} e^{-\imath p\varphi} e^{p\zeta} \beta_{p}^{-1} \right)
    b_{p}(r,\theta).
\end{equation}
The modulus of the coefficients
\(\imath^{p} e^{-\imath p\varphi} e^{p\zeta} \beta_{p}^{-1}\)
in the modal expansion are reported in
Figure~\ref{fig:Jacobi--Anger_k16} (right) as functions of \(p\).
On this graph, we have conveniently normalized the coefficients
according to a normalization factor (depending only on \(\zeta\))
which is described in the following sections, see~\eqref{eq:evanescent-pw-sets}.
We see that by tuning the evanescence parameter \(\zeta\) we are able to shift
the modal content of the plane waves to higher-frequency regimes.
As a result, we expect EPWs to be able to capture well
the higher-frequency modes of Helmholtz solutions that are less regular. These may arise, for instance, 
in the presence of close-by singularities.
The difficulty then is to properly choose suitable values for this new
evanescence parameter \(\zeta\) in order to build approximation spaces that are
reasonable in size.
This will be the main objective of the remainder of this paper.

\begin{figure}[htb]
    \centering
    \includegraphics[height=0.2\textheight]{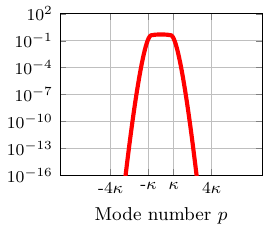}
    \hspace{0.1\textwidth}
    \includegraphics[height=0.2\textheight]{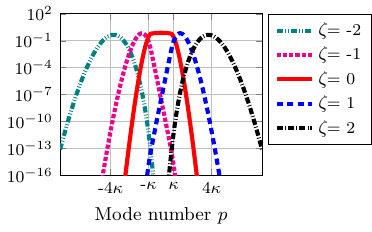}
    \caption{Modal analysis computed using Jacobi--Anger identity~\eqref{eq:Jacobi--Anger-epw}
    of \(\mathrm{PW}_{\varphi}\) (left)
    and \(\mathrm{EW}_{\varphi,\zeta}\) after normalization (right).
    In both cases, the absolute values of the
    coefficients of the expansion of the plane wave in the basis
    \(\{b_p\}_{p\in\mathbb Z}\) is represented against the mode number \(p\).
    Wavenumber \(\kappa=16\).
    Modifying \(\varphi\) has no influence, modifying \(\zeta\) shifts the
    modal content in the Fourier space.
    }\label{fig:Jacobi--Anger_k16}
\end{figure}

\section{Mapping Herglotz densities to Helmholtz solutions}\label{sec:continuous-analysis}

In this section we introduce an integral transform
between a space of functions defined on the parametric domain
\([0,2\pi) \times \mathbb{R}\) and the space of Helmholtz solutions in the unit
disk~\(\mathcal{B}\).

\subsection{Space of Herglotz densities}

To shorten notations we denote the parametric domain as the cylinder
\begin{equation}
    Y := [0,2\pi) \times \mathbb{R}.
\end{equation}
We introduce a weighted \(L^{2}\) space defined on \(Y\).
The weight function is (the square of)
\begin{equation}\label{eq:weight_function}
    w_{z}(\mathbf{y}) = w_{z}(\zeta) := 
    e^{-\kappa\sinh|\zeta|+z|\zeta|},
    \qquad\forall \mathbf{y}=(\varphi,\zeta)\in Y,
\end{equation}
for some \(z\in\mathbb{R}\).
In this section, the parameter \(z\) is temporarily not specified, although
the following analysis shows that it cannot be chosen freely
and should take the specific value \(z=1/4\),
see~\eqref{eq:def-z-to-be-a-quarter}.
We stress that \(w_{z}\) does not depend on the angle \(\varphi\).
The weighted scalar product and associated norm are then defined by:
\begin{equation}
    \left(u,\, v\right)_{\mathcal{A}} :=
    \int_{Y} u(\mathbf{y})\overline{v(\mathbf{y})}
    \; w_{z}^{2}(\mathbf{y})\mathrm{d}\mathbf{y},
    \qquad\qquad
    \|u\|_{\mathcal{A}}^{2} :=
    \left(u,\, u\right)_{\mathcal{A}}.
\end{equation}

We now introduce a subspace of \(L^{2}(Y;w_{z}^{2})\) which we call space of
\emph{Herglotz densities} for reasons that will be clear in the following.
\begin{definition}[Herglotz density]\label{def:Herglotz-space}
    We define, for any \(p\in\mathbb{Z}\),
    \begin{equation}\label{eq:atildep}
        \begin{cases}
            \tilde{a}_{p}(\mathbf{y}) := 
            e^{p\zeta} e^{\imath p \varphi},
            \quad\forall \mathbf{y}=\left(\varphi,\zeta\right)\in Y,\\
            a_{p} := \alpha_{p}\tilde{a}_{p},
            \quad\text{where}\quad
            \alpha_{p} := \|\tilde{a}_{p}\|_{\mathcal{A}}^{-1},
        \end{cases}
        \quad\text{and}\qquad
        \mathcal{A} := \overline{
            \operatorname{span}
            \left\{{a}_{p}\right\}_{p\in\mathbb{Z}}
        }^{\|\cdot\|_{\mathcal{A}}}
        \subsetneq L^{2}(Y;w_{z}^{2}).
    \end{equation}
\end{definition}

The wavenumber \(\kappa\) appears explicitly in the weight function
\(w_{z}\).
Therefore, each \(a_{p}\) for \(p\in\mathbb{Z}\) has an implicit dependence in
the wavenumber \(\kappa\) through the normalization factor \(\alpha_{p}\).
Some functions \(a_{p}\), weighted by \(w_{1/4}\)
(see~\eqref{eq:def-z-to-be-a-quarter}), are represented in
Figure~\ref{fig:wap_k16}.

\begin{figure}
    \centering
    \includegraphics{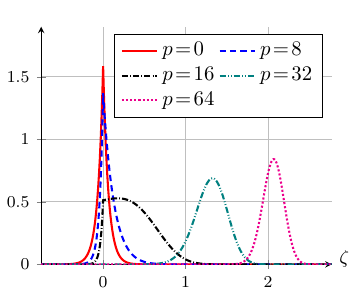}
    \caption{Representation of
    \(\zeta \mapsto |w_{1/4}(\zeta) a_{p}(\zeta,\cdot)|\),
    which is independent of the second argument of the function \(a_{p}\),
    for mode number \(p \in \{0,\kappa/2,\kappa,2\kappa,4\kappa\}\)
    and wavenumber \(\kappa=16\).}\label{fig:wap_k16}
\end{figure}

For any \(p\in\mathbb{Z}\),
the complex-valued function \((\zeta+\imath\varphi) \mapsto
a_{p}(\varphi,\zeta)\) is a holomorphic function of the complex variable
\(\zeta+\imath\varphi \in \mathbb{C}\) for any \((\varphi,\zeta) \in Y\).
It follows that its real and imaginary parts are harmonic functions on the
cylinder \(Y\).

\begin{lemma}\label{lem:a_Hilbert_basis}
    The space
    \(\left(\mathcal{A},\, \|\cdot\|_{\mathcal{A}}\right)\)
    is a Hilbert space and
    the family \(\{a_{p}\}_{p\in\mathbb{Z}}\) is a Hilbert basis:
    \begin{equation}
        \left(a_{p},\, a_{q}\right)_{\mathcal{A}}
        =\delta_{pq},
        \qquad\forall p,q\in \mathbb{Z},
        \qquad\text{and}\qquad
        v = \sum_{p\in\mathbb{Z}}
        \left(v,\, a_{p}\right)_{\mathcal{A}} a_{p},
        \qquad\forall v\in \mathcal{A}.
    \end{equation}
\end{lemma}
The coefficients \(\alpha_{p}\) defined in~\eqref{eq:atildep} decay
super-exponentially with \(|p|\) after a pre-asymptotic regime up to
\(|p|\approx \kappa\).
The precise asymptotic behavior is given by the following lemma.
\begin{lemma}\label{lem:alpha_asymptotic}
    For a constant \(c(\kappa)\) only depending on \(\kappa\), we have
    \begin{equation}\label{eq:alpha_asymptotic}
        \alpha_{p} \sim c(\kappa)\;
        \left(\frac{e\kappa}{2}\right)^{|p|}|p|^{1/4-z-|p|}
        \qquad\text{as}\ |p|\to +\infty.
    \end{equation}
\end{lemma}
\begin{proof}
    It is clear that \(\alpha_{-p}=\alpha_{p}\) for all \(p\in\mathbb{Z}\).
    Let \(p \in \mathbb{N}\), we have
    \begin{equation}\label{eq:bounds-L2-norm}
        \begin{alignedat}{3}
            2\pi\int_{-\infty}^{+\infty}
            e^{2p\zeta+2z|\zeta|}
            e^{-2\kappa\sinh\left|\zeta\right|}
            \;\mathrm{d}\zeta
            &= \|\tilde{a}_{p}\|_{\mathcal{A}}^2 &&\leq
            2\pi\int_{-\infty}^{+\infty}
            e^{2p|\zeta|+2z|\zeta|}
            e^{-2\kappa\sinh\left|\zeta\right|}
            \;\mathrm{d}\zeta,\\
            2\pi\int_{0}^{+\infty}
            e^{2(p+z)\zeta}
            e^{-2\kappa\sinh\zeta}
            \;\mathrm{d}\zeta
            &\leq \|\tilde{a}_{p}\|_{\mathcal{A}}^2 &&\leq
            4\pi\int_{0}^{+\infty}
            e^{2(p+z)\zeta}
            e^{-2\kappa\sinh\zeta}
            \;\mathrm{d}\zeta,\\
            2\pi\kappa^{-m}\int_{\kappa}^{+\infty}
            \eta^{m-1}
            e^{-\eta+\frac{\kappa^2}{\eta}}
            \;\mathrm{d}\eta
            &\leq \|\tilde{a}_{p}\|_{\mathcal{A}}^2 &&\leq
            4\pi\kappa^{-m}\int_{\kappa}^{+\infty}
            \eta^{m-1}
            e^{-\eta+\frac{\kappa^2}{\eta}}
            \;\mathrm{d}\eta,\\
            2\pi\kappa^{-m}\int_{\kappa}^{+\infty}
            \eta^{m-1}
            e^{-\eta}
            \;\mathrm{d}\eta
            &\leq \|\tilde{a}_{p}\|_{\mathcal{A}}^2 &&\leq
            4\pi e^{\kappa}\kappa^{-m} 
            \int_{\kappa}^{+\infty}
            \eta^{m-1}
            e^{-\eta}
            \;\mathrm{d}\eta,\\
            2\pi\kappa^{-m}\Gamma(m,\kappa)
            &\leq \|\tilde{a}_{p}\|_{\mathcal{A}}^2 &&\leq
            4\pi e^{\kappa}\kappa^{-m}\Gamma(m,\kappa),
        \end{alignedat}
    \end{equation}
    where we used the change of variable \(\eta = {\kappa}e^{\zeta}\),
    introduced \(m=2(p+z)\) and used the upper incomplete Gamma function
    defined in\ \cite[Eq.~(8.2.2)]{DLMF}.
    The Gamma function \(\Gamma(m)\) and the upper incomplete counterpart
    \(\Gamma(m,\kappa)\) have the same asymptotic behavior for a fixed
    \(\kappa\) when \(m\) goes to infinity,
    see\ \cite[Eq.~(8.11.5)]{DLMF}
    which gives the asymptotic behavior of \(1-\Gamma(m,\kappa)/\Gamma(m)\).
    Using\ \cite[Eq.~(5.11.3)]{DLMF} we get
    $\Gamma(m,\kappa) \sim \Gamma(m) \sim \sqrt{2\pi} e^{-m} m^{m-1/2}$, as
    $m\to\infty$.
    We obtain
    \begin{equation}
        \kappa^{-2(p+z)}\Gamma\big(2(p+z),\kappa\big)
        \sim
        \sqrt{\pi}
        \left(\frac{2}{e\kappa}\right)^{2(p+z)}\; p^{2(p+z)-1/2}
        \left(1+\frac{z}{p}\right)^{2(p+z)-1/2}
        \quad\text{as }p\to+\infty,
    \end{equation}
    and the last term is in fact equivalent to \(e^{2z}\) at infinity;
    the claimed result follows.
\end{proof}

Using our definitions, the Jacobi--Anger expansion \eqref{eq:Jacobi--Anger-epw}
takes the simple form
\begin{equation}\label{eq:Jacobi--Anger-apbp}
   {\mathrm{EW}_{\mathbf{y}}(\mathbf{x})
    = \sum_{p\in\mathbb{Z}} \imath^{p} \;
    \overline{\tilde{a}_{p}(\mathbf{y})} \; \tilde{b}_{p}(\mathbf{x})
    = \sum_{p\in\mathbb{Z}} \tau_{p} \;
    \overline{a_{p}(\mathbf{y})} \; b_{p}(\mathbf{x}),
    \qquad\forall (\mathbf{x},\,\mathbf{y})\in B_{1} \times Y,}
\end{equation}
where we introduced
\begin{equation}\label{eq:tau_p}
    \tau_{p} := 
    \imath^{p} \left(\alpha_{p}\beta_{p}\right)^{-1},
    \qquad\forall p\in\mathbb{Z}.
\end{equation}
Formula~\eqref{eq:Jacobi--Anger-apbp} relates the basis
$\{a_p\}_{p\in\mathbb Z}$ of the space $\mathcal A$ to EPWs
$\mathrm{EW}_{\mathbf y}$ and circular waves $b_p$ on $B_1$ and is the key reason for
introducing the space $\mathcal A$.
The behavior of \(|\tau_{p}|\) is of crucial importance in the following
analysis and is given in Figure~\ref{fig:tau_all} for various wavenumber
\(\kappa\).
From the asymptotics given in Lemma~\ref{lem:beta_asymptotic} and
Lemma~\ref{lem:alpha_asymptotic} we deduce the following result.
\begin{lemma}\label{lem:tau_bounds}
    We have
    \begin{equation}
        |\tau_{p}| \sim c(\kappa)\;|p|^{z-1/4}
        \qquad\text{as}\ |p|\to +\infty,
    \end{equation}
    where the constant \(c(\kappa)\) only depends on \(\kappa\).
    Hence, choosing \(z=1/4\), we get
    \begin{equation}\label{eq:def_taupm}
        \tau_{-}:=\inf_{p\in\mathbb{Z}}|\tau_{p}| > 0,
        \qquad\text{and}\qquad
        \tau_{+}:=\sup_{p\in\mathbb{Z}}|\tau_{p}| < \infty.
    \end{equation}
\end{lemma}
It is clear that the uniform bounds for \(|\tau_{p}|\) are possible only for a
precise pair of norms for the space of Helmholtz solutions and the space of
Herglotz densities.
The bounds \(\tau_{\pm}\) depend implicitly on the wavenumber \(\kappa\),
see Figure~\ref{fig:tau_all}.

\begin{figure}
    \centering
    \includegraphics[width=0.42\textwidth]{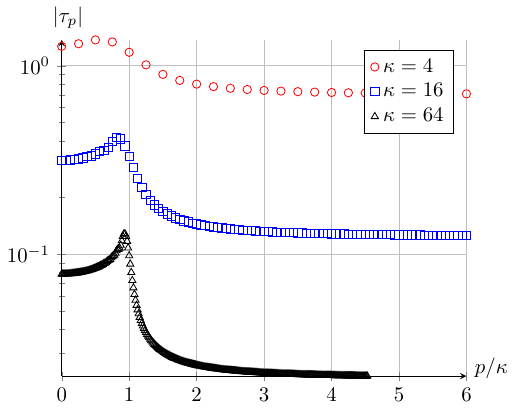}
    \includegraphics[width=0.42\textwidth]{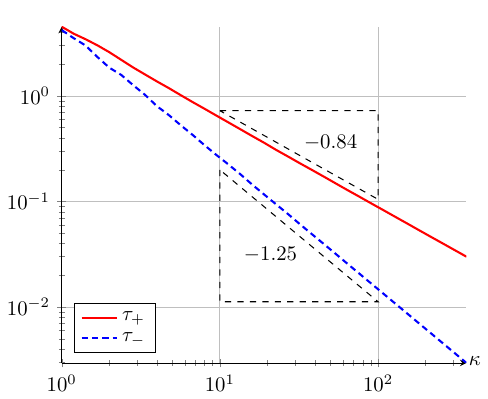}
    \caption{Left: dependence of \(|\tau_{p}|\) defined in~\eqref{eq:tau_p} on
    the mode number \(p\) for various wavenumber \(\kappa\) and \(z=1/4\).
    Right: dependence of \(\tau_{\pm}\) defined in~\eqref{eq:def_taupm}
    on the wavenumber \(\kappa\).}\label{fig:tau_all}
\end{figure}

The uniform boundedness of \(\tau_{p}\) is the key to the following analysis.
In the remainder of the paper, we set \(z = 1/4\) in~\eqref{eq:weight_function}
and we let 
\begin{equation}\label{eq:def-z-to-be-a-quarter}
    w := w_{1/4}.
\end{equation}

We conclude this subsection with a lemma that will be useful in the
following.
\begin{lemma}\label{lem:epw-in-A}
    For any \(\mathbf{x} \in B_{1}\),
    \(\mathbf{y} \mapsto \overline{\mathrm{EW}_{\mathbf{y}}(\mathbf{x})}
    \in \mathcal{A}\).
\end{lemma}
\begin{proof}
    Let \(\mathbf{x} \in B_{1}\) and define
    \(v_{\mathbf{x}} : \mathbf{y}
    \mapsto \overline{\mathrm{EW}_{\mathbf{y}}(\mathbf{x})}\).
    The Jacobi--Anger identity~\eqref{eq:Jacobi--Anger-apbp} reads
    $v_{\mathbf{x}}(\mathbf{y})
        = \sum_{p\in\mathbb{Z}} \overline{\tau_{p}} \;
        \overline{b_{p}(\mathbf{x})} \; a_{p}(\mathbf{y})$
        for all $\mathbf{y} \in Y$.
    Since \(\{a_{p}\}_{p\in\mathbb{Z}}\) is a Hilbert basis for
    \(\mathcal{A}\),
    if we write \(\mathbf{x} = (r,\theta) \in [0,1) \times [0,2\pi)\),
    we get
    \begin{equation}
        \|v_{\mathbf{x}}\|_{\mathcal{A}}^{2}
        = \sum_{p\in\mathbb{Z}} |\tau_{p} b_{p}(\mathbf{x})|^{2}
        \leq \tau_{+}^{2} \sum_{p\in\mathbb{Z}}
        \beta_{p}^{2} |J_{p}(\kappa r)|^{2}.
    \end{equation}
    Using the estimates~\eqref{eq:beta_explicit}
    and~\eqref{eq:bessel-asymptotic} from the proof of
    Lemma~\ref{lem:beta_asymptotic},
    we get 
    \begin{equation}
        \beta_{p}^{2} |J_{p}(\kappa r)|^{2}
        \sim \frac{\kappa^2}{2\pi} \frac{r^{2|p|}}{|p|},
        \qquad\text{as}\ |p|\to +\infty,
    \end{equation}
    from which we conclude that
    \(\|v_{\mathbf{x}}\|_{\mathcal{A}}^{} < \infty\).
\end{proof}

If \(\mathbf{x} \in \partial B_{1}\), so that \(r = |\mathbf{x}| = 1\),
then \(\mathbf{y} \mapsto \overline{\mathrm{EW}_{\mathbf{y}}(\mathbf{x})}\) does not
belong to \(\mathcal{A}\),
as is readily seen from the proof of Lemma~\ref{lem:epw-in-A}.

\subsection{Herglotz transform}

We introduce an integral operator $T$ that allows to write every Helmholtz
solution in $\mathcal B$ as a continuous linear combination of EPWs
weighted by an element of $\mathcal A$.
We also describe its adjoint operator $T^*$, the corresponding frame and Gram
operators $S$ and $G$, and prove some of their properties.
The terminology of this section is borrowed from \emph{Frame Theory},
see\ \cite{Christensen2016} for a reference on this field.

\paragraph{Synthesis operator.}

The first and most important definition concerns the transform that maps
Herglotz densities to Helmholtz solutions as we prove next.

\begin{definition}\label{def:Herglotz-transform}
    Using the weight~\eqref{eq:def-z-to-be-a-quarter},
    we introduce the Herglotz transform \(T\):
    for any \(v\in \mathcal{A}\), 
    \begin{equation}\label{eq:Herglotz-transform}
        \boxed{
            (Tv) (\mathbf{x})
            := \int_{Y}
            v(\mathbf{y})
            \mathrm{EW}_{\mathbf{y}}(\mathbf{x})
            \; w_{}^{2}(\mathbf{y})\mathrm{d}\mathbf{y},
            \qquad\forall \mathbf{x}\in B_{1}.
        }
    \end{equation}
\end{definition}

This operator is well-defined on \(\mathcal{A}\) thanks to
Lemma~\ref{lem:epw-in-A}.
In the setting of continuous-frame theory,
see e.g.\ \cite[Eq.~(5.27)]{Christensen2016},
this operator is called \emph{synthesis} operator.

The Herglotz transform \(T\) is bounded and invertible between
the space of Herglotz densities \(\mathcal{A}\) and the space of Helmholtz
solutions \(\mathcal{B}\).
\begin{theorem}\label{th:T_op}
    The operator \(T\) is bounded and invertible 
    from \(\mathcal{A}\) to  \(\mathcal{B}\):
    \begin{equation}\label{eq:T_op_bounds}
        T\;:\;
        \mathcal{A} \to \mathcal{B},\quad
        v \mapsto 
        \sum_{p\in\mathbb{Z}} \tau_{p}
        \left(v,\, a_{p}\right)_{\mathcal{A}} b_{p},\\
        \quad\text{and}\quad
        \tau_{-} \|v\|_{\mathcal{A}}
        \leq \|Tv\|_{\mathcal{B}} \leq
        \tau_{+} \|v\|_{\mathcal{A}}
        \ \forall v\in\mathcal{A}.
    \end{equation}
    Moreover, $T a_p=\tau_p b_p$ for all $p\in\mathbb Z$.
\end{theorem}
\begin{proof}
    Using the Jacobi--Anger formula~\eqref{eq:Jacobi--Anger-apbp},
    for any \(v\in \mathcal{A}\) and \(\mathbf{x} \in B_{1}\) we get
    \begin{equation}
        \begin{aligned}
            (Tv)(\mathbf{x})
            &= \int_{Y}
            \mathrm{EW}_{\mathbf{y}}(\mathbf{x}) v(\mathbf{y})
            \; w_{}^{2}(\mathbf{y})\mathrm{d}\mathbf{y}
            = \int_{Y}
            \left(
                \sum_{p\in\mathbb{Z}} \tau_{p}\;
                b_{p}(\mathbf{x})\overline{a_{p}(\mathbf{y})}
            \right)\; v(\mathbf{y})
            \; w_{}^{2}(\mathbf{y})\mathrm{d}\mathbf{y}\\
            &= \sum_{p\in\mathbb{Z}} \tau_{p}
            \int_{Y}
                \overline{a_{p}(\mathbf{y})} \; v(\mathbf{y})
            \; w_{}^{2}(\mathbf{y})\mathrm{d}\mathbf{y}
            \; b_{p}(\mathbf{x})
            = \sum_{p\in\mathbb{Z}} \tau_{p}
            \left(v,\, a_{p}\right)_{\mathcal{A}}
            b_{p}(\mathbf{x}).
        \end{aligned}
    \end{equation}
    Hence, from Lemma~\ref{lem:b_Hilbert_basis},
    \(
    \|Tv\|_{\mathcal{B}}^2 = \sum_{p\in\mathbb{Z}} |\tau_{p}|^{2}
    |\left(v,\, a_{p}\right)_{\mathcal{A}}|^2,
    \)
    and the result~\eqref{eq:T_op_bounds} follows from
    Lemma~\ref{lem:a_Hilbert_basis} and Lemma~\ref{lem:tau_bounds}.
    It is readily checked that the inverse is given,
    for any \(u\in\mathcal{B}\), by
    \begin{equation}\label{eq:Tinv_op_series}
        T^{-1}u
        = \sum_{p\in\mathbb{Z}} \tau_{p}^{-1}
        \left(u,\, b_{p}\right)_{\mathcal{B}}
        a_{p}.
    \end{equation}
\end{proof}

From~\eqref{eq:Tinv_op_series},
the inverse operator \(T^{-1}\) can also be written as an integral operator:
for \(u \in \mathcal{B}\),
\begin{align*}
    (T^{-1}u)(\mathbf y)
    =&\int_{B_1} u(\mathbf x) \Psi (\mathbf{x}, \mathbf{y})
    \;\mathrm{d}\mathbf{x} + \kappa^{-2}
    \int_{B_1} \nabla u(\mathbf x) \cdot \nabla \Psi(\mathbf{x}, \mathbf{y})
    \;\mathrm{d}\mathbf{x},
    \qquad\forall \mathbf{y}\in Y,
    \\
    \text{where }\; \Psi (\mathbf{x}, \mathbf{y})
    :=&\sum_{p\in\mathbb Z} 
    \tau_{p}^{-1} a_{p}(\mathbf{y}) \overline{b_{p}(\mathbf{x})}
    \qquad\forall \mathbf{x}\in B_{1},\;\mathbf{y}\in Y.
\end{align*}

The integral representation \(Tv\) in \eqref{eq:Herglotz-transform}
is similar to the Herglotz representation~\eqref{eq:Herglotz-ppw}.
This is the reason why we refer to elements of \(\mathcal{A}\) as
\emph{Herglotz densities}.
For any \(p \in \mathbb{Z}\),
the Herglotz densities \(\tau_{p}^{-1}a_{p}\) of the circular waves \(b_{p}\)
are bounded in the \(\mathcal{A}\)-norm by \(\tau_{-}^{-1}\),
hence uniformly with respect to the index \(p\).
This should be contrasted with the standard Herglotz
representation~\eqref{eq:Herglotz-representation-of-bp-using-ppw}
using only PPWs, where the associated Herglotz densities
cannot be bounded uniformly with respect to the index \(p\)
in \(L^{2}([0,2\pi])\).
As we explained in Section~\ref{sec:Herglotz-representation},
not all Helmholtz solutions admit a bounded Herglotz representation that uses only
PPWs~\eqref{eq:Herglotz-ppw}
(with density \(v\in L^{2}([0,2\pi])\)).
In contrast, using EPWs the generalized
Herglotz representation~\eqref{eq:Herglotz-transform} can represent any
Helmholtz solution.
Indeed, since \(T\) is an isomorphism between \(\mathcal{A}\) and
\(\mathcal{B}\), for any \(u\in\mathcal{B}\), there exists a unique
\(v\in\mathcal{A}\) such that \(u=Tv\).
The price to pay for this result is the need for a two-dimensional parameter
domain, the cylinder $Y$, in place of a one-dimensional one, the interval $[0,2\pi)$,
and thus of a double integral; the added dimension corresponds to the
evanescence parameter $\zeta$.

Theorem~\ref{th:T_op} is a stability result stated at the continuous level.
Next, we aim to obtain a discrete version of this integral
representation.
    
\paragraph{Analysis operator.}

In the continuous-frame setting, see\ \cite[Eq.~(5.28)]{Christensen2016}, the
adjoint operator \(T^{*}\) of \(T\) is called \emph{analysis} operator.
\begin{lemma}\label{lem:Tstar_op}
    The adjoint \(T^{*}\) of \(T\) is given for any \(u\in\mathcal{B}\) by
    $   (T^{*}u)(\mathbf{y}) :=
        \left( u,\, \mathrm{EW}_{\mathbf{y}}\right)_{\mathcal{B}}$, 
    $\forall \mathbf{y}\in Y$.
    The operator \(T^{*}\) is bounded and invertible on \(\mathcal{B}\):
    \begin{equation}
        T^{*}\;:\;
        \mathcal{B} \to \mathcal{A},\ 
        u \mapsto 
        \sum_{p\in\mathbb{Z}} \overline{\tau_{p}}
        \left(u,\, b_{p}\right)_{\mathcal{B}} a_{p},
        \quad\text{and}\quad
        \tau_{-} \|u\|_{\mathcal{B}}
        \leq \|T^{*}u\|_{\mathcal{A}} \leq
        \tau_{+} \|u\|_{\mathcal{B}},
        \quad\forall u\in\mathcal{B}.
    \end{equation}
\end{lemma}
\begin{proof}
    We have, for any \(v\in\mathcal{A}\) and \(u\in\mathcal{B}\)
    \begin{equation}
        \begin{aligned}
            \left(Tv,\, u\right)_{\mathcal{B}}
            &= 
            \left(\int_{Y}
                \mathrm{EW}_{\mathbf{y}} v(\mathbf{y})
                \; w_{}^{2}(\mathbf{y})\mathrm{d}\mathbf{y},
                \, u
            \right)_{\mathcal{B}}
            = \int_{Y} v(\mathbf{y})
            \left(
                \mathrm{EW}_{\mathbf{y}},
                \, u
            \right)_{\mathcal{B}}
            \; w_{}^{2}(\mathbf{y})\mathrm{d}\mathbf{y}\\
            &= \left(
                v,\,
                \overline{\left(
                    \mathrm{EW}_{\mathbf{y}},
                    \, u
                \right)_{\mathcal{B}}}
            \right)_{\mathcal{A}}
            = \left(
                v,\,
                \left(
                    u,
                    \, \mathrm{EW}_{\mathbf{y}}
                \right)_{\mathcal{B}}
            \right)_{\mathcal{A}}.
        \end{aligned}
    \end{equation}
    In addition, using the Jacobi--Anger formula~\eqref{eq:Jacobi--Anger-apbp},
    for any \(u\in \mathcal{B}\)
    and \(\mathbf{y} \in Y\)
    \begin{equation}
        (T^{*}u)(\mathbf{y})
        = \left(
            u,\, \mathrm{EW}_{\mathbf{y}}
        \right)_{\mathcal{B}}
        = \bigg(
            u,\,
            \sum_{p\in\mathbb{Z}} {\tau_{p}}\;
            \overline{a_{p}(\mathbf{y})}{b_{p}}
        \bigg)_{\mathcal{B}}
        = \sum_{p\in\mathbb{Z}} \overline{\tau_{p}}
        \left(u,\, b_{p}\right)_{\mathcal{B}}
        a_{p}(\mathbf{y}).
    \end{equation}
    From Lemma~\ref{lem:b_Hilbert_basis},
     $   \|T^{*}u\|_{\mathcal{A}}^2
        = \sum_{p\in\mathbb{Z}} |\tau_{p}|^{2}
        |\left(u,\, b_{p}\right)_{\mathcal{B}}|^2,
    $
    and the result follows from Lemma~\ref{lem:tau_bounds}.
\end{proof}

\paragraph{Frame and Gram operators.}

We introduce two other important operators in Frame Theory.

\begin{corollary}\label{cor:SG_SPD}
    The \emph{frame operator} \(S:=TT^{*}\) and
    the \emph{Gram operator} \(G:=T^{*}T\)
    are bounded, invertible, self-adjoint and positive operators:
    \begin{equation}
        \begin{aligned}
            & S:=TT^{*} \;:\; \mathcal{B} \to \mathcal{B},\ 
            u \mapsto \sum_{p\in\mathbb{Z}} |\tau_{p}|^{2}
            \left(u,\, b_{p}\right)_{\mathcal{B}} b_p,
            \quad \text{and} \quad
            \tau_{-}^{2} \|u\|_{\mathcal{B}}
            \leq \|Su\|_{\mathcal{B}} \leq
            \tau_{+}^{2} \|u\|_{\mathcal{B}},\ \forall u \in\mathcal{B},\\
            & G:=T^{*}T \;:\; \mathcal{A} \to \mathcal{A},\ 
            v \mapsto \sum_{p\in\mathbb{Z}} |\tau_{p}|^{2}
            \left(v,\, a_{p}\right)_{\mathcal{A}} a_p,
            \quad \text{and} \quad
            \tau_{-}^{2} \|v\|_{\mathcal{A}}
            \leq \|Gv\|_{\mathcal{A}} \leq
            \tau_{+}^{2} \|v\|_{\mathcal{A}},\ \forall v \in\mathcal{A}.
        \end{aligned}
    \end{equation}
\end{corollary}
\begin{proof}
    This result stems directly from Theorem~\ref{th:T_op}
    and Lemma~\ref{lem:Tstar_op}.
\end{proof}

The frame operator admits the more explicit formula:
for any \(u \in \mathcal{B}\),
\begin{equation}
    Su(\mathbf{x}) = 
    \int_{Y} \left(u,\, \mathrm{EW}_{\mathbf{y}}\right)_{\mathcal{B}}
    \mathrm{EW}_{\mathbf{y}}(\mathbf{x})
    \; w_{}^{2}(\mathbf{y})\mathrm{d}\mathbf{y},
    \qquad\forall \mathbf{x}\in B_{1}.
\end{equation}

\paragraph{A continuous frame result.}

We are now ready to prove that EPWs form a continuous
frame for the space of Helmholtz solutions in the unit disk.
We recall \cite[Def.~5.6.1]{Christensen2016}:
given a complex Hilbert space $\mathcal H$ and a measure space $M$ with
positive measure $\mu$, a family $\{f_k\}_{k\in M}\subset\mathcal H$ is called
``continuous frame'' if, $\forall f\in\mathcal H$, $k\mapsto\langle
f,f_k\rangle$ is measurable in $M$, and $\exists A,B>0$ such that
$A\|f\|^2\le\int_M|\langle f,f_k\rangle|^2\mathrm d\mu(k)\le B\|f\|^2$.

\begin{theorem}\label{th:continuous_frame}
    The family \(\left\{\mathrm{EW}_{\mathbf{y}}\right\}_{\mathbf{y}\in Y}\)
    is a continuous frame for \(\mathcal{B}\).
    Besides, the optimal frame bounds are \(A=\tau_{-}^2\) and
    \(B=\tau_{+}^2\).
\end{theorem}
\begin{proof}
    We need to verify the definition of a continuous frame,
    see\ \cite[Def.~5.6.1]{Christensen2016}.
    For any \(u\in\mathcal{B}\), the measurability of
        $\mathbf{y}\mapsto\left(u,\mathrm{EW}_{\mathbf{y}}\right)_{\mathcal{B}}
        = (T^{*}u)(\mathbf{y})$,
    stems from \(T^{*}u \in \mathcal{A}\) according to Lemma~\ref{lem:Tstar_op}
    and \(\mathcal{A} \subset L^{2}(Y;w^{2})\).
    The frame condition, namely
    \begin{equation}
        A \|u\|_{\mathcal{B}}^2 \leq
        \int_{Y} |\left(u,\, \mathrm{EW}_{\mathbf{y}}\right)_{\mathcal{B}}|^2
        \; w_{}^{2}(\mathbf{y})\mathrm{d}\mathbf{y}
        \leq B \|u\|_{\mathcal{B}}^2,
        \qquad\forall u\in\mathcal{B},
    \end{equation}
    for some constants \(A\) and \(B\)
    is a consequence of the boundedness and positivity of the frame operator
    \(S\) which was established in Corollary~\ref{cor:SG_SPD}.
    Indeed, for any \(u\in\mathcal{B}\), we have
    \begin{equation}
        \int_{Y} |\left(u,\, \mathrm{EW}_{\mathbf{y}}\right)_{\mathcal{B}}|^2
        \; w_{}^{2}(\mathbf{y})\mathrm{d}\mathbf{y}
        = \left(Su,\, u\right)_{\mathcal{B}}
        = \sum_{p\in\mathbb{Z}} |\tau_{p}|^{2}
        |\left(u,\, b_{p}\right)_{\mathcal{B}}|^2,
    \end{equation}
    which also establishes the optimality of the claimed frame bounds.
\end{proof}

\subsection{The reproducing kernel property}

The continuous frame result implies additional structure on the Herglotz
density space \(\mathcal{A}\), which then allows to characterize the preimages
of the EPWs under the integral transform \(T\).
For a general reference on Reproducing Kernel Hilbert Spaces (RKHS), we refer
to\ \cite{Paulsen2016}.

\begin{lemma}\label{lem:rkhs}
    The range of the analysis operator \(T^{*}\),
    i.e.\ the space $\mathcal A$ defined in \eqref{eq:atildep},
    has the reproducing kernel property.
    The reproducing kernel is given by
    \begin{equation}\label{eq:kernel_expansion}
        K(\mathbf{z},\mathbf{y})
        = K_{\mathbf{y}}(\mathbf{z})
        = \left(K_{\mathbf{y}},\, K_{\mathbf{z}}\right)_{\mathcal{A}}
        = \sum_{p\in\mathbb{Z}}
        \overline{a_{p}(\mathbf{y})}
        a_{p}(\mathbf{z}),
        \qquad\forall \mathbf{y}, \mathbf{z}\in Y,
    \end{equation}
    with pointwise convergence of the series and
    where  \(K_{\mathbf{y}}\in\mathcal{A}\) is the (unique) Riesz
    representation of the evaluation functional at \(\mathbf{y} \in Y\), namely
    \begin{equation}
        v(\mathbf{y}) = \left(v, K_{\mathbf{y}}\right)_{\mathcal{A}},
        \qquad\forall v\in\mathcal{A}.
    \end{equation}
\end{lemma}
\begin{proof}
    Take any \(v\in\mathcal{A}\) and let \(u\in\mathcal{B}\) such that
    \(v=T^{*}u\),
    which exists thanks to Lemma~\ref{lem:Tstar_op}.
    From Corollary~\ref{cor:SG_SPD}, we have
    \begin{equation}
        u = S^{-1}S u
        = \int_{Y} \left(u,\, \mathrm{EW}_{\mathbf{z}}\right)_{\mathcal{B}}
        S^{-1}\mathrm{EW}_{\mathbf{z}}
        \; w_{}^{2}(\mathbf{z})\mathrm{d}\mathbf{z}.
    \end{equation}
    Then we obtain the reproducing identity,
    for any \(\mathbf{y}\in Y\)
    \begin{equation}
        \begin{aligned}
            v(\mathbf{y}) &= (T^{*}u)(\mathbf{y}) = \left(u,\, \mathrm{EW}_{\mathbf{y}}\right)_{\mathcal{B}} 
            = \int_{Y} \left(u,\, \mathrm{EW}_{\mathbf{z}}\right)_{\mathcal{B}}
            \left(S^{-1}\mathrm{EW}_{\mathbf{z}},\, \mathrm{EW}_{\mathbf{y}}\right)_{\mathcal{B}} 
            \; w_{}^{2}(\mathbf{z})\mathrm{d}\mathbf{z}\\
            &= \int_{Y} v(\mathbf{z})
            \left(S^{-1}\mathrm{EW}_{\mathbf{z}},\, \mathrm{EW}_{\mathbf{y}}\right)_{\mathcal{B}} 
            \; w_{}^{2}(\mathbf{z})\mathrm{d}\mathbf{z}
            = \left( v,\, K_{\mathbf{y}}\right)_{\mathcal{A}},
        \end{aligned}
    \end{equation}
    where we introduced (the Riesz representation of) the evaluation functional
    at the point \(\mathbf{y}\) defined as
    $   K_{\mathbf{y}}(\mathbf{z})
        := \left(\mathrm{EW}_{\mathbf{y}},\, S^{-1}\mathrm{EW}_{\mathbf{z}}\right)_{\mathcal{B}} 
        $, $\forall\mathbf{z}\in Y.
    $ 
    It is a direct consequence of Corollary~\ref{cor:SG_SPD}
    that the kernel admits the series
    representation~\eqref{eq:kernel_expansion}.
    Alternatively, we refer to\ \cite[Th.~2.4]{Paulsen2016} for a direct proof
    of this result (valid in the general setting), since
    \(\{a_{p}\}_{p\in\mathbb{Z}}\) is an orthonormal basis for \(\mathcal{A}\).
\end{proof}

Lemma~\ref{lem:rkhs} does not stem from any specific property of
\(\mathcal{A}\) or the EPWs, it follows only from the continuous frame result.
The reproducing kernel property implies that pointwise
evaluation of elements of \(\mathcal{A}\) in the cylinder \(Y\) is a
continuous operation\ \cite[Def.~1.2]{Paulsen2016}:
for all \(\mathbf{y}\in Y\) there is \(c>0\) such that
\begin{equation}
    |v(\mathbf{y})| = |\left(v, K_{\mathbf{y}}\right)_{\mathcal{A}}|
    \leq c \|v\|_{\mathcal{A}},
    \qquad\forall v\in\mathcal{A}.
\end{equation}
Examples of (normalized) evaluation functionals are given in Figure~\ref{fig:wKy_k16}.

\begin{figure}
    \centering
    \includegraphics[width=0.45\textwidth]{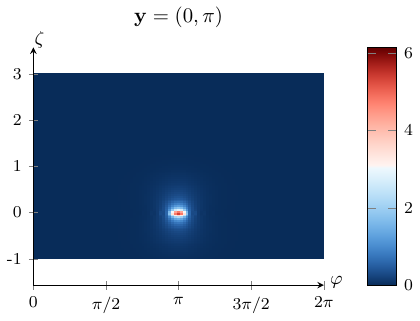}
    \includegraphics[width=0.45\textwidth]{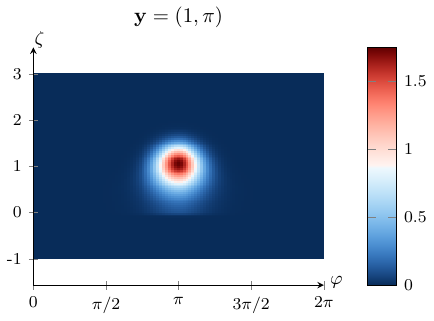}
    \caption{Representation of normalized evaluation functionals
    \(|w K_{\mathbf{y}}| / \|K_{\mathbf{y}}\|_{\mathcal{A}}\)
    in the cylinder \(Y\)
    for wavenumber \(\kappa=16\).}\label{fig:wKy_k16}
\end{figure}

The interest in introducing the reproducing kernel property stems from the
following result, which is a direct consequence of Lemma~\ref{lem:rkhs},
Theorem~\ref{th:T_op},
and the Jacobi--Anger identity~\eqref{eq:Jacobi--Anger-apbp}.

\begin{corollary}\label{cor:relation_between_K_and_epw}
    The EPWs are the images under $T$ of the Riesz
    representation of the evaluation functionals, namely
    \begin{equation}
        \mathrm{EW}_{\mathbf{y}} = TK_{\mathbf{y}},
        \qquad\forall \mathbf{y}\in Y.
    \end{equation}
\end{corollary}

As a consequence, the construction of an approximation of a Helmholtz
solution \(u\in\mathcal{B}\) as an expansion of EPWs is, up to the
isomorphism \(T\), equivalent to the approximation of its Herglotz density
\(v:=T^{-1}u\in\mathcal{A}\)
as an expansion of evaluation functionals, i.e.
\begin{equation}\label{eq:equiv-approx-pb}
    v \approx \sum_{m=1}^{M} \mu_{m} K_{\mathbf{y}_{m}}
    \qquad
    \substack{T\\\displaystyle\longrightarrow\\\displaystyle\longleftarrow\\T^{-1}}
    \qquad
    u \approx \sum_{m=1}^{M} \mu_{m} \mathrm{EW}_{\mathbf{y}_{m}},
\end{equation}
for some set of coefficients \(\boldsymbol{\mu} = \{\mu_{m}\}_{m=1}^{M}\).
This remark justifies the use of
the sampling techniques described in the next
section to discretize the integral representation in~\eqref{eq:Herglotz-transform}.
Section~\ref{sec:numerics} provides numerical evidence that 
such approximations can be built, for a suitable normalization of the
sets
\(\{K_{\mathbf{y}_{m}}\}_{m}\) and \(\{\mathrm{EW}_{\mathbf{y}_{m}}\}_{m}\).

\section{A concrete evanescent plane wave approximation set}\label{sec:discrete-recipe}

We describe a method for the numerical approximation of a general
Helmholtz solution in the unit disk by EPWs.
We exploit the equivalence of this approximation problem with the approximation
problem of the corresponding Herglotz density, see~\eqref{eq:equiv-approx-pb}.
The main idea is to adapt the sampling procedure
of\ \cite{Hampton2015,Cohen2017,Migliorati2022}
(sometimes called \emph{coherence-optimal sampling})
to our case, in order to generate a distribution of
sampling nodes in the cylinder \(Y\) that will be used to reconstruct
the Herglotz density.
While the numerical recipe that we describe below is found to be numerically
very effective, see Section~\ref{sec:numerics}, our theoretical analysis
still lacks a formal proof of the accuracy and stability of the approximation
of Helmholtz solutions using EPWs.

Let \(u \in \mathcal{B}\) be the Helmholtz solution, target of the
approximation problem, and let \(v := T^{-1}u \in \mathcal{A}\) be its
associated Herglotz density.
Let also some tolerance \(\eta>0\) be given.

\subsection{Truncation of the modal expansion}

Since \(u\) (resp.\ \(v\)) \textit{a priori} lives in an infinite dimensional
space \(\mathcal{B}\) (resp.\ \(\mathcal{A}\)),
the idea behind the construction of finite dimensional approximation sets is to
exploit the natural hierarchy of finite dimensional subspaces constructed by
truncation of the Hilbert basis \(\left\{{b}_{p}\right\}_{p\in\mathbb{Z}}\).

\paragraph{Truncation in the Helmholtz solution space.}

For any \(P \in \mathbb{N}\), we define
\begin{equation}
    \mathcal{B}_{P} := \operatorname{span}
    \left\{{b}_{p}\right\}_{|p| \leq P}
    \subset \mathcal{B},
    \qquad\text{and}\qquad
    \Pi_{P}
    \;:\; \mathcal{B}\to \mathcal{B},\ 
    u\mapsto
    \sum_{|p|\leq P} \left(u,\, b_{p}\right)_{\mathcal{B}} b_{p}.
\end{equation}
Here \(\Pi_{P}\) is the orthogonal projection from \(\mathcal{B}\)
onto the finite dimensional subspace \(\mathcal{B}_{P}\).
A natural approach to compute an approximation of \(u \in \mathcal{B}\) is to
approximate its projection onto \(\mathcal{B}_{P}\), namely
\begin{equation}
    u_{P} := \Pi_{P}u \in \mathcal{B}_{P},
    \qquad\forall P \in \mathbb{N},
\end{equation}
for some \(P\) large enough.
It is immediate that the sequence of projections
\(\{u_{P}\}_{P\in\mathbb{N}}\) converges to \(u\) in \(\mathcal{B}\).
In particular, we can define for any \(\eta > 0\)
\begin{equation}\label{eq:def-Pstar}
    P^{*} = P^{*}(u,\eta) := \min
    \big\{
        P \in \mathbb{N} \;|\;
        \|u - u_{P}\|_{\mathcal{B}} < \eta \|u\|_{\mathcal B}
    \big\}.
\end{equation}

Unfortunately, it is not possible to compute such a \(P^{*}\) in most practical
configurations.
It may be possible though to give estimates on \(P^{*}\), based on some
regularity assumption on \(u\) and the decay of its coefficients in its modal
expansion.
For instance, it might be physically realistic to assume that all coefficients
of the propagative modes \(|p| \leq \kappa\) are \(\mathcal{O}(1)\) and the
coefficients associated to the subsequent evanescent modes \(|p| \geq \kappa\)
decay in modulus with a given algebraic or exponential rate.

\paragraph{Truncation in the Herglotz density space.}

Similarly, for any \(P \in \mathbb{N}\), we define
\begin{equation}
    \mathcal{A}_{P} := \operatorname{span}
    \left\{{a}_{p}\right\}_{|p| \leq P}
    = T^{-1} \mathcal{B}_{P}
    \subset \mathcal{A},
    \qquad\text{and}\qquad
    v_{P} := T^{-1} u_{P} \in \mathcal{A}_{P},
    \qquad\forall P \in \mathbb{N}.
\end{equation}
Theorem~\ref{th:T_op} implies that the sequence \(\{v_{P}\}_{P\in\mathbb{N}}\)
converges to \(v\) in \(\mathcal{A}\).
In particular, for any \(P \geq P^{*}\),
where \(P^{*}\) was defined in~\eqref{eq:def-Pstar}, we have
\begin{equation}\label{eq:Fourier-truncation-vP}
    \|v - v_{P}\|_{\mathcal{A}}
    \leq \tau_{-}^{-1} \|u - u_{P}\|_{\mathcal{B}}
    < \tau_{-}^{-1}\eta \|u\|_{\mathcal B}.
\end{equation}

\subsection{Parameter sampling in the cylinder \texorpdfstring{$Y$}{Y}}

Our objective is to approximate
the truncated Fourier series
\(u_{P} = \Pi_{P} u \in \mathcal{B}_{P}\)
for some \(P \in \mathbb{N}\), instead of \(u\).
Up to the Herglotz transform, this problem is equivalent to the approximation
of \(v_{P} = T^{-1} u_{P} \in \mathcal{A}_{P}\).
In this subsection let us fix a \(P \in \mathbb{N}\),
not necessarily equal to \(P^{*}\).
We propose to build approximations of elements of \(\mathcal{A}_{P}\)
(resp.\ \(\mathcal{B}_{P}\))
by constructing a finite set of sampling nodes
\(\{\mathbf{y}_{m}\}_{m}\) in the cylinder \(Y\) according to
the distribution advocated in\ \cite[Sec.~2.1]{Hampton2015},
\cite[Sec.~2.2]{Cohen2017} and \cite[Sec.~2]{Migliorati2022}.
Despite having an unbounded parametric domain \(Y\), the finite
integrability of the weight function \(w^{2}\) allows to sample \(Y\) on a
bounded region only.
The associated set of sampling functionals \(\{K_{\mathbf{y}_{m}}\}_{m}\)
(up to some normalization factor)
is expected to provide a good approximation of \(v_{P}\).
The approximation set for \(u_{P}\) will then be given by the EPWs
\(\{\mathrm{EW}_{\mathbf{y}_{m}}\}_{m}\) (up to some normalization factor).

We denote the dimension of both spaces \(\mathcal{A}_{P}\) and
\(\mathcal{B}_{P}\) by
\begin{equation}
    N_{P} :=
    \dim \mathcal{B}_{P} = \dim \mathcal{A}_{P}
    = 2P+1.
\end{equation}
The probability density function \(\rho_{P}\) is defined (up to normalization)
as the reciprocal of the \(N_{P}\)-term \emph{Christoffel function}
\(\mu_{P}\) in the spirit of\ \cite[Eq.~(2.6)]{Cohen2017}:
\begin{equation}\label{eq:density-function}
    \rho_{P} := \frac{w^2}{N_{P} \mu_{P}},
    \quad\text{where}\quad
    \mu_{P}(\mathbf{y}) := \Big(
        \sum_{|p| \leq P} |a_{p}(\mathbf{y})|^{2}
    \Big)^{-1},
    \qquad \forall\mathbf{y}=(\zeta,\varphi)\in Y.
\end{equation}
Observe that \(\rho_{P}\) and \(\mu_{P}\) are well-defined since
\(0 < \mu_{P} \leq \mu_{0} < \infty\) from the fact that \(a_{0}\)
is just a non-vanishing constant.
The density function \(\rho_{P}\) is a univariate
function on $Y$ since it is independent of the angle \(\varphi\).
We point out that \(1/\mu_{P}\) corresponds to the truncated series expansion
of the diagonal of the reproducing kernel \(K\), which amounts to taking
\(\mathbf{z}=\mathbf{y}\) and truncating at \(P\) the series
in~\eqref{eq:kernel_expansion}.

The numerical recipe consists, for each \(P \in \mathbb{N}\),
in generating a sequence of sampling node sets in the parametric domain \(Y\)
\begin{equation}\label{eq:sampling-sets}
    \mathbb{Y}_{P} := \{\mathbb{Y}_{P,M}\}_{M \in \mathbb{N}},
    \qquad\text{where}\qquad
    \mathbb{Y}_{P,M} := \{\mathbf{y}_{m}\}_{m=1}^{M},
    \quad\forall M \in \mathbb{N},
\end{equation}
using one's preferred sampling strategy such that \(|\mathbb{Y}_{P,M}| = M\)
for all \(M \in \mathbb{N}\) and the sequence \(\mathbb{Y}_{P}\) converges (in
a suitable sense) to the density \(\rho_{P}\) defined
in~\eqref{eq:density-function} as \(M\) tends to infinity.
The sampling method could be a deterministic, a random or even a quasi-random
strategy, see Section~\ref{sec:numerics}.
The sets are not assumed to be nested.

This choice of EPW parameters
is a major difference from the heuristic choice described in~\cite[Eq.~(5)]{Huybrechs2019} 
where the parameters are chosen in order to approximate solutions
defined in a rectangle containing the physical domain of interest (\(B_{1}\) in
our case).

\subsection{Evanescent plane wave approximation sets}\label{sec:conj-stability}

From the sampling node sets~\eqref{eq:sampling-sets} we can construct two
approximations sets:
one set of sampling functionals in \(\mathcal{A}\)
and
one set of EPWs in \(\mathcal{B}\).

\paragraph{Approximation sets in the Herglotz density space.}

Associated to the sampling node sets~\eqref{eq:sampling-sets},
we introduce a sequence of finite sets in
\(\mathcal{A}\)
\begin{equation}\label{eq:Aspace-approx-sets}
    \boldsymbol{\Psi}_{P} := \{\boldsymbol{\Psi}_{P,M}\}_{M \in \mathbb{N}}
    \quad\text{where}\quad
    \boldsymbol{\Psi}_{P,M} := \left\{
        \sqrt{\frac{\mu_{P}(\mathbf{y}_{m})}{M}}
        K_{\mathbf{y}_{m}},
    \right\}_{\mathbf{y}_{m} \in \mathbb{Y}_{P,M}}
    \quad \forall M \in \mathbb{N}.
\end{equation}
The normalization of \(K_{\mathbf{y}_{m}}\) in~\eqref{eq:Aspace-approx-sets} is
crucial for the stable approximation property~\eqref{eq:stability}.
In the approximation sets, each sampling functional \(K_{\mathbf{y}_{m}}\)
has been normalized by the real constant
\(\sqrt{\mu_{P}(\mathbf{y}_{m})/M}\)
which is (numerically) close to
\(\|K_{\mathbf{y}_{m}}\|_{\mathcal{A}}^{-1}/\sqrt{M}\).
More precisely, we have
\begin{equation}
    {\sqrt{\mu_{P}(\mathbf{y})}} \; {\|K_{\mathbf{y}}\|_{\mathcal{A}}}
    = \Big(
        \sum_{|p| \leq P} |a_{p}(\mathbf{y})|^{2}
    \Big)^{-1/2}
    \Big(
        \sum_{p \in \mathbb{Z}} |a_{p}(\mathbf{y})|^{2}
    \Big)^{1/2}
    \geq 1
    \qquad \forall\mathbf{y}\in Y.
\end{equation}

\paragraph{Approximation sets in the Helmholtz solution space.}

Associated to the sampling set sequences~\eqref{eq:sampling-sets} and
approximation set sequences~\eqref{eq:Aspace-approx-sets} in \(\mathcal{A}\),
we define the sequence of approximation sets of (normalized) EPWs in
\(\mathcal{B}\) as follows
\begin{equation}\label{eq:evanescent-pw-sets}
    \boldsymbol{\Phi} := \{\boldsymbol{\Phi}_{P,M}\}_{P \in \mathbb{N}, M \in \mathbb{N}},
    \qquad
    \boldsymbol{\Phi}_{P,M} :=
    \left\{
        \sqrt{\frac{\mu_{P}(\mathbf{y}_{m})}{M}}
        \mathrm{EW}_{\mathbf{y}_{m}}
    \right\}_{\mathbf{y}_{m} \in \mathbb{Y}_{P,M}}
    \quad\forall P \in \mathbb{N}, M \in \mathbb{N}.
\end{equation}
Following Corollary~\ref{cor:relation_between_K_and_epw}, 
the sequence of sets~\eqref{eq:evanescent-pw-sets}
is the image of the sequence of sets~\eqref{eq:Aspace-approx-sets}
by the Herglotz transform operator \(T\).

\paragraph{Discussion on the parameters.}

Our numerical recipe for building the approximation sets
\(\boldsymbol{\Phi}_{P,M}\) is based on only two parameters, \(P\) and \(M\),
whose tuning is intuitive:
\begin{enumerate}
    \item The first one is the Fourier truncation parameter \(P\).
        Increasing \(P\) will improve the accuracy of the approximation of
        \(u\) (resp.\ \(v = T^{-1}u\)) by \(u_{P} = \Pi_{P}u\)
        (resp.\ \(v_P = T^{-1}u_{P}\)).
        The appropriate value for \(P \geq P^{*}\) will solely depend on
        the decay of the coefficients in the modal expansion, which is
        intimately linked to the regularity of the Helmholtz solution.
    \item The second one is the dimension \(M\) of the EPW
        approximation space, which is also the number of sampling points in the
        parameter cylinder \(Y\).
        For a fixed \(P\), increasing \(M\) should allow to control
        the accuracy of the approximation of
        \(u_{P}\) (resp.\ \(v_{P} = T^{-1}u_{P}\))
        by \(\mathcal{T}_{\boldsymbol{\Phi}_{P,M}}{\boldsymbol{\xi}}\)
        (resp.\ \(\mathcal{T}_{\boldsymbol{\Psi}_{P,M}}{\boldsymbol{\xi}}\))
        for some bounded coefficients \(\boldsymbol{\xi} \in \mathbb{C}^{M}\).
        The numerical results presented below corroborate this conjecture and
        show experimentally that \(M\) should scale linearly with \(P\), with a
        moderate proportionality constant (see
        Section~\ref{sec:numerics-quasi-optimality}).
\end{enumerate}
For a fixed DOF budget \(M\), the numerical experiments in
Section~\ref{sec:numerics-quasi-optimality} suggests that using 
a Fourier truncation parameter
\(P = \max\left(\lceil\kappa\rceil, \lfloor M/4\rfloor\right)\)
gives accurate and reliable approximations.

Once the approximation sets \(\boldsymbol{\Phi}_{P,M}\)
are chosen, our concrete implementation (see
Section~\ref{sec:regularization}) to compute a particular set of coefficients
\(\boldsymbol{\xi}_{S,\epsilon}\) includes two additional parameters,
\(S\) and \(\epsilon\):
\begin{enumerate}
    \item The first parameter \(S\) is the number of sampling points on the
        boundary of the physical domain \(B_1\).
        According to~\eqref{eq:proof-approx-parseval-frame}
        and following\ \cite{Adcock2019,Adcock2020}, sufficient oversampling
        should be used.
        In practice, we chose for simplicity an oversampling ratio of \(2\),
        namely \(S = 2M\).
        This amount of oversampling may not be necessary and further
        numerical experiments could investigate a reduction of the oversampling
        ratio \(S/M\) to reduce the computational cost.
    \item The second parameter \(\epsilon\) is the regularization parameter,
        i.e.~the truncation threshold of the singular values.
        We set this parameter to \(\epsilon = 10^{-14}\) in
        the numerical experiments presented below.
        If one is interested in less accurate approximations than ours,
        this parameter could be set to larger values.
\end{enumerate}
We stress that the construction of the approximation sets
\(\boldsymbol{\Phi}_{P,M}\), together with their accuracy and stability, are
not influenced by the choice of the
reconstruction strategy made in Section~\ref{sec:Sampling}.
Although we focus on the simple method of boundary sampling together
with regularized SVD, alternative reconstruction strategies (such as sampling
in the bulk of the domain or taking inner product with elements of other types
of test spaces, for instance) and other regularization techniques (such as
Tikhonov regularization) can also be successfully used in practice.
Irrespective of the strategy, sufficient oversampling and regularization
need to be used.

\paragraph{Relation with the literature.}

As we have already alluded to, our construction is based on similar ideas
that pre-exist in the literature but in a different context.
Indeed, sampling node sets similar to the ones we propose
here can be found in\ \cite{Hampton2015,Cohen2017,Migliorati2022}.
The context of these works is the reconstruction of elements of
finite-dimensional subspaces (with explicit orthonormal basis) in weighted
\(L^2\) spaces from sampling\ \cite{Cohen2017} and it was subsequently used to
construct random cubature rules\ \cite{Migliorati2022}.
The underlying idea is that the information gathered from sampling at these
nodes is enough to allow accurate reconstruction as an expansion in the
(truncated) orthonormal basis.

Translated into our setting, the results available in the literature
say that to reconstruct an element \(v_{P} = \Pi_{P} v\)
of the finite dimensional subspace \(\mathcal{A}_{P}\),
it is enough to sample at the nodes \(\boldsymbol{\Psi}_{P,M}\) for
some sufficiently large \(M\).
In contrast, the numerical recipe described above seeks to construct an
approximation of the element \(v_{P} = \Pi_{P} v\in\mathcal{A}_{P}\) as an
expansion in the set of evaluation functionals \(\boldsymbol{\Psi}_{P,M}\) for
some sufficiently large \(M\).
In other words, the approximation we are looking for belongs to the span of the
evaluation functionals, \(\operatorname{span}\boldsymbol{\Psi}_{P,M}\),
which has trivial intersection with \(\mathcal{A}_{P}\).
By Corollary~\ref{cor:relation_between_K_and_epw}, applying the Herglotz
transform \(T\) to this approximation in
\(\operatorname{span}\boldsymbol{\Psi}_{P,M}\)
yields an element in \(\operatorname{span}\boldsymbol{\Phi}_{P,M}\)
(i.e.\ a finite superposition of EPWs)
that approximates
\(u_{P} = Tv_{P} \in\mathcal{B}_{P}\).

Unfortunately, besides the links with these works, we are not yet able to prove a
rigorous theoretical analysis to support our numerical recipe.
Yet, extensive numerical experiments in Section~\ref{sec:numerics} illustrate
the excellent approximation and stability properties of the sets
\(\boldsymbol{\Phi}_{P,M}\).

\subsection{A conjectural stable approximation result}\label{sec:conj}

We formalize below our speculations, which are hinted by the numerical
experiments given in the next section.
First, we state our main conjecture.

\begin{conjecture}\label{conj:approx-conjecture}
    The sequence of approximation sets \(\boldsymbol{\Psi}_{P}\)
    defined in~\eqref{eq:Aspace-approx-sets}
    is a stable approximation for \(\mathcal{A}_{P}\),
    in the following sense:
    there exist \(s\geq 0\) and \(C > 0\) such that, for all
    \(P\in\mathbb{N}\), there exists \(M^{*} = M(P, \eta)\) such that
    \begin{equation}\label{eq:conjecture}
        \forall v_{P} \in \mathcal{A}_{P},\ 
        \exists M \in \mathbb{N},\ 
        \boldsymbol{\mu} \in \mathbb{C}^{M},
        \quad
        \|v_{P} - \mathcal{T}_{\boldsymbol{\Psi}_{P,M}}{\boldsymbol{\mu}}\|_{\mathcal{A}}
        \leq \eta \|v_P\|_{\mathcal A}
        \ \ \text{and}\ \ 
        \|\boldsymbol{\mu}\|_{\ell^{2}} \leq C M^{s} \|v_P\|_{\mathcal A}.
    \end{equation}    
\end{conjecture}

In the following we assume for simplicity that all \(M \geq M^{*}\) satisfy the
two inequalities appearing in~\eqref{eq:conjecture} (otherwise the proofs can be
easily adapted).
This holds true if the sets are hierarchical, for instance, but this is not
necessary.

Provided the above conjecture holds, the stability of the approximation sets of
EPWs constructed above would follow as we prove next.

\begin{proposition}\label{prop:stability-epw}
    Let \(\delta > 0\).
    If Conjecture~\ref{conj:approx-conjecture} holds
    then the sequence of approximation sets~\eqref{eq:evanescent-pw-sets}
    provides a stable approximation for \(\mathcal{B}\).
    In particular,
    if $\kappa^2$ is not a Dirichlet eigenvalue on $B_1$,
    \begin{equation}
        \forall u \in \mathcal{B} \cap C^{0}(\overline{B_{1}}),\ 
        \exists P \in \mathbb{N},\ M \in \mathbb{N},\ 
        S \in \mathbb{N},\ \epsilon \in (0,1],
        \qquad
        \|u - \mathcal{T}_{\boldsymbol{\Phi}_{P,M}}\boldsymbol{\xi}_{S,\epsilon}\|_{L^{2}(B_{1})}
        \leq \delta
        \|u\|_{\mathcal{B}},
    \end{equation}
    where
    \(\boldsymbol{\xi}_{S,\epsilon} \in \mathbb{C}^{|\boldsymbol{\Phi}_{P,M}|}\)
    is computed with the regularization procedure
    in~\eqref{eq:solution_SVDr}.
    The SVD regularization parameter \(\epsilon\) can be chosen
    as~\eqref{eq:epsilon-estimate}.
\end{proposition}
\begin{proof}
    We need to prove the stability of the sequence of approximation sets,
    namely that for any \(\tilde\eta>0\), there exists \(\tilde{s} \geq 0\) and
    \(\tilde{C} > 0\)
    such that
    \begin{equation}\label{eq:proof-stability}
        \forall u \in \mathcal{B},\ 
        \exists P \in \mathbb{N},\ 
        M \in \mathbb{N},\ 
        \boldsymbol{\mu} \in \mathbb{C}^{M},
        \quad 
        \|u - \mathcal{T}_{\boldsymbol{\Phi}_{P,M}}{\boldsymbol{\mu}}\|_{\mathcal{B}}
        \leq \tilde\eta\|u\|_{\mathcal B}
        \ \ \text{and} \ \ 
        \|\boldsymbol{\mu}\|_{\ell^{2}} \leq \tilde{C} M^{\tilde{s}}\|u\|_{\mathcal B}.
    \end{equation}    
    Provided this holds, the claimed result is a direct application of
    Corollary~\ref{cor:approx-error-estimate}.
    
    Let \(\eta>0\),
    \(u \in \mathcal{B}\) and set \(v := T^{-1}u \in \mathcal{A}\).
    For any \(P \geq P^{*}=P^{*}(u,\eta)\)
    with \(P^{*}\) defined in~\eqref{eq:def-Pstar}, 
    if we let \(u_{P} := \Pi_{P}u\) and \(v_{P} := T^{-1}u_{P}\) we have
    (recall~\eqref{eq:Fourier-truncation-vP})
    \begin{equation}
        \|u - u_{P}\|_{\mathcal{B}} \leq \eta \|u\|_{\mathcal B},
        \qquad\text{and}\qquad
        \|v - v_{P}\|_{\mathcal{A}} \leq \tau_{-}^{-1} \eta \|u\|_{\mathcal B}.
    \end{equation}
    Assuming that Conjecture~\ref{conj:approx-conjecture} holds,
    there exist \(s\) and \(C\) (both independent of \(P\))
    such that, for any \(M \geq M^{*}(P^{*},\eta)\), 
    there exists a set of coefficients \(\boldsymbol{\mu} \in \mathbb{C}^{M}\)
    such that
    \begin{equation}
        \|v_{P} - \mathcal{T}_{\boldsymbol{\Psi}_{P,M}}\boldsymbol{\mu}\|_{\mathcal{A}}
        \leq \eta \|v_P\|_{\mathcal A},
        \qquad\text{and}\qquad 
        \|\boldsymbol{\mu}\|_{\ell^{2}} \leq C M^{s} \|v_P\|_{\mathcal A}.
    \end{equation}
    The properties of the isomorphism \(T\) given in Theorem~\ref{th:T_op}
    imply that
    \begin{equation}
        \|u_{P} - \mathcal{T}_{\boldsymbol{\Phi}_{P,M}}\boldsymbol{\mu}\|_{\mathcal{B}}
         < \tau_{+} \eta \|v_P\|_{\mathcal A}
        \qquad \text{and}\qquad
        \|v_P\|_{\mathcal A}
        \le \tau_-^{-1}\|u_P\|_{\mathcal B}
        \le \tau_-^{-1}\|u\|_{\mathcal B}.
    \end{equation}
    For any \(P \geq P^{*}(u,\eta)\) and \(M \geq M^{*}(P^{*},\eta)\),
    the total approximation error 
    for the Herglotz density \(v\) can be estimated as
    \begin{equation}\label{eq:error-estimate-v}
        \|v - \mathcal{T}_{\boldsymbol{\Psi}_{P,M}}\boldsymbol{\mu}\|_{\mathcal{A}}
        \leq \|v - v_{P}\|_{\mathcal{A}} +
        \|v_{P} - \mathcal{T}_{\boldsymbol{\Psi}_{P,M}}\boldsymbol{\mu}\|_{\mathcal{A}}
        \leq 2\tau_-^{-1} \eta \|u\|_{\mathcal B},
    \end{equation}
    and for the Helmholtz solution \(u\) as
    \begin{equation}\label{eq:error-estimate-u}
    \begin{aligned}
        \|u - \mathcal{T}_{\boldsymbol{\Phi}_{P,M}}\boldsymbol{\mu}\|_{\mathcal{B}} 
       &\leq \|u - u_{P}\|_{\mathcal{B}} +
        \|u_{P} - \mathcal{T}_{\boldsymbol{\Phi}_{P,M}}\boldsymbol{\mu}\|_{\mathcal{B}}\\
       &\leq \left(1 + \tau_{+}\tau_-^{-1}\right) \eta \|u\|_{\mathcal B},
    \end{aligned}
   \qquad\text{and}\qquad
        \|\boldsymbol{\mu}\|_{\ell^{2}} \leq C M^{s}
        \tau_-^{-1}\|u\|_{\mathcal B}.
    \end{equation}
    Choosing \(\eta=\tilde\eta/(1+\tau_+\tau_-^{-1})\), we can conclude 
    since~\eqref{eq:error-estimate-u} is~\eqref{eq:proof-stability} with
    $\tilde{s}=s$ and $\tilde{C}=C\tau_-^{-1}$.
\end{proof}

\section{Numerical results}\label{sec:numerics}

We provide numerical evidence that the procedure described above
allows to compute controllably accurate approximations of Helmholtz solutions
in the unit disk and in other domains\footnote{The \textsc{Julia} code used to generate the
numerical results of this paper is available at\\
\href
{https://github.com/EmileParolin/evanescent-plane-wave-approx}
{https://github.com/EmileParolin/evanescent-plane-wave-approx}}.

\subsection{Probability densities and samples}

\paragraph{Probability density and cumulative distributions functions.}

We represent the probability density function \(\rho_{P}\)
(see~\eqref{eq:density-function})
as a function of the evanescence parameter \(\zeta\)
on the left in Figure~\ref{fig:sampling_density}.
Here \(P\) denotes the truncation parameter, meaning that the sampling is
performed to approximate elements of \(\mathcal{A}_{P}\),
which has dimension \(N_{P}\).
The associated cumulative distribution function with respect to the
evanescence parameter \(\zeta\) is defined as
\begin{equation}
    \Upsilon_{P}(\zeta) := \int_{-\infty}^{\zeta}
    \rho_{P}(\tilde{\zeta})
    \; \mathrm{d}\tilde{\zeta},
    \qquad
    \forall \zeta \in \mathbb{R}.
\end{equation}
It is represented in the right of Figure~\ref{fig:sampling_density}.
Recall that while \(\rho_{P}\) is a bi-variate function on the cylinder $Y$, it
is constant with respect to the angle \(\varphi\).
As a result, the cumulative distribution with respect to this variable
\(\varphi\) is a linear function.
This is why we represent these two functions \(\rho_{P}\) and
\(\Upsilon_{P}\) only with respect to the evanescence parameter \(\zeta\).

\begin{figure}
    \centering
    \includegraphics[width=0.54\textwidth]{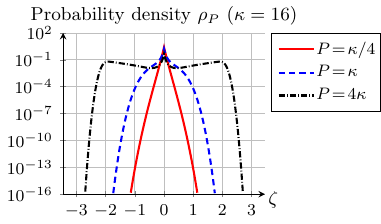}\hspace{0.01\textwidth}
    \includegraphics[width=0.41\textwidth]{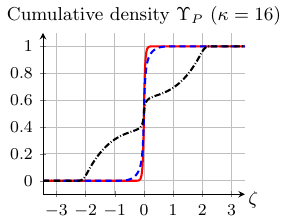}\vspace{0.02\textwidth}
    \caption{Sampling density functions \(\rho_{P}\) (left) and
    \(\Upsilon_{P}\) (right)
    with respect to the evanescence parameter \(\zeta\)
    constructed for the subspace \(\mathcal{A}_{P}\).
    Wavenumber \(\kappa=16\).
    }\label{fig:sampling_density}
\end{figure}

We observe that the probability density \(\rho_{P}\) is a symmetric even
function and exhibits a main mode at \(\zeta=0\) which corresponds to purely
PPWs.
Moreover, the \(\epsilon\)-support of this density is rather tight and the
probability eventually tends to zero exponentially as \(|\zeta|\) gets large
enough.
When \(P \leq \kappa\) the density is a unimodal distribution
whereas for \(P \gg \kappa\) (e.g.~$P=4\kappa$) the density is a multimodal
distribution.
Indeed, in the latter case, there are two symmetric modes for relatively large
evanescence parameter, in addition to the main mode at \(\zeta=0\).
The cumulative distribution function \(\Upsilon_{P}\) is close to a step
function in the case where \(\mathcal{A}_{P}\) contains only elements
associated to the propagative regime \(P \leq \kappa\).
In contrast, for \(P > \kappa\) the distribution is non-trivial
for moderate values of the evanescence parameter \(\zeta\).
This means that for \(P \leq \kappa\) one can safely choose only PPWs, while
for \(P>\kappa\) EPWs are needed and their choice is non-trivial.

\paragraph{Parameter sampling.}

For any \(P\) we generate \(M = \nu N_{P}\) samples in the cylinder \(Y\) using
the technique called
\emph{Inversion Transform Sampling} (ITS)\ \cite[Sec.~5.2]{Cohen2017}.
It consists in first generating sampling sets
in the unit square \([0,1]^{2}\) that converge (in a suitable sense) to the
uniform distribution \(\mathcal{U}_{[0,1]^{2}}\) when \(M\to\infty\),
\begin{equation}
    \{\mathbf{z}_{m}\}_{m},
    \qquad\text{with}\qquad
    \mathbf{z}_{m} =
    (z_{m,\varphi}, z_{m,\zeta}) \in [0,1]^{2}, \quad m = 1,\ldots,M,
\end{equation}
and then map back to the cylinder \(Y\), to obtain
sampling sets that converge
to the probability density function \(\rho_{P}\) when \(M\to\infty\), namely
\begin{equation}\label{eq:ymSamples}
    \{\mathbf{y}_{m}\}_{m},
    \qquad\text{with}\qquad
    \mathbf{y}_{m} :=
    \left(2\pi z_{m,\varphi} ,\Upsilon_{P}^{-1}(z_{m,\zeta})\right)
    \in Y, \quad m = 1,\ldots,M.
\end{equation}
The fact that the density function is constant with respect to
\(\varphi\) considerably simplifies the generation of the samples. 
The inversion \(\Upsilon_{P}^{-1}\) can be performed using elementary
root-finding techniques,
our implementation resorts to the bisection method. 

In our numerical experiments we tested three types of sampling methods, which
differ by how we generate the first sampling distribution
\(\{\mathbf{z}_{m}\}_{m}\) in the unit square:
\begin{enumerate}
    \item \emph{deterministic} sampling: the initial samples in the unit square
        are a Cartesian product of two sets of equispaced points with the same
        number of points in both directions (all numerical results
        presented are obtained by using as approximation set dimension
        the smallest square integer larger than or equal to \(M\));
    \item \emph{Sobol} sampling: the initial samples in the unit square
        corresponds to Sobol sequences which are quasi-random low-discrepancy
        sequences\footnote{We used the \textsc{Julia} packages
        \texttt{Sobol.jl} and \texttt{QuasiMonteCarlo.jl},
        which are themselves based on\ \cite{Bratley1988,Joe2003}.};
    \item \emph{random} sampling: the initial samples in the unit square are
        drawn randomly according to the product of two uniform distributions
        \(\mathcal{U}_{[0,1]}\).
\end{enumerate}

Some examples of sampling sets corresponding to the probability density function
\(\rho_{P}\) for \(\kappa=16\) are reported in
Figure~\ref{fig:sampling_samples}.
For these examples the number of sampling nodes is set to \(M = \nu N_{P}\) with
\(\nu=4\), for the three types of sampling considered.
As expected, the sampling points cluster near the line \(\zeta=0\) for smaller
$P$.
This is the (propagative) regime for which PPWs alone provide a good
approximation.
When \(P>\kappa\) the evanescence parameter \(\zeta\) spreads in a wider
domain, with some clustering at the secondary modes of the distribution, in
agreement with Figure~\ref{fig:sampling_density}.

\begin{figure}
    \centering
    \raisebox{-.5\height}{\includegraphics[width=0.3\textwidth]{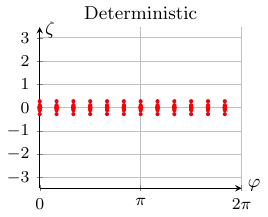}}
    \raisebox{-.5\height}{\includegraphics[width=0.3\textwidth]{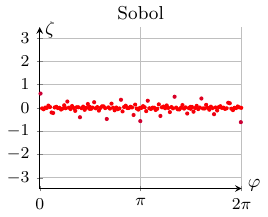}}
    \raisebox{-.5\height}{\includegraphics[width=0.3\textwidth]{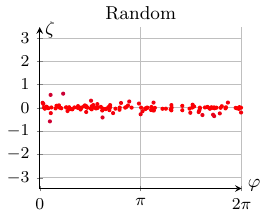}}
    \(P=\kappa\ \)
    \raisebox{-.5\height}{\includegraphics[width=0.3\textwidth]{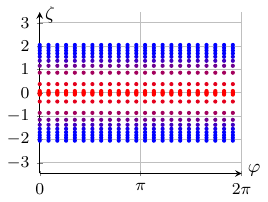}}
    \raisebox{-.5\height}{\includegraphics[width=0.3\textwidth]{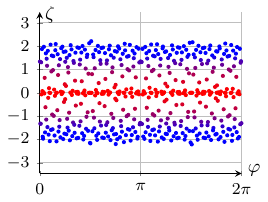}}
    \raisebox{-.5\height}{\includegraphics[width=0.3\textwidth]{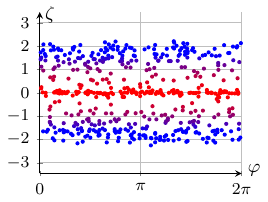}}
    \(P=4\kappa\)
    \caption{\(M = 4 N_{P}\) samples in the cylinder \(Y\)
    for \(P=\kappa\) (top) and \(P=4\kappa\) (bottom) and 
    various types of sampling method (left to right).
    Wavenumber \(\kappa=16\).
    Large \(|\zeta|\) implies fast EPW decay.
    }\label{fig:sampling_samples}
\end{figure}

\subsection{Propagative plane waves are unstable}\label{sec:PWnumerics}

Before presenting EPW approximations, we report some numerical experiments
dedicated to verifying numerically the instability result of
Lemma~\ref{lem:instability-propagative} when using PPWs.
These will also serve as a reference point to compare with the results obtained
using our EPW recipe.

Let us consider the approximation problem of
Section~\ref{sec:instability}, namely the approximation of the circular wave
\(b_{p}\) for some \(p\in\mathbb{Z}\) by an approximation set
\(\boldsymbol{\Phi}_{M}\) of \(M \in \mathbb{N}\) PPWs
defined in~\eqref{eq:propagative-pw-sets}.
The sampling matrix \(A\) was defined in~\eqref{eq:def_matrix_rhs}, using \(M\)
PPWs with equispaced angles and \(S:=\max(2M, 2|p|)\)
sampling points (we impose \(S \geq 2|p|\) to avoid spurious results due to
aliasing).
The entries of the matrix $A$ are immediately computed as
$A_{s,m}=e^{\imath\kappa\cos(2\pi(\frac sS-\frac mM))}$ for $s=1,\ldots,S$,
$m=1,\ldots M$.
The right-hand side \(\mathbf{b}\) is defined as in~\eqref{eq:def_matrix_rhs}
for \(b_{p}\) in place of \(u\); we recall that we use Dirichlet data in all
our numerical experiments.

The matrix \(A\) is notoriously ill-conditioned (see
Figure~\ref{fig:propagative_instability_k16_svs_scatter}):
its condition number grows exponentially
with respect to the number of plane waves \(M\) in the
approximation set \(\boldsymbol{\Phi}_{M}\).
This is well-known, see for instance the numerical experiments
in\ \cite[Sec.~2.3]{PerreyDebain2006} for the circular geometry and
\(S=M\).
This is not a feature of the sampling method: we refer to similar
experiments in\ \cite[Sec.~4.3]{Hiptmair2016} for the mass matrix of a
Galerkin formulation in a Cartesian geometry, again for \(S=M\).
The least-squares formulation suffers from an even worse
condition number: proportional to the square of the condition number of the
sampling method, see e.g.\ \cite[Eq.~(30)]{PerreyDebain2006}.
We apply the regularization procedure described in
Section~\ref{sec:regularization}
with threshold parameter \(\epsilon = 10^{-14}\).

The numerical results are reported in
Figure~\ref{fig:propagative_instability_k16}.
On the left panel we report the relative residual \(\mathcal{E}\)
defined in~\eqref{eq:relative-error} as a measure of the accuracy of the
approximation.
On the right panel we report the size of the coefficients
\(\|\boldsymbol{\xi}_{S,\epsilon}\|_{\ell^2}\) as a measure of the stability of
the approximation.
Relative residuals and coefficient norms were already used
in\ \cite{Huybrechs2019} to assess the stability of the approximations.

We observe three regimes.
First, for the propagative modes, i.e.\ the circular waves
with mode number \(|p| \leq \kappa\), the approximation is accurate ($\mathcal
E<10^{-13}$) and the size of the coefficients is moderate
($\|\boldsymbol{\xi}_{S,\epsilon}\|<10$).
Second, for mode numbers \(|p|\) roughly larger than the wavenumber \(\kappa\),
the norms of the coefficients of the computed approximations blow up
exponentially. The accuracy is spoiled proportionally.
Third, for evanescent modes with $|p|$ larger than about $2\kappa$ or
$3\kappa$, the size of the coefficients
completely destroys the stability of the approximation, and we cannot
approximate the target $b_p$ with any decent accuracy.
Of course, for a relative error at \(\mathcal{O}(1)\), the coefficient norm
reported is not meaningful, and taking \(\boldsymbol{\xi}_{S,\epsilon}\)
identically zero would provide a similar error.

Increasing the number of plane waves \(M\) has no effect on the accuracy beyond
a certain point.
Indeed, Figure~\ref{fig:propagative_instability_k16_svs_scatter}
shows that the 
\(\epsilon\)-rank (i.e.~the number of singular values larger than \(\epsilon\))
of the matrix \(A\) does not increase when \(M\) is raised.
Although increasing \(M\) does not bring any higher accuracy, it does not
increase any further the numerical instability.
For a fixed \(M\), the same matrix \(A\) is used to approximate all the
\(b_{p}\)'s for any mode number \(p\) (i.e.~to compute all markers of the
same color in Figure~\ref{fig:propagative_instability_k16}).
Even when the matrix \(A\) is extremely ill-conditioned (say \(M=32\kappa\) in
the numerical experiments presented here), we get at the same time almost
machine-precision accuracy for all propagative modes \(|p| \leq \kappa\) while
having \(\mathcal{O}(1)\) error for evanescent modes with larger mode number
\(|p| \geq 3\kappa\).
It is the simple regularization procedure described in
Section~\ref{sec:regularization} that allows us to obtain such results.
No other technique can overcome the inherent
instability of PPWs.
In particular, even with regularization, 
accuracy in the approximation of the evanescent modes remains out of reach for
a given floating-point precision.

Analoguous numerical results are also observed in the context of the
MFS, 
see\ \cite[Fig.~3]{Barnett2008}.

\begin{figure}
    \centering
    \begin{subfigure}{0.14\textwidth}
        \subcaption{PPW}\label{fig:propagative_instability_k16}
    \end{subfigure}
    \begin{subfigure}{0.85\textwidth}
        \raisebox{-.5\height}{\includegraphics[scale=1.2]{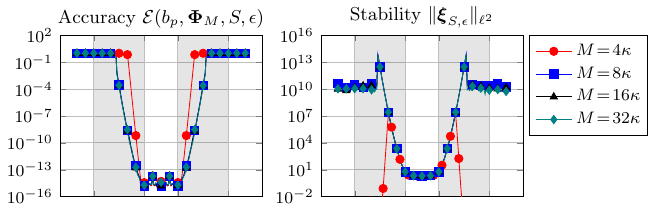}\hspace{-7mm}}
    \end{subfigure}
    \begin{subfigure}{0.14\textwidth}
        \subcaption{EPW:\\ deterministic\\ sampling}\label{fig:evanescent_stability_uniform_sampling_k16}
    \end{subfigure}
    \begin{subfigure}{0.85\textwidth}
        \raisebox{-.5\height}{\includegraphics[scale=1.2]{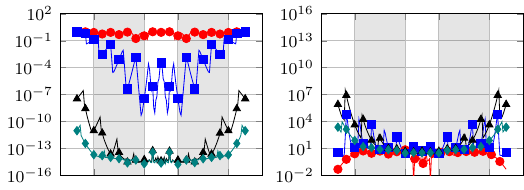}}
    \end{subfigure}
    \begin{subfigure}{0.14\textwidth}
        \subcaption{EPW:\\ Sobol\\ sampling}\label{fig:evanescent_stability_sobol_sampling_k16}
    \end{subfigure}
    \begin{subfigure}{0.85\textwidth}
        \raisebox{-.5\height}{\includegraphics[scale=1.2]{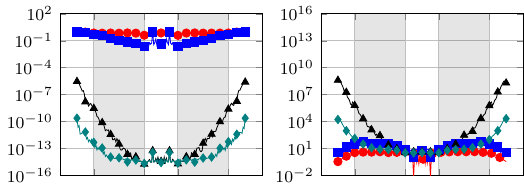}}
    \end{subfigure}
    \begin{subfigure}{0.14\textwidth}
        \subcaption{EPW:\\ random\\ sampling}\label{fig:evanescent_stability_random_sampling_k16}
    \end{subfigure}
    \begin{subfigure}{0.85\textwidth}
        \raisebox{-.5\height}{\includegraphics[scale=1.2]{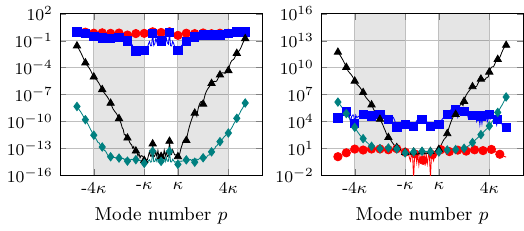}}
    \end{subfigure}
    \caption{Accuracy \(\mathcal{E}\), as defined in~\eqref{eq:relative-error},
    (left) and stability \(\|\boldsymbol{\xi}_{S,\epsilon}\|_{\ell^2}\) (right)
    of the approximation of circular waves \(b_{p}\)
    by PPW (top row) and EPWs (three bottom rows). 
    Truncation at $P=4\kappa$ for EPWs, wavenumber~\(\kappa=16\).
    With PPWs, the approximation accuracy does not improve as
    \(M\) increases beyond some value, because of exponentially large (with respect
    to \(p\)) coefficients, as proved in Lemma~\ref{lem:instability-propagative}.
    With EPWs, the approximation accuracy improves as \(M\) increases,
    thanks to a decrease of the size of the coefficients.
    }\label{fig:stability_k16}
\end{figure}

\subsection{Evanescent plane waves are stable}

We investigate, for the same test cases, whether the EPW sets proposed in
Section~\ref{sec:conj-stability} achieve better stability properties while not
compromising the accuracy of the approximation.
The approximation sets \(\boldsymbol{\Phi}_{P,M}\) are defined
in~\eqref{eq:evanescent-pw-sets} and the \(M\) EPWs have parameters
\(\{\mathbf{y}_{m}\}_{m=1}^{M}\)
computed as in~\eqref{eq:ymSamples}, i.e.\ distributed
according to the sampling distribution
\(\rho_{P}\) defined in~\eqref{eq:density-function}.
These EPWs are normalized as
in~\eqref{eq:evanescent-pw-sets}.
Here the parameter \(P\) used to generate the \(M\) samples
(which are adapted to the space \(\mathcal{A}_{P}\))
is set to \(4\kappa\).
The numerical results are reported in
Figure~\ref{fig:stability_k16}.

The main observation is that by using sufficiently many waves (i.e.~setting
\(M\) sufficiently large, on the order of \(M = 32\kappa \approx 4 N_{P}\)) we
are now able to approximate to (almost) machine precision all the modes \(|p|
\leq P = 4\kappa\).
This includes the propagative modes \(|p| \le \kappa\)
(which were already well-approximated by purely PPWs),
but more importantly, this also includes evanescent modes
\(\kappa < |p| \le P = 4\kappa\) (corresponding to the greyed out area), for
which purely PPWs failed to provide any meaningful approximation.
Moreover, even much higher modes \(|p| > P = 4\kappa\)
are approximated to acceptable accuracy.
Further, we stress that the norms of the coefficients
\(\|\boldsymbol{\xi}_{S,\epsilon}\|_{\ell^{2}}\) used in the approximate
expansions remain moderate, especially for large \(M\).
This is in stark contrast with the results of Section~\ref{sec:PWnumerics},
where the exponential growth of the coefficients prevented any accurate
numerical approximation.

\begin{figure}
    \centering
    \begin{subfigure}{0.35\textwidth}
        \centering
        \includegraphics[width=\textwidth]{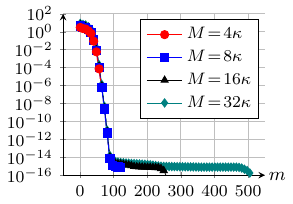}
        \caption{PPW.}\label{fig:propagative_instability_k16_svs_scatter}
    \end{subfigure}
    \begin{subfigure}{0.35\textwidth}
        \centering
        \includegraphics[width=\textwidth]{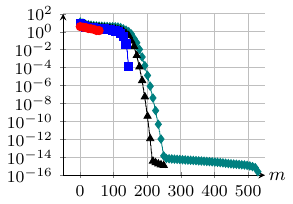}
        \caption{EPW: deterministic sampling}\label{fig:evanescent_stability_uniform_sampling_k16_svs_scatter}
    \end{subfigure}
    \begin{subfigure}{0.35\textwidth}
        \centering
        \includegraphics[width=\textwidth]{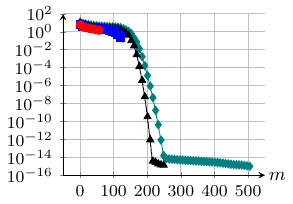}
        \caption{EPW: Sobol sampling}\label{fig:evanescent_stability_sobol_sampling_k16_svs_scatter}
    \end{subfigure}
    \begin{subfigure}{0.35\textwidth}
        \centering
        \includegraphics[width=\textwidth]{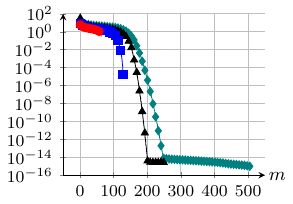}
        \caption{EPW: random sampling}\label{fig:evanescent_stability_random_sampling_k16_svs_scatter}
    \end{subfigure}
    \caption{Singular values \(\{\sigma_{m}\}_{m=1}^{M}\) of the matrix \(A\)
    when using a set of \(M\) plane waves.
    Truncation at $P=4\kappa$ for EPWs, wavenumber \(\kappa=16\).
    The matrices associated to EPWs are not better conditioned than the ones
    associated to PPWs,
    however the number of singular values above the regularization threshold
    \(\epsilon=10^{-14}\) increases with \(M\) and $P$.
    }\label{fig:k16_svs_scatter}
\end{figure}

In Figure~\ref{fig:k16_svs_scatter}, we observe that
the condition number of the matrix \(A\) is of the same order for PPWs and EPWs, when \(M\) is large enough.
The improved accuracy for evanescent modes is not due to an improved
conditioning of the underlying linear system
but to an increase of the \(\epsilon\)-rank of the matrix, i.e.\ the number of singular values
larger than \(\epsilon\). This number goes from less than \(100\) for PPWs to
around \(250\) for EPWs in the case \(M=32\kappa\).
To further increase the \(\epsilon\)-rank, one needs to increase the
truncation parameter \(P\).

Comparing PPWs and EPWs,
we see that for small \(M\) (e.g.\ \(M=4\kappa\) and \(M=8\kappa\)) purely
PPWs provide better approximation of propagative modes than EPWs.
This is because the approximation spaces made of PPWs are
tuned for propagative modes, which span a space of dimension \(2\kappa+1\).
In contrast, the approximation spaces made of EPWs target a
larger number of modes, including some evanescent modes, which span a space of
dimension \(N_{P}=2P+1\) with \(P=4\kappa\) in this numerical experiment.
For a general target solution containing evanescent modes,
one does not expect any advantage in using PPWs only.

\subsection{Approximation of random-expansion solutions}\label{sec:NumericsApproxConv}

We test the procedure described so far by reconstructing a solution of the form
\begin{equation}\label{eq:sol-surrogate}
    u
    := \sum_{|p| \leq P} \hat u_{p} \left[\max\left(1, |p|-\kappa\right)\right]^{-1/2}
    b_{p}
    \;\in \mathcal{B}_{P}
\end{equation}
in which $\hat u_p$ are
normally-dis\-trib\-uted random numbers with mean \(0\) and standard deviation \(1\).
The coefficients of any element of $\mathcal B$ decay in modulus as
$o(|p|^{-1/2})$ for $|p|\to\infty$;
this is therefore a rather difficult scenario for an approximation problem.

We then apply the procedure described above for the three types of sampling
strategies considered.
The sampling points are constructed knowing that \(T^{-1}u\) is an element of
\(\mathcal{A}_{P}\).
In other words, the optimal modal truncation parameter \(P^{*}=P\) (where
\(P\) appears in~\eqref{eq:sol-surrogate}) is assumed to be known in this
numerical experiment.
The main purpose is to investigate the validity of
Conjecture~\ref{conj:approx-conjecture}.
We study here the convergence of the error with respect to the 
dimension of the approximation space \(M\).
The number of sampling points on the boundary of the disk is set to
\(S=2M\).
The numerical results are given in Figure~\ref{fig:sol-surrogate_k16} for the
Sobol sampling strategy only.
On the left panel we report the relative residual \(\mathcal{E}\), defined
in~\eqref{eq:relative-error}, as a measure of the accuracy of the approximation.
On the right panel we report the size of the coefficients, namely
\(\|\boldsymbol{\xi}_{S,\epsilon}\|_{\ell^2} / \|u\|_{\mathcal{B}}\), as a
measure of the stability of the approximation.

\begin{figure}
    \centering
    \includegraphics[scale=1.2]{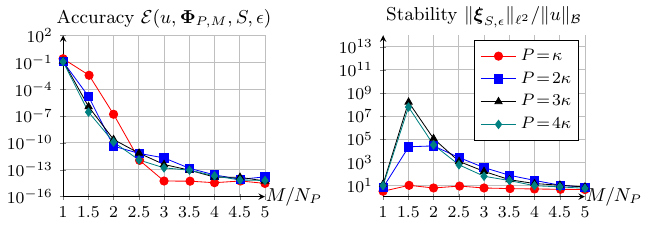}
    \caption{Accuracy \(\mathcal{E}\), as defined in~\eqref{eq:relative-error},
    (left) and stability
    \(\|\boldsymbol{\xi}_{S,\epsilon}\|_{\ell^2} / \|u\|_{\mathcal{B}}\)
    (right) of the
    approximation by \(M\) EPWs (constructed using Sobol sampling) of a solution
    \(u\) in the form~\eqref{eq:sol-surrogate} that belong to the space
    \(\mathcal{B}_{P}\) of dimension
    \(N_{P}=2P+1\).
    The horizontal axis represents the ratio \(M/N_{P}\).
    Wavenumber \(\kappa = 16\).
    The number \(M\) of EPWs necessary to approximate elements of the space
    \(\mathcal{B}_{P}\) seems to scale linearly with the space dimension
    \(N_{P}\).
    }\label{fig:sol-surrogate_k16}
\end{figure}

The main observation is that the error quickly decays with respect to the ratio
\(M/N_{P} = M/(2P+1)\), which represents the ratio of the dimension of the
approximation set \(M\) over the dimension of the space \(\mathcal{B}_{P}\) the
solution~\eqref{eq:sol-surrogate} lives in.
When \(P\) is large enough (say \(P \geq 2\kappa\) which remains moderate), the
decay is relatively independent of \(P\).
The second observation is that the norm of the coefficients
\(\|\boldsymbol{\xi}_{S,\epsilon}\|_{\ell^{2}} / \|u\|_{\mathcal{B}}\)
in the expansions is a decreasing function of the size \(M\) of the
approximation space.
We see once more that one gets accurate and stable approximations.
The values of
\(\|\boldsymbol{\xi}_{S,\epsilon}\|_{\ell^{2}} / \|u\|_{\mathcal{B}}\)
reported for small values of \(M/N_{P}\), and in particular the increase at the
start, are not significant since they correspond to inaccurate approximations.

We report in Figure~\ref{fig:bulk} the plots of
a solution~\eqref{eq:sol-surrogate} for a larger
frequency \(\kappa = 64\) and truncation parameter \(P = 3\kappa = 192\).
The approximation error when using \(M=3(2P+1)=1155\) PPWs or EPWs is also given, with points
in \(Y\) sampled as a Sobol sequence.
In the first case the absolute error in the disk is much larger, more than
\(12\) orders of magnitude larger if measured in $L^\infty(B_1)$ norm, and
concentrated near the boundary.
The number of degrees of freedom per wavelength \(\lambda={2\pi}/{\kappa}\)
used in each direction can be estimated by
\(\lambda\sqrt{{M}/{\pi}} \approx 1.9\). Note that \(\pi\) here represents the
area of the unit disk.
For low-order methods, a common rule of thumb is to
use around \(6 \sim 10\) degrees of freedom per wavelength
to have \(1\) or \(2\) digits of accuracy.
We obtain \(12\) digits of accuracy for only a fraction of this number.
For \(M=2(2P+1)=770\), the maximum absolute error reached is
measured to \(1.3 \cdot 10^{-10}\) (not plotted).

\begin{figure}[t]
    \centering
    \begin{subfigure}{0.49\textwidth}
        \centering
        \includegraphics[height=0.2\textheight]{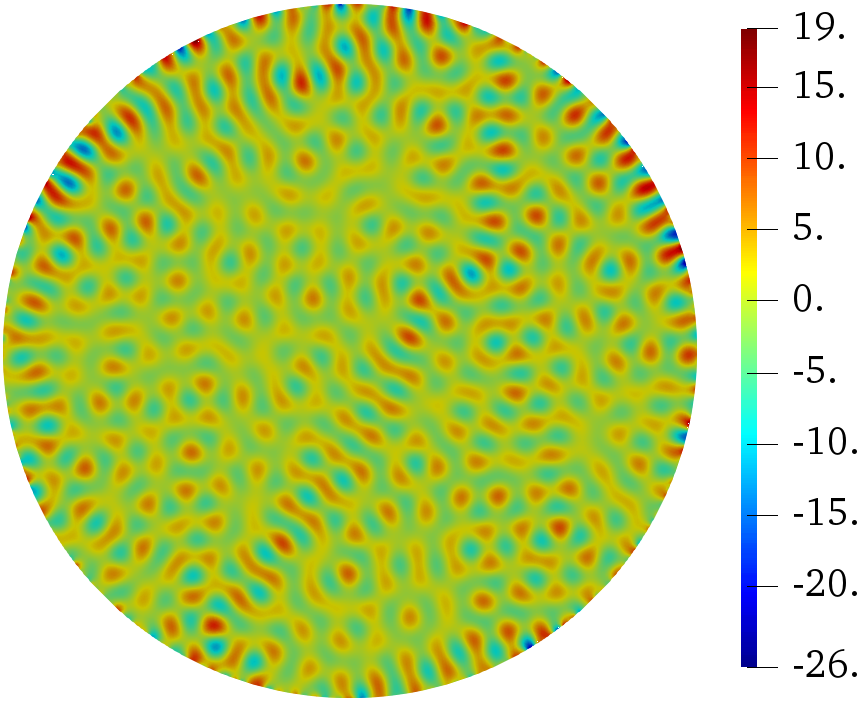}
        \caption{Real part of target solution \(\Re u\)}\label{fig:surrogate_sqrt_decay_k64_P192_M1155_real}
        \vspace{0.02\textwidth}
    \end{subfigure}
    \begin{subfigure}{0.49\textwidth}
        \centering
        \includegraphics[height=0.2\textheight]{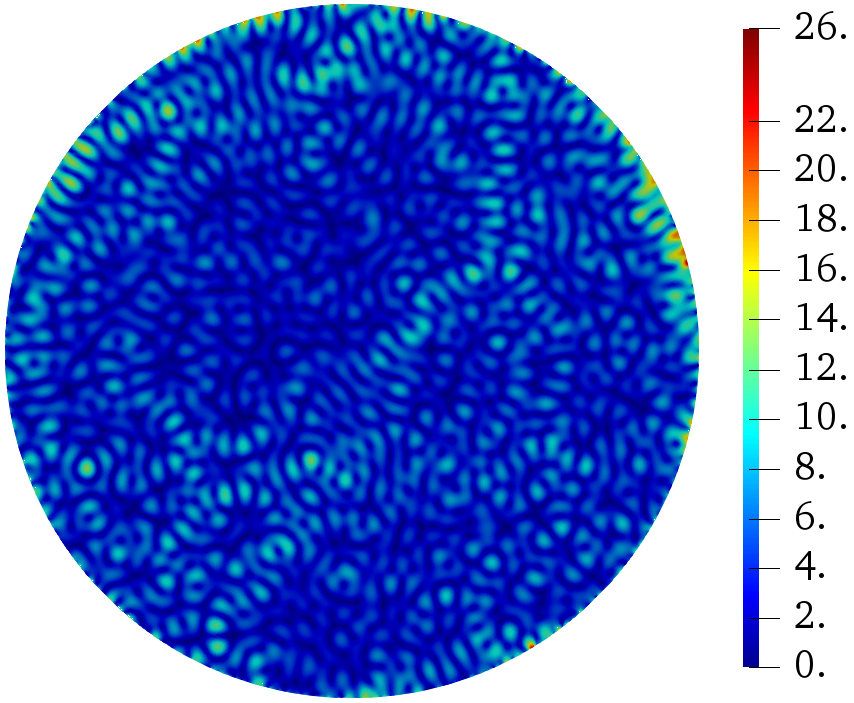}
        \caption{Modulus of target solution \(|u|\)}\label{fig:surrogate_sqrt_decay_k64_P192_M1155_abs}
        \vspace{0.02\textwidth}
    \end{subfigure}
    \begin{subfigure}{0.49\textwidth}
        \centering
        \includegraphics[height=0.2\textheight]{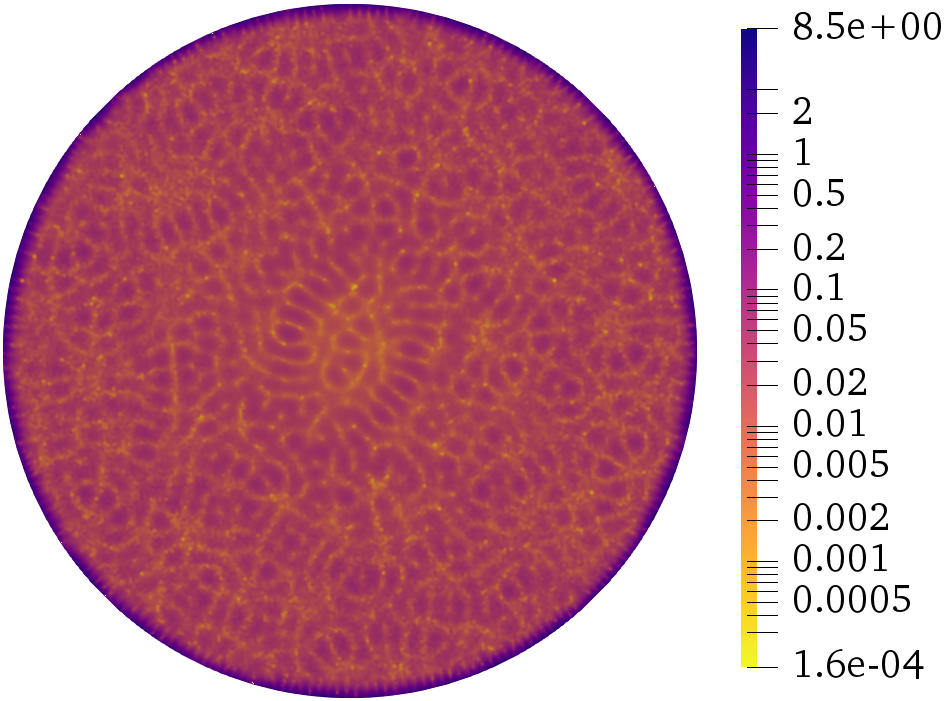}
        \caption{Absolute error using PPW
        \(|u - \mathcal{T}_{\boldsymbol{\Phi}_{M}}\boldsymbol{\xi}_{S,\epsilon}|\)}\label{fig:surrogate_sqrt_decay_k64_P192_M1155_ppw}
    \end{subfigure}
    \begin{subfigure}{0.49\textwidth}
        \centering
        \includegraphics[height=0.2\textheight]{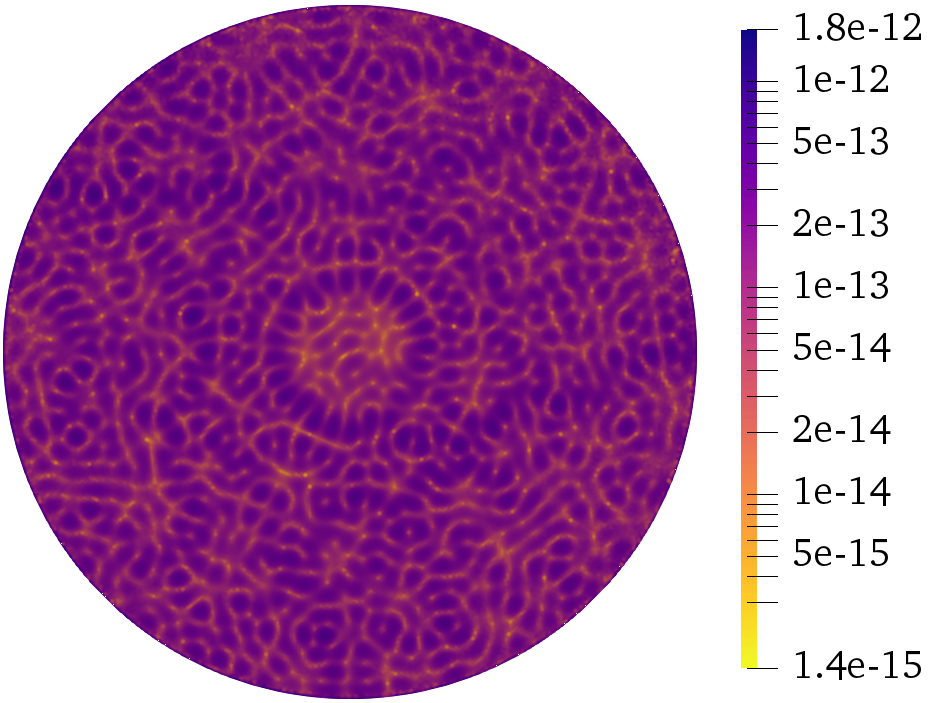}
        \caption{Absolute error using EPW
        \(|u - \mathcal{T}_{\boldsymbol{\Phi}_{P,M}}\boldsymbol{\xi}_{S,\epsilon}|\)}\label{fig:surrogate_sqrt_decay_k64_P192_M1155_epw_s}
    \end{subfigure}
    \caption{Solution \(u\), target of the approximation, defined
    in~\eqref{eq:sol-surrogate} with \(P=3\kappa=192\) (top) and associated
    absolute errors when approximated by \(M=3(2P+1)=1155\) plane waves,
    either propagative ones \(\boldsymbol{\Phi}_{M}\)
    from~\eqref{eq:propagative-pw-sets} (bottom left) or evanescent
    ones \(\boldsymbol{\Phi}_{P,M}\)
    from~\eqref{eq:evanescent-pw-sets},
    whose parameters are constructed using a Sobol type sampling (bottom
    right).
    The colormaps associated to absolute errors are logarithmic for better
    visualization.
    Wavenumber \(\kappa = 64\).
    Note the different color scales, which shows a factor-$10^{12}$ improvement in using
    EPWs instead of PPWs.
    }\label{fig:bulk}
\end{figure}

Overall, the numerical results are perfectly consistent with
Conjecture~\ref{conj:approx-conjecture}.

\subsection{Numerical evidence of quasi-optimality}\label{sec:numerics-quasi-optimality}

An important question regarding the efficiency of the proposed method concerns
how the size of the approximation set \(M\) should vary with respect to
the truncation parameter \(P\).
Fixing \(P\) amounts to looking at the finite dimensional subspace
\(\mathcal{B}_{P}\) which contains the first \(N_{P} = 2P+1\) modes.
Since \(N_{P}\) is the dimension of \(\mathcal{B}_{P}\) there is no hope to have 
approximation spaces with dimension \(M < N_{P}\) that are able to approximate
all elements of this space. 
An \emph{optimal} approximation set would therefore achieve this with \(M=N_{P}\)
elements at best.
We show numerical evidence that we achieve \emph{quasi-optimality},
in the sense that the approximation spaces \(\boldsymbol{\Phi}_{P,M}\) 
defined in~\eqref{eq:evanescent-pw-sets} only need \(M = \mathcal{O}(N_{P})\)
with a moderate proportionality constant to approximate the \(N_{P}\) circular
modes with reasonable accuracy.

We investigate numerically the linearity of the relation
\(P \to M^{*}(P,\eta)\), where \(M^{*}(P,\eta)\) was defined in
Conjecture~\ref{conj:approx-conjecture} (for a fixed \(\eta)\),
namely the validity of a law of the form
\(M^{*}(P,\eta) \approx \nu N_{P} = \nu (2P+1)\)
for some \(\nu = \nu(\eta) > 0\).
To that end, for some \(\sigma > 0\), we vary \(P\) and compute
\begin{equation}
    \widetilde{M}^{*} = \widetilde{M}^{*}(P,\sigma)
    := \min \big\{
        M \in \mathbb{N} \;|\;
        \mathcal{E}(b_{p}, \boldsymbol{\Phi}_{P,M}, S, \epsilon) \leq \sigma,
        \ \forall |p| \leq P
    \big\},
\end{equation}
where \(\mathcal{E}\) was defined in~\eqref{eq:relative-error}.
The quantity \(\widetilde{M}^{*}\) is expected to be a good estimate of
\(M^{*}(P,\eta)\).
The number of sampling points on the boundary of the disk is set to
\(S=2M\).

The numerical results are given in Figure~\ref{fig:oversampling} for the
accuracy level \(\sigma=10^{-12}\).
We represent here the variation of the ratio \(\widetilde{M}^{*}(P,\sigma)/N_{P}\)
with respect to the truncation parameter \(P\).
If the optimal law for \(\widetilde{M}^{*}(P,\sigma)\) was linear with
respect to \(P\), we would expect constant values.
Regardless of the type of sampling, we observe decreasing
curves that converge to some asymptotic value for \(\nu\) that falls within the
rather moderate range \([3,6]\).
This means that the first \(N_{P}\) circular modes (propagative and evanescent) can
be stably approximated with uniform relative error $\le10^{-12}$ using roughly
$3N_{P}$ to $6N_{P}$ EPWs.
Moreover, this asymptotic behavior seems to be robust with respect to the
wavenumber~\(\kappa\).
These more systematic results confirm
what was already observed in Section~\ref{sec:NumericsApproxConv}.
The behavior of the optimal asymptotic \(\widetilde{M}^{*}\) with
respect to \(N_{P}\) seems indeed to be linear or even sub-linear.

\begin{figure}
    \centering
    \begin{subfigure}{0.32\textwidth}
        \centering
        \includegraphics[width=\textwidth]{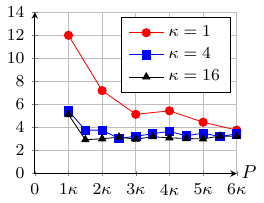}
        \caption{Deterministic sampling}\label{fig:uniform_sampling_quasi-optimality}
    \end{subfigure}
    \begin{subfigure}{0.32\textwidth}
        \centering
        \includegraphics[width=\textwidth]{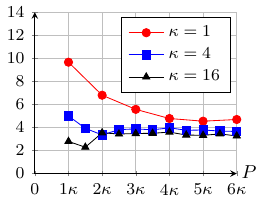}
        \caption{Sobol sampling}\label{fig:sobol_sampling_quasi-optimality}
    \end{subfigure}
    \begin{subfigure}{0.32\textwidth}
        \centering
        \includegraphics[width=\textwidth]{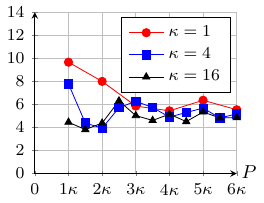}
        \caption{Random sampling}\label{fig:random_sampling_quasi-optimality}
    \end{subfigure}
    \caption{Ratio \(\widetilde{M}^{*}(P,\sigma)/N_{P}\) with respect
    to the truncation parameter \(P\) for various types of sampling
    method and \(\sigma=10^{-12}\).
    The number of EPWs necessary to approximate elements of the space
    \(\mathcal{B}_{P}\) to relative accuracy $\sigma$
    seems to scale linearly with the space dimension \(N_{P}\).
    }\label{fig:oversampling}
\end{figure}

\subsection{Triangular domain}

We conclude this section with some numerical results on a triangular
geometry.
Our purpose is to show that the approximation sets that we constructed also
exhibit good approximation properties on other shapes, despite being built
following the analysis for the disk.

We consider a triangle \(\Omega\) inscribed in the unit disk, with
vertices
\(\mathbf{v}_{1} = (1,0)\),
\(\mathbf{v}_{2} = (-1,0)\) and
\(\mathbf{v}_{3} = (\cos(5\pi/8), \sin(5\pi/8))\).
The target of the approximation problem is the Helmholtz fundamental solution
$\mathbf{x} \mapsto ({\imath}/{4}) H^{(1)}_{0}(\kappa |\mathbf{x} - \mathbf{s}|)$,
for wavenumber \(\kappa=16\) and for two different locations
\(\mathbf{s} \in \mathbb{R}^{2}\setminus \overline{\Omega}\)
of the singularity,
see~Figure~\ref{fig:triangle_k16_targets}.

\begin{figure}
    \centering
    \begin{subfigure}{0.49\textwidth}
        \centering
        \includegraphics[width=\textwidth]{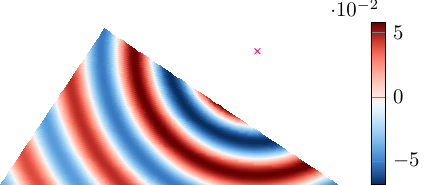}
        \caption{Singularity close to one edge.}\label{fig:triangle_k16_Flat_real}
    \end{subfigure}
    \begin{subfigure}{0.49\textwidth}
        \centering
        \includegraphics[width=\textwidth]{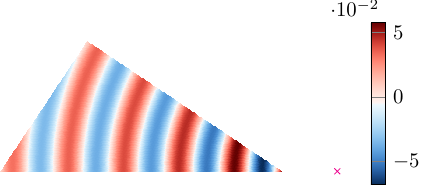}
        \caption{Singularity close to one vertex.}\label{fig:triangle_k16_Tip_real}
    \end{subfigure}
    \caption{Real part of the fundamental solutions used as target for the
    approximation problem in the triangle.
    The magenta cross \({\color{magenta}\times}\) indicates the position of
    the singularity \(\mathbf{s}\) and is located one wavelength
    \(\lambda=2\pi/\kappa\) away from the boundary of the triangle. 
    Wavenumber \(\kappa=16\).}\label{fig:triangle_k16_targets}
\end{figure}

We study the convergence of the approximation by plane waves
for increasing size of the approximation set \(M\).
The approximation is constructed as indicated in
Section~\ref{sec:Sampling}--\ref{sec:regularization}
from Dirichlet data at equispaced points on the boundary of the triangle
and by solving the oversampled linear systems using a regularized SVD.
The plane waves used in the approximation sets are either propagative,
with uniformly spaced angles as described in~\eqref{eq:propagative-pw-sets},
or evanescent, as described in~\eqref{eq:evanescent-pw-sets}.
The approximation set using EPWs is constructed from sampling
the probability density function \(\rho_{P}\) defined
in~\eqref{eq:density-function} following a Sobol sequence.
For a given size \(M\) of the approximation set, the Fourier truncation
parameter is computed as
\(P := \max\left(\lceil \kappa \rceil, \lfloor M/4 \rfloor \right)\),
as suggested by Figure~\ref{fig:sobol_sampling_quasi-optimality}.
Finally, the EPWs are re-normalized to have 
unit \(L^{\infty}\) norm on the boundary of the triangle. The latter normalization is the only
modification with respect to the sets used for the circular geometry.

The convergence results are presented in
Figure~\ref{fig:triangle_convergence_vs_M_k16_res}.
When using PPWs, the residual initially
decreases rapidly with \(M\) but stalls well before reaching machine precision
due to the rapidly growing coefficients.
In contrast, when using EPWs, the residual
converges to machine precision and the size of the
coefficients remains moderate when the final accuracy is reached.

\begin{figure}
    \centering
    \includegraphics[width=0.9\textwidth]{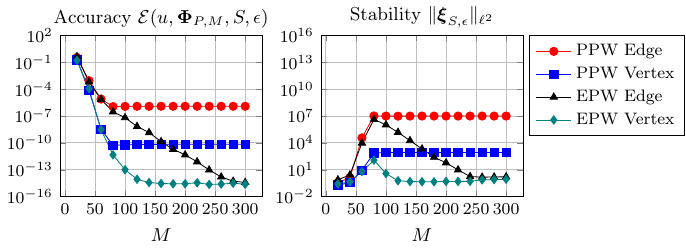}
    \caption{Accuracy \(\mathcal{E}\), as defined in~\eqref{eq:relative-error},
    (left) and stability \(\|\boldsymbol{\xi}_{S,\epsilon}\|_{\ell^2}\) (right)
    of the approximation of the fundamental solutions on the triangle $\Omega$ (see
    Figure~\ref{fig:triangle_k16_targets} for the meaning of the ``edge'' and
    ``vertex'' configurations) by PPWs or EPWs. 
    Wavenumber \(\kappa=16\) and regularization parameter
    $\epsilon=10^{-14}$.
    The convergence with respect to the size of the approximation set \(M\)
    stalls when using PPWs, due to the need for large coefficients, while EPWs
    reach machine precision.
    }\label{fig:triangle_convergence_vs_M_k16_res}
\end{figure}

We also report in Figure~\ref{fig:triangle_k16_errors} the point-wise absolute
error in the bulk of the triangle between the exact solution and the computed
approximation, linearly interpolated on a triangular mesh for visualisation
purposes.
The \(L^{\infty}\)-norm of the error inside the triangle
is of the same order of magnitude as the residual reported in
Figure~\ref{fig:triangle_convergence_vs_M_k16_res}.
The error with EPWs is of the order of machine
precision, whereas the error with PPWs
is mainly concentrated on the boundary of the triangle.

\begin{figure}
    \centering
    \begin{subfigure}{0.48\textwidth}
        \centering
        \includegraphics[width=\textwidth]{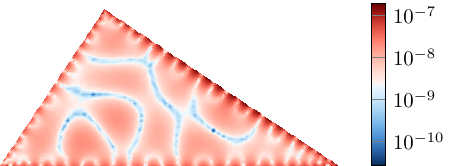}
        \caption{PPW - Edge
        }\label{fig:triangle_k16_M300_PPW_Flat_log}
    \end{subfigure}
    \begin{subfigure}{0.48\textwidth}
        \centering
        \includegraphics[width=\textwidth]{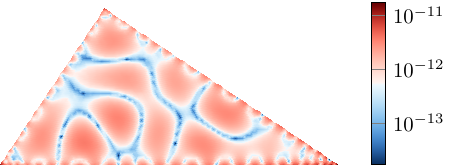}
        \caption{PPW - Vertex
        }\label{fig:triangle_k16_M300_PPW_Tip_log}
    \end{subfigure}
    \begin{subfigure}{0.48\textwidth}
        \centering
        \includegraphics[width=\textwidth]{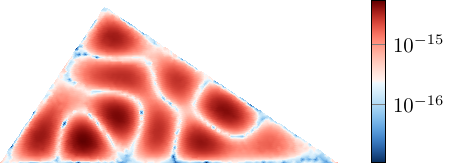}
        \caption{EPW - Edge
        }\label{fig:triangle_k16_M300_EPW_Flat_log}
    \end{subfigure}
    \begin{subfigure}{0.48\textwidth}
        \centering
        \includegraphics[width=\textwidth]{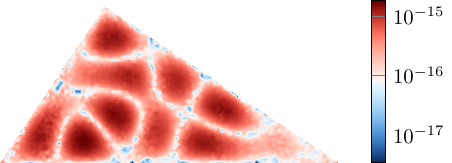}
        \caption{EPW - Vertex
        }\label{fig:triangle_k16_M300_EPW_Tip_log}
    \end{subfigure}
    \caption{Point-wise error in the triangle between the target of the
    approximation problem (see Figure~\ref{fig:triangle_k16_targets}) and the
    approximation using propagative (top) and evanescent (bottom) plane waves.
    The singularity in the solution is either close to the edge (left)
    or close to the vertex (right).
    Wavenumber \(\kappa=16\) and \(M=300\).}\label{fig:triangle_k16_errors}
\end{figure}

These results show the potential of the proposed numerical recipe for Trefftz
methods and plane wave approximations.
This is even more striking considering that the numerical recipe used
to construct the approximations is not tuned for the triangular
geometry, with the exception of the re-normalization.
Better rules adapted to the underlying geometry might yield even more efficient
approximation schemes and are the subject of ongoing investigations.

\section{Conclusions}\label{sec:conclusions}

Ill-conditioning is inherent in plane-wave based Trefftz schemes but can be
overcome if there exist accurate approximations that are moreover stable, in the sense of having expansions with bounded coefficients.
To approximate Helmholtz solutions, PPWs are known to
provide accurate approximations.
However, the associated expansions are necessarily \emph{unstable}:
the norm of the coefficients blow up for solutions with high-frequency Fourier
modes.
In contrast, EPWs, which contain high-frequency content, give
accurate as well as \emph{stable} results.
To construct stable sets of EPWs, we show numerically that an
effective strategy is to sample the parametric domain according to a
fully explicit probability measure.

This paper is only the first step towards stable and accurate approximation
schemes based on EPWs.
A theoretical problem that we have left open is the analysis of the
approximation properties of the sets of EPWs
constructed using our numerical recipe.
Next steps include the extensions to more general geometries, three-dimensional
problems (see\ \cite{Galante2023}), time-harmonic Maxwell and elastic wave
equations, the application to Trefftz schemes and to sound-field reconstruction
algorithms.
Preliminary experiments show that the proposed numerical recipe performs
well for convex polygons and in Trefftz-Discontinuous Galerkin schemes with
several cells, and provides a considerable improvement over standard
PPW schemes.

\paragraph{Acknowledgements}

The authors are grateful to Albert Cohen, Matthieu Dolbeault and Ralf
Hiptmair for helpful discussions,
and to Nicola Galante for his careful proofreading.
AM and EP acknowledge support from PRIN project ``NA--FROM--PDEs'' and from
MIUR through the ``Dipartimenti di Eccellenza'' Program (2018--2022) -- Dept.
of Mathematics, University of Pavia.

\appendix

\section{Proofs of Section~\ref{sec:Helmholtz_circular-waves}}\label{app:sec2-proofs}

\begin{proof}[Proof of Lemma~\ref{lem:b_Hilbert_basis}]
    We only need to prove that the family \(\{b_{p}\}_{p\in\mathbb{Z}}\) is
    orthogonal, which is a consequence of the orthogonality of the complex
    exponentials
    \(\{\theta\mapsto e^{\imath p \theta}\}_{p\in\mathbb{Z}}\)
    on the unit circle \(\partial B_{1}\).
    For \(p,q\in\mathbb{Z}\), we have
    \begin{equation}
        (\tilde{b}_{p},\, \tilde{b}_{q})_{L^{2}(B_{1})}
        = \int_{0}^{1} J_{p}(\kappa r)J_{q}(\kappa r)r\;\mathrm{d}r
        \int_{0}^{2\pi} e^{\imath (p-q)\theta}\;\mathrm{d}\theta
        = 2\pi\int_{0}^{1} J_{p}^2(\kappa r)r\;\mathrm{d}r\;
        \delta_{pq}.
    \end{equation} 
    The orthogonality in \(H^{1}(B_{1})\) is easily seen from
    \begin{equation}\label{eq:semi-norm-H1-bp}
        \begin{aligned}
            (\nabla \tilde{b}_{p},\, \nabla \tilde{b}_{q})_{L^{2}(B_{1})^2}
            = (\partial_{\mathbf{n}} \tilde{b}_{p},\, \tilde{b}_{q})_{L^{2}(\partial B_{1})}
            - (\Delta \tilde{b}_{p},\, \tilde{b}_{q})_{L^{2}(B_{1})}
            = (\partial_{\mathbf{n}} \tilde{b}_{p},\, \tilde{b}_{q})_{L^{2}(\partial B_{1})}
            + \kappa^2(\tilde{b}_{p},\, \tilde{b}_{q})_{L^{2}(B_{1})},
        \end{aligned}
    \end{equation} 
    where we denoted by \(\mathbf{n}\) the outward unit normal vector
    and
    \begin{equation}\label{eq:bnd-term-betap}
        (\partial_{\mathbf{n}} \tilde{b}_{p},\, \tilde{b}_{q})_{L^{2}(\partial B_{1})}
        = \kappa J_{p}'(\kappa)J_{q}(\kappa)
        \int_{0}^{2\pi} e^{\imath (p-q)\theta}\;\mathrm{d}\theta
        = 2\pi \kappa J_{p}'(\kappa)J_{p}(\kappa)
        \delta_{pq}.
    \end{equation}
\end{proof}

\begin{proof}[Proof of Lemma~\ref{lem:B_space_Helmholtz_solution}]
    It is straightforward to check that any \(b_{p}\),
    for \(p \in \mathbb{Z}\), is solution to the Helmholtz
    equation~\eqref{eq:Helmholtz}.
    The continuity of the Helmholtz operator
    \begin{equation}
        \begin{aligned}
            & \mathcal{L} \;:\;
            H^{1}(B_{1}) \to H^{-1}(B_{1}) = \left(H^{1}_{0}(B_{1})\right)^{*},
            \ \text{defined by:}\\
            & \langle \mathcal{L}u,\, v \rangle_{H^{-1} \times H^{1}_{0}}
            := \left(\nabla u,\, \nabla v\right)_{L^{2}(B_{1})}
            - \kappa^{2} \left( u,\,  v\right)_{L^{2}(B_{1})},
            \qquad\forall u \in H^{1}(B_{1}),\ v \in H^{1}_{0}(B_{1}),
        \end{aligned}
    \end{equation}
    implies that the kernel of \(\mathcal{L}\) is a closed subspace of
    \(H^{1}(B_{1})\).
    From the definition of \(\mathcal{B}\) given in~\eqref{eq:btilde_p},
    it follows that
    \begin{equation}
        \mathcal{B} \subset \ker \mathcal{L}
        := \left\{u \in H^{1}(B_{1}) \;|\;
        \mathcal{L} u = 0\right\}.
    \end{equation}

    Conversely, let \(u \in H^{1}(B_{1})\) satisfy~\eqref{eq:Helmholtz}
    and set \(g := \partial_{\mathbf{n}} u
    - \imath \kappa u \in H^{-1/2}(\partial B_{1})\).
    The Robin trace \(g\) can be written
    \begin{equation}
        g(\theta) = \sum_{p \in \mathbb{Z}} \hat{g}_{p} e^{\imath p \theta},
        \qquad \forall \theta \in [0,2\pi),
        \quad \text{with} \quad 
        \sum_{p \in \mathbb{Z}} |\hat{g}_{p}|^{2} (1+p^{2})^{-1/2}
        < \infty.
    \end{equation}
    Let \(P \geq 0\), and set
    \(g_{P}(\theta) := \sum_{|p| < P} \hat{g}_{p} e^{\imath p \theta}\),
    for $\theta \in [0,2\pi)$.
    Then there exists
    a unique \(u_{P} \in \operatorname{span} \{ b_{p} \}_{|p| < P}\),
    such that
    \( g_{P} = \partial_{\mathbf{n}} u_{P} - \imath \kappa u_{P}\),
    namely $u_P=\sum_{|p|<P}\hat g_p(\kappa\beta_p(J'_p(\kappa)-\imath
    J_p(\kappa)))^{-1}b_p$ (the term $J'_p(\kappa)-\imath J_p(\kappa)$ at the
    denominator is non-zero because of\ \cite[Eq.~(10.21.2)]{DLMF}).
    The well-posedness \cite[Prop.~8.1.3]{Melenk1995} of the problem:
    find \(v \in H^{1}(B_{1})\) such that
    \begin{equation}
        -\Delta v - \kappa^{2} v = 0, \quad\text{in}\ B_{1},
        \qquad\text{and}\qquad
        \partial_{\mathbf{n}} v - \imath \kappa v = h,
        \quad\text{on}\ \partial B_{1},
    \end{equation}
    for \(h \in H^{-1/2}(\partial B_{1})\), 
    implies that there exists a constant \(C > 0\),
    independent of \(P\), such that
    $\|u - u_{P}\|_{\mathcal{B}} \leq C 
        \|g - g_{P}\|_{H^{-1/2}(\partial B_{1})}$.
    Letting \(P\) tend to infinity, we obtain that \(u \in \mathcal{B}\).
\end{proof}

\begin{proof}[Proof of Lemma~\ref{lem:beta_asymptotic}]
    The explicit expression for \(\beta_{p}\) can be deduced by integrating by
    parts as in the proof of Lemma~\ref{lem:b_Hilbert_basis}.
    From~\eqref{eq:semi-norm-H1-bp}, the explicit expression for the boundary
    term~\eqref{eq:bnd-term-betap} and\ \cite[Eq.~(10.22.5)]{DLMF},
    \begin{equation}\label{eq:L2-norm-betap}
        \|\tilde{b}_{p}\|^2_{L^{2}(B_{1})}
        = 2\pi\int_{0}^{1} J_{p}^2(\kappa r)r\;\mathrm{d}r
        = \pi \left(J_{p}^2(\kappa) - J_{p-1}(\kappa)J_{p+1}(\kappa)\right),
    \end{equation}
    we 
    deduce the expression in~\eqref{eq:beta_explicit}.
    Then the asymptotic behavior is obtained by proving that
    \begin{equation}\label{eq:asymptotics_btildep_norms}
        \begin{aligned}
            &\|\tilde{b}_{p}\|_{L^{2}(\partial B_{1})}
            \sim\left({e\kappa}/{2}\right)^{|p|}\; |p|^{-\left(|p|+1/2\right)},\\
            &\|\tilde{b}_{p}\|_{L^{2}(B_{1})}
            \sim 2^{-1/2}
            \left({e\kappa}/{2}\right)^{|p|}\; |p|^{-\left(|p|+1\right)},\\
            &\|\tilde{b}_{p}\|_{\mathcal{B}}
            \sim \kappa^{-1}
            \left({e\kappa}/{2}\right)^{|p|}\; |p|^{-|p|},
        \end{aligned}
        \qquad\text{as}\ |p|\to +\infty.
    \end{equation}
    For any \(p\in\mathbb{Z}\),
    \(J_{-p} = (-1)^{p}J_{p}\) from\ \cite[Eq.~(10.4.1)]{DLMF}.
    Therefore, the asymptotic behavior will not depend on the sign of
    \(p\), and we suppose \(p>0\) in the following.
    We start with the trace: from the definition \eqref{eq:btilde_p} of
    $\tilde b_p$, 
    \(
        \|\tilde{b}_{p}\|^2_{L^{2}(\partial B_{1})}
        = 2\pi J_{p}^2(\kappa),
    \)
    and from\ \cite[Eq.~(10.19.1)]{DLMF}, namely
    \begin{equation}\label{eq:bessel-asymptotic}
        J_{\nu}\left(z\right)
        \sim
        ({2\pi\nu})^{-1/2}({ez}/{2\nu})^{\nu},
        \qquad\text{as}\ \nu\to +\infty,
        \qquad z\neq 0,
    \end{equation}
    the first result in \eqref{eq:asymptotics_btildep_norms} follows.
    We now consider the \(L^{2}(B_{1})\) norm.
    From~\eqref{eq:L2-norm-betap} and~\eqref{eq:bessel-asymptotic}, we get
    as \(p\to +\infty\)
    \begin{equation}
        \|\tilde{b}_{p}\|^2_{L^{2}(B_{1})}
        \sim
        \frac{1}{2}
        \left(\frac{e\kappa}{2}\right)^{2p}\; p^{-\left(2p+1\right)}
        \left[ 1 - \frac{p^{2p+1}}{(p-1)^{p-1/2}(p+1)^{p+3/2}} \right],
    \end{equation}
    and it is readily checked that the term inside the square brackets is
    equivalent to \(p^{-1}\) at infinity,
    so the second result in \eqref{eq:asymptotics_btildep_norms} follows.
    We now consider the \(\kappa\)-weighted \(H^{1}(B_{1})\)
    norm~\eqref{eq:Bnorm}.
    We need to study the asymptotic of the boundary
    term~\eqref{eq:bnd-term-betap}.
    From\ \cite[Eq.~(10.6.1)]{DLMF}
    \begin{equation}
        (\partial_{\mathbf{n}} \tilde{b}_{p},\, \tilde{b}_{p})_{L^{2}(\partial B_{1})}
        = 2\pi \kappa J_{p}'(\kappa)J_{p}(\kappa)
        = \pi\kappa \big(J_{p-1}(\kappa) - J_{p+1}(\kappa)\big)J_{p}(\kappa).
    \end{equation}
    From~\eqref{eq:bessel-asymptotic}, we get
    as \(p\to +\infty\)
    \begin{equation}
        (\partial_{\mathbf{n}} \tilde{b}_{p},\, \tilde{b}_{p})_{L^{2}(\partial B_{1})}
        \sim
        \frac{\kappa}{2}
        \left(\frac{e\kappa}{2}\right)^{2p}\; p^{-\left(2p+1\right)}
        \left[
            \frac{2}{e\kappa}
            \frac{p^{p+1/2}}{(p-1)^{p-1/2}} -
            \frac{e\kappa}{2}
            \frac{p^{p+1/2}}{(p+1)^{p+3/2}}
        \right],
    \end{equation}
    and it is readily checked that the first term inside the square brackets is
    dominant and equivalent to \(\frac{2}{\kappa}p\) at infinity.
    Thus, the dominant term in~\eqref{eq:semi-norm-H1-bp} in the limit
    \(p\to\infty\) is the boundary term.
\end{proof}

\section{Proofs of Section~\ref{sec:stability-notion}}\label{app:sec3-proofs}

\begin{proof}[Proof of Proposition~\ref{prop:approx-error-estimate}]
    The method of proof closely follows that of\ \cite[Th.~3.7]{Adcock2020}.
    In particular we first establish a so-called Marcinkiewicz--Zygmund condition,
    akin to\ \cite[Eq.~(3.2)]{Adcock2020}.

    The regularity assumption 
    for \(u\) and \(\boldsymbol{\Phi}_{k}\), which are assumed in
    \(\mathcal{B} \cap C^{0}(\overline{B_{1}})\),
    allows to have well-defined pointwise evaluations
    of their image by the Dirichlet trace operator \(\gamma\) on the boundary
    \(\partial B_{1}\).
    Recall that the sampling nodes \(\{\mathbf{x}_{s}\}_{s}\) are defined
    in~\eqref{eq:boundary_sampling_nodes}.
    For any \(v \in \mathcal{B} \cap C^{0}(\overline{B_{1}})\),
    \begin{equation}\label{eq:proof-approx-parseval-frame}
        \lim_{S \to +\infty} \frac{2\pi}{S}
        \sum_{s=1}^{S} \left| (\gamma v)(\mathbf{x}_{s}) \right|^{2}
        = \|\gamma v\|_{L^{2}(\partial B_{1})}^{2}.
    \end{equation}
    The argument of the limit in the left-hand-side is a Riemann sum
    approximant of the right-hand-side.
    A similar argument is developed in\ \cite[Ex.~3.3]{Adcock2020}
    (note that \(A'=B'=1\) in the notations of\ \cite{Adcock2020}).
    We will repeatedly use~\eqref{eq:proof-approx-parseval-frame}
    in the remainder of the proof.
    
    Let \(\boldsymbol{\mu} \in \mathbb{C}^{|\boldsymbol{\Phi}_{k}|}\).
    From~\eqref{eq:solution_SVDr}, we have
    \begin{equation}\label{eq:proof-approx-error-decomposition}
        \begin{aligned}
            u
            -
            \mathcal{T}_{\boldsymbol{\Phi}_{k}}\boldsymbol{\xi}_{S,\epsilon}
            & = 
            [
                u
                -
                \mathcal{T}_{\boldsymbol{\Phi}_{k}}\boldsymbol{\mu}
            ]
            +
            [
                \mathcal{T}_{\boldsymbol{\Phi}_{k}} A_{S,\epsilon}^{\dagger}
                A \boldsymbol{\mu}
                -
                \mathcal{T}_{\boldsymbol{\Phi}_{k}}\boldsymbol{\xi}_{S,\epsilon}
            ]
            +
            [
                \mathcal{T}_{\boldsymbol{\Phi}_{k}}\boldsymbol{\mu}
                -
                \mathcal{T}_{\boldsymbol{\Phi}_{k}} A_{S,\epsilon}^{\dagger}
                A \boldsymbol{\mu}
            ]\\
            & = 
            [
                u
                -
                \mathcal{T}_{\boldsymbol{\Phi}_{k}}\boldsymbol{\mu}
            ]
            +
                \mathcal{T}_{\boldsymbol{\Phi}_{k}} A_{S,\epsilon}^{\dagger}
            [
                A \boldsymbol{\mu}
                -
                \mathbf{b}
            ]
            +
            \mathcal{T}_{\boldsymbol{\Phi}_{k}}
            [
                \mathrm{Id}
                -
                A_{S,\epsilon}^{\dagger} A
            ] \boldsymbol{\mu}.
        \end{aligned}
    \end{equation}
    The proof proceeds by estimating the \(L^{2}\) norm of the trace on
    \(\partial B_{1}\) of each term.

    The first term appears in the estimate we want to derive, so 
    we examine the second term in~\eqref{eq:proof-approx-error-decomposition}.
    From~\eqref{eq:proof-approx-parseval-frame},
    provided \(S\) has been chosen sufficiently large,
    we can write
    (picking the constant \(2\) on the right-hand-side for simplicity, but
    any constant \(>1\) would work)
    \begin{equation}
        \|\gamma(
            \mathcal{T}_{\boldsymbol{\Phi}_{k}} A_{S,\epsilon}^{\dagger}
            [ A \boldsymbol{\mu} - \mathbf{b} ]
        )\|_{L^{2}(\partial B_{1})}^{2}
        \leq 2
        \frac{2\pi}{S}
        \sum_{s=1}^{S}
        |\gamma(
            \mathcal{T}_{\boldsymbol{\Phi}_{k}} A_{S,\epsilon}^{\dagger}
            [ A \boldsymbol{\mu} - \mathbf{b} ])(\mathbf{x}_{s})
        |^{2}
        \leq 
        \frac{4\pi}{S}
        \|
            A A_{S,\epsilon}^{\dagger}
            \left[
                A \boldsymbol{\mu}
                -
                \mathbf{b}
            \right]
        \|_{\ell^{2}}^{2}.
    \end{equation}
    Our choice of regularization~\eqref{eq:mat_approx_SVDr} ensures that 
    \(\|A A_{S,\epsilon}^{\dagger}\| \leq 1\), from which we deduce
    \begin{equation}
        \|\gamma(
            \mathcal{T}_{\boldsymbol{\Phi}_{k}} A_{S,\epsilon}^{\dagger}
            [ A \boldsymbol{\mu} - \mathbf{b} ]
        )\|_{L^{2}(\partial B_{1})}^{2}
        \leq 2
        \frac{2\pi}{S}
        \| A \boldsymbol{\mu} - \mathbf{b} \|_{\ell^{2}}^{2}
        = 2
        \frac{2\pi}{S}
        \sum_{s=1}^{S}
        |
            \gamma(\mathcal{T}_{\boldsymbol{\Phi}_{k}}
            \boldsymbol{\mu} - u)(\mathbf{x}_{s})
        |^{2}.
    \end{equation}
    Using once more~\eqref{eq:proof-approx-parseval-frame},
    provided \(S\) 
    is sufficiently large,
    we can write
    (with an additional factor~2)
    \begin{equation}
        \|\gamma(
            \mathcal{T}_{\boldsymbol{\Phi}_{k}} A_{S,\epsilon}^{\dagger}
            [ A \boldsymbol{\mu} - \mathbf{b} ]
        )\|_{L^{2}(\partial B_{1})}^{2}
        \leq 4
        \|\gamma(
            u
            -
            \mathcal{T}_{\boldsymbol{\Phi}_{k}}\boldsymbol{\mu}
        )\|_{L^{2}(\partial B_{1})}^{2}.
    \end{equation}

    We now examine the third term in~\eqref{eq:proof-approx-error-decomposition}.
    Arguing as before, from~\eqref{eq:proof-approx-parseval-frame}, there exists \(S\)
    sufficiently large such that
    \begin{equation}
        \|\gamma(
            \mathcal{T}_{\boldsymbol{\Phi}_{k}}
            [ \mathrm{Id} - A_{S,\epsilon}^{\dagger} A ] \boldsymbol{\mu}
        )\|_{L^{2}(\partial B_{1})}^{2}
        \leq 2
        \frac{2\pi}{S}
        \sum_{s=1}^{S}
        |\gamma(
            \mathcal{T}_{\boldsymbol{\Phi}_{k}}
            [ \mathrm{Id} - A_{S,\epsilon}^{\dagger} A ] \boldsymbol{\mu}
        )(\mathbf{x}_{s})|^{2}
        \leq 2
        \frac{2\pi}{S}
        \|
            A [ \mathrm{Id} - A_{S,\epsilon}^{\dagger} A ] \boldsymbol{\mu}
        \|_{\ell^{2}}^{2}.
    \end{equation}
    Our choice of regularization~\eqref{eq:mat_approx_SVDr} ensures that 
    \(\| 
        A [ \mathrm{Id} - A_{S,\epsilon}^{\dagger} A ]
    \| \leq \epsilon \sigma_{\max}
    \) so that
    \begin{equation}
        \|\gamma(
            \mathcal{T}_{\boldsymbol{\Phi}_{k}}
            [ \mathrm{Id} - A_{S,\epsilon}^{\dagger} A ] \boldsymbol{\mu}
        )\|_{L^{2}(\partial B_{1})}^{2}
        \leq 2
        \frac{2\pi}{S} \epsilon^{2} \sigma_{\max}^{2}
        \| \boldsymbol{\mu} \|_{\ell^{2}}^{2}.
    \end{equation}
    Combining all estimates,~\eqref{eq:approx-error-estimate1} is readily
    obtained.

    In order to show~\eqref{eq:approx-error-estimate2}, note first that the
    continuity of the trace operator \(\gamma\) from
    \(\mathcal{B}\) to \(L^{2}(\partial B_{1})\) allows to write,
    for any \(\boldsymbol{\mu} \in \mathbb{C}^{|\boldsymbol{\Phi}_{k}|}\),
    $\|\gamma(u - \mathcal{T}_{\boldsymbol{\Phi}_{k}}\boldsymbol{\mu})\|_{L^{2}(\partial B_{1})}
        \leq \| \gamma \|\;
        \|u - \mathcal{T}_{\boldsymbol{\Phi}_{k}}\boldsymbol{\mu}\|_{\mathcal{B}}$.
    It remains to bound the \(L^{2}(B_{1})\) norm
    of \(u - \mathcal{T}_{\boldsymbol{\Phi}_{k}}\boldsymbol{\xi}_{S,\epsilon}\),
    by the \(L^{2}(\partial B_{1})\) norm of its trace.
    Let \(\{\hat{e}_{p}\}_{p \in \mathbb{Z}}\) be the coefficients
    of \(e := u - \mathcal{T}_{\boldsymbol{\Phi}_{k}}\boldsymbol{\xi}_{S,\epsilon}\),
    in the Hilbert basis \(\{b_{p}\}_{p \in \mathbb{Z}}\).
    From the asymptotics~\eqref{eq:asymptotics_btildep_norms}, we have
    \begin{equation}
        \|e\|_{\mathcal{B}}^{2}
        = \sum_{p\in\mathbb{Z}} |\hat{e}_{p}|^{2}, \qquad
        \|e\|_{L^{2}(B_{1})}^{2}
        = \sum_{p\in\mathbb{Z}} c^{(1)}_{p}
        \frac{|\hat{e}_{p}|^{2}}{1 + p^{2}}, \qquad
        \|e\|_{L^{2}(\partial B_{1})}^{2}
        = \sum_{p\in\mathbb{Z}} c^{(2)}_{p}
        \frac{|\hat{e}_{p}|^{2}}{\sqrt{1 + p^{2}}},
    \end{equation}
    where 
    \(\{c^{(1)}_{p}\}_{p\in\mathbb{Z}}\)
    and \(\{c^{(2)}_{p}\}_{p\in\mathbb{Z}}\) are two sequences of positive
    constants both bounded below and above, and independent
    of \(u - \mathcal{T}_{\boldsymbol{\Phi}_{k}}\boldsymbol{\xi}_{S,\epsilon}\).
    The sequence $\{c^{(2)}_{p}\}_{p\in\mathbb{Z}}$ is bounded below because
    $\kappa^2$ is not a Dirichlet eigenvalue.
    We derive~\eqref{eq:approx-error-estimate2}
    from this remark and~\eqref{eq:approx-error-estimate1}.
\end{proof}

\begin{proof}[Proof of Corollary~\ref{cor:approx-error-estimate}]
    Let \(\eta > 0\) and \(u \in \mathcal{B} \cap C^{0}(\overline{B_{1}})\).
    The stability assumption implies that
    there exists \(\boldsymbol{\Phi}_{k}\) and
    \(\boldsymbol{\mu} \in \mathbb{C}^{|\boldsymbol{\Phi}_{k}|}\)
    such that
    \begin{equation}
        \|u - \mathcal{T}_{\boldsymbol{\Phi}_{k}}\boldsymbol{\mu}\|_{\mathcal{B}}
        \leq \eta \|u\|_{\mathcal B}
        \quad\text{and}\quad
        \|\boldsymbol{\mu}\|_{\ell^{2}} \leq
        C_{\mathrm{stb}} |\boldsymbol{\Phi}_{k}|^{s} \|u\|_{\mathcal B}.
    \end{equation}
    Let \(\epsilon \in (0, 1]\).
    Proposition~\ref{prop:approx-error-estimate}
    implies the existence of \(S \in \mathbb{N}\) such that
    for this particular \(\boldsymbol{\mu}\),
    \begin{equation}
        \|u - \mathcal{T}_{\boldsymbol{\Phi}_{k}}\boldsymbol{\xi}_{S,\epsilon}\|_{L^{2}(B_{1})}
        \leq C_{\mathrm{err}} \;
        \Big(
            \|u - \mathcal{T}_{\boldsymbol{\Phi}_{k}}\boldsymbol{\mu}\|_{\mathcal{B}}
            + \frac{\epsilon \sigma_{\max}}{\sqrt{S}}
            \|\boldsymbol{\mu}\|_{\ell^{2}}
        \Big)
        \leq C_{\mathrm{err}} \;
        \Big(
            \eta  +
            \frac{\epsilon \sigma_{\max}}{\sqrt{S}} C_{\mathrm{stb}} |\boldsymbol{\Phi}_{k}|^{s}
        \Big) \|u\|_{\mathcal B}.
    \end{equation}
    It remains to choose the free parameters \(\eta > 0\) and
    \(\epsilon \in (0,1]\) small enough to get the right-hand-side below
    \(\delta\), namely \(\eta \leq \frac{\delta}{2 C_{\mathrm{err}}}\)
    and \(\epsilon \leq \epsilon_{0}\) with \(\epsilon_{0}\) given
    in~\eqref{eq:epsilon-estimate}.
\end{proof}

\phantomsection
\addcontentsline{toc}{part}{References}
\printbibliography

\end{document}